\documentclass{sinst}

\usepackage{amsbsy,amsfonts,amsmath,amssymb,amsthm}
\usepackage{array}
\usepackage{bm}
\usepackage{color}
\usepackage{enumerate}
\usepackage{mathrsfs}
\usepackage{latexsym}
\usepackage[colorlinks=true]{hyperref}
\usepackage{mathtools}
\mathtoolsset{showonlyrefs, showmanualtags}
\hypersetup{urlcolor=blue, citecolor=blue, linkcolor=blue}

\definecolor{wineRed}{rgb}{0.7,0,0.3}
\definecolor{grandBleu}{rgb}{0,0,0.8}
\definecolor{darkGreen}{rgb}{0,0.4,0}
\definecolor{blueViolet}{rgb}{0.4,0,1.0}
\definecolor{bloodOrange}{rgb}{0.85,0.05,0}
\definecolor{mycolor}{rgb}{0.8,0,0.2}

\usepackage[textsize=small]{todonotes}
\usepackage{cite}

\usepackage{comment}

\DeclareMathAlphabet{\mathpzc}{OT1}{pzc}{m}{it}

\usepackage[textsize=small]{todonotes}
\numberwithin{equation}{section}

\theoremstyle{plain}
\newtheorem{mainThm}{Main Theorem} 
\newtheorem{lemma}{Lemma}[section]

\newtheorem{theorem}{Theorem}

\theoremstyle{definition}
\newtheorem{definition}{Definition}
\newtheorem{rem}{Remark}
\newtheorem{notn}{Notation}
\newtheorem{ex}{Example}

\def\N{\mathbb{N}}
\def\R{\mathbb{R}}

\def\ds{\displaystyle}
\def\ts{\textstyle}

\DeclareMathOperator{\argmin}{arg min}
\DeclareMathOperator{\diver}{div}

\begin{document}
\label{page:t}
\thispagestyle{plain}

\vspace*{-4ex}
\title{
        Well-Posedness of Pseudo-Parabolic Gradient Systems \\[-1ex] with State-Dependent Dynamics
\vspace{-2ex}
}
\author{Harbir Antil
}
\affiliation{~~\\[-4ex] Department of Mathematical Sciences and \\[-0.5ex] the Center for Mathematics and Artificial Intelligence (CMAI), 
\\
George Mason University, Fairfax, VA 22030, USA\\[-0.5ex]}
\email{hantil@gmu.edu}

\sauthor{~~\\[-4ex]{\sc Daiki Mizuno}
}
\saffiliation{~~\\[-5ex] Division of Mathematics and Informatics, \\[-0.5ex] Graduate School of Science and Engineering, Chiba University, \\[-0.5ex] 1-33, Yayoi-cho, Inage-ku, 263-8522, Chiba, Japan\\[-0.5ex]}
\semail{d-mizuno@chiba-u.jp}
\tauthor{~~\\[-4ex]{\sc Ken Shirakawa}
}
\taffiliation{~~\\[-5ex] Department of Mathematics, Faculty of Education, Chiba University \\[-0.5ex] 1--33 Yayoi-cho, Inage-ku, 263--8522, Chiba, Japan\\[-0.5ex]}
\temail{sirakawa@faculty.chiba-u.jp}

\fauthor{~~\\[-3.5ex]{\sc Naotaka Ukai}
}
\faffiliation{~~\\[-5ex] Division of Mathematics and Informatics, \\[-0.5ex] Graduate School of Science and Engineering, Chiba University, \\[-0.5ex] 1-33, Yayoi-cho, Inage-ku, 263-8522, Chiba, Japan\\[-0.5ex]}
\femail{24wd0101@student.gs.chiba-u.jp}
\vspace{-1.75ex}

\footcomment{{
$^*$\,This work is supported by JST SPRING Grant Number JPMJSP2109. This work is partially supported by the Office of Naval Research (ONR) under Award NO: N00014-24-1-2147, NSF grant DMS-2408877, the Air Force Office of Scientific Research (AFOSR) under Award NO: FA9550-25-1-0231.
}
}
\maketitle
\vspace{-0ex}

\noindent
{\bf Abstract. } 
This paper develops a general mathematical framework for pseudo-parabolic gradient systems with state-dependent dynamics. The state dependence is induced by variable coefficient fields in the governing energy functional. Such coefficients arise naturally in scientific and technological models, including state-dependent mobilities in KWC-type grain boundary motion and variable orientation-adaptation operators in anisotropic image denoising. 
We establish two main results: the existence of energy-dissipating solutions, and the uniqueness and continuous dependence on initial data. The proposed framework yields a general well-posedness theory for a broad class of nonlinear evolutionary systems driven by state-dependent operators. 
As illustrative applications, we present an anisotropic image-denoising model and a new pseudo-parabolic KWC-type model for anisotropic grain boundary motion, and prove that both fit naturally within the abstract structure of (S)$_\nu$.

\pagebreak 

\section*{Introduction}
Let $T>0$ be a fixed final time and let $M,N\in\mathbb{N}$. Let $\Omega\subset\mathbb{R}^N$ be a bounded domain with boundary $\Gamma:=\partial\Omega$ and unit outer normal $\bm n_\Gamma$, which is assumed to be Lipschitz when $N>1$. We set $Q:=(0,T)\times\Omega$ and $\Sigma:=(0,T)\times\Gamma$. 
\bigskip

In this paper, for a fixed parameter $\nu\ge 0$, we consider the following system of pseudo-parabolic PDEs, denoted by $(\mathrm{S})_\nu$:
\begin{align}
    \mbox{(S)$_\nu$:} ~&
    \nonumber
    \\
    & 
    \begin{cases}
        A(\bm{u})\partial_t\bm{u}-\mathrm{div}\bigl(\alpha(\bm{u}) B^*(\bm{u}) \partial \gamma({B}(\bm{u})\nabla\bm{u})+\nu \nabla\Upsilon_p(\nabla \bm{u}) +\mu \nabla \partial_t \bm{u}\bigr)
        \\
        \quad +[\nabla\alpha](\bm{u})\gamma({B}(\bm{u})\nabla\bm{u})+\alpha(\bm{u})\partial \gamma({B}(\bm{u})\nabla\bm{u}):[\nabla{B}](\bm{u})\nabla\bm{u}
        \\
        \quad +\nabla_{\bm{u}} G(x, \bm{u}) \ni \bm{0} \mbox{ in $ Q $,}
        \\[1ex]
        \bigl(\alpha(\bm{u})B^*(\bm{u}) \partial \gamma({B}(\bm{u})\nabla\bm{u})+\nu \nabla\Upsilon_p(\nabla \bm{u}) +\mu \nabla \partial_t \bm{u}\bigr) \bm{n}_\Gamma\ni \bm{0} \mbox{ on } \Sigma,
        \\[1ex]
        \bm{u}(0,x)=\bm{u}_0(x), ~x\in\Omega.
    \end{cases}
\end{align}
This system is formally derived as the gradient flow associated with the following energy functional:
\begin{gather}
    E_\nu: \bm{u} = {}^\top \bigl[ u_1, \dots, u_M \bigr] \in [W^{1, p}(\Omega)]^M \mapsto E_\nu(\bm{u}) 
  \nonumber
\\
  := \int_\Omega \alpha(\bm{u})\gamma(B(\bm{u})\nabla \bm{u}) \, dx+{\nu} \int_\Omega \Upsilon_p(\nabla \bm{u}) +\int_\Omega G(x,\bm{u}) \, dx \in [0, \infty).
  \label{energy01}
\end{gather}
In this context, the coefficient $ \mu > 0 $ is a fixed constant specifying the \emph{pseudo-parabolicity} of the gradient system. The exponent $ p \in (2, \infty) \cap [N, \infty) $ is a fixed constant, and the potential $ 0 \leq \Upsilon_p \in C^1(\R^M) $ is a fixed function having $p$-growth order. The coefficient fields $ 0 < \alpha \in W^{1, \infty}(\R^M) $, $ 0 < A \in [W^{1, \infty}(\R^M)]^{M \times M} $, and nonlinear operator $ B $ are fixed functions of variable $ \bm{u} $ ($ \in \R^M $). $ B^*$ is a conjugate operator of $B$. $ 0 \leq \gamma \in C^{0, 1}(\R^M)$ is a fixed Lipschitz function. $ G(x, \cdot) \in C^{1, 1}(\R^M) $, for $x \in \Omega$, is an $x$-dependent potential function, and $ \nabla_{\bm{u}} G = \nabla_{\bm{u}} G(x, \cdot) $ denotes its gradient. $ \bm{u}_0 = {}^\top\bigl[ u_{0, 1}, \dots, u_{0, M} \bigr] \in [H^1(\Omega)]^M $ is the fixed initial datum. 
\bigskip

The objective of this study is to develop a mathematical framework that accommodates complex systems with \emph{state-dependent dynamics} evolving in time. In the system (S)$_\nu$, this state-dependence is determined by the coefficient fields $ \alpha(\bm{u}) $, $ A(\bm{u}) $, and $ B(\bm{u}) $, which are motivated by scientific and technological applications.
\bigskip

For instance, the coefficient fields $\alpha(\bm{u})$ and $A(\bm{u})$ are motivated by the \emph{variable-dependent mobility} appearing in KWC-type models of grain boundary motion proposed by Kobayashi et al.\ \cite{MR1752970,MR1794359}, together with subsequent developments \cite{MR2469586,MR2548486,MR3038131,MR3082861,MR3203495,MR2836555,MR2668289,MR3888633,MR3155454,MR4395725}. Meanwhile, the operator $B(\bm{u})$ models a variable-dependent operator of orientation adaptation used in anisotropic image denoising \cite{berkels2006cartoon,AMSU202411}, with anisotropic metric $\gamma$.
\bigskip

In view of these backgrounds, the analysis of system $(\mathrm{S})_\nu$ provides a general mathematical framework with a wide range of applications to nonlinear phenomena, including material science and image engineering.
Owing to the presence of state-dependent operators, the proposed framework acquires a high degree of dynamical flexibility. This enables us to treat systems with markedly adaptive, heterogeneous, or otherwise highly variable behavior that lies beyond the scope of classical pseudo-parabolic or gradient-flow formulations.
\bigskip

In addition, this framework yields a new anisotropic model of grain boundary motion with guaranteed well-posedness. Moreover, the vectorial structure of the variable $\bm{u}$ offers the potential for further applications to other nonlinear systems, such as the Fr\'{e}mond model of shape memory alloys \cite{MR1040232,MR1885252}.
\bigskip

The mathematical framework is formulated in terms of the well-posedness results for (S)$_\nu$, which are stated in the following Main Theorems.
\begin{description}
    \item[Main Theorem \ref{mainThm1}:] Existence of solutions with energy dissipation
    \item[Main Theorem \ref{mainThm2}:] Uniqueness and continuous dependence of solutions
\end{description}
\bigskip

The Main Theorems are stated in Section \ref{secmain}, after the preliminaries in Section \ref{secpre}. The proofs of Main Theorems \ref{mainThm1} and \ref{mainThm2} are given in Sections \ref{sec:proof1} and \ref{sec:proof2}, respectively, and are based on the time-discretization method developed in Section \ref{timedis003}. Finally, Section 6 is devoted to demonstrate two applications of Main Theorems, concerned with:
\begin{description}
    \item[\S\,6.1:]The anisotropic image denoising studied in \cite{AMSU202411}.
    \item[\S\,6.2:]A new pseudo-parabolic KWC-type model of grain boundary motion with anisotropy.
\end{description}

\section{Preliminaries}\label{secpre}
We begin by specifying the notation adopted throughout the paper.
\begin{notn}[{Real analysis}]\label{notn1}We define:
  \[a \lor b : = \max\{ a , b \} \mbox{ and } a \land b : = \min\{ a , b \} , \mbox{ for all } a , b \in [-\infty, \infty],\]
  and especially, we note:
  \[[a]^+:= a \lor 0 \mbox{ and } [a]^-:= -(a \land 0), \mbox{ for all } a \in [-\infty,\infty].\]

Let $d \in \N$ be fixed dimension. We denote by $ | x | $ and $ x \cdot y $ the Euclidean norm of $ x \in \R^d$ and the standard scalar product of $ x , y \in \R^d$, respectively, i.e.,
\begin{gather*}
  | x | : = \sqrt{x_1^2 + \cdots + x_d^2} \quad \mbox{and} \quad x \cdot y := x_1 y_1 + \cdots + x_d y_d,
  \\
  \mbox{ for all } x = [x_1 , \dots , x_d], \, y = [y_1 , \dots , y_d] \in \R^d.
\end{gather*}
\end{notn}
Additionally, we note the following elementary fact:
\begin{description}
  \item[\textbf{(Fact 1)}](cf.\cite[Proposition 1.80]{MR1857292}) Let $ m \in \N $ be fixed. If $ \{A_k\}_{k=1}^m \subset \R $ and $ \{ a_n^k \}_{n\in\N}\subset\R $, 
  $ k = 1, \dots , m $ satisfies that:
  \[ \liminf_{ n \rightarrow \infty } a^k_n \geq A_k, \mbox{ for } k = 1, \dots , m, \mbox{ and }\limsup_{ n \rightarrow \infty} \sum_{k=1}^{m} a_n^k \leq \sum_{k=1}^{m} A_k. \]
  Then, $ \lim_{n \rightarrow \infty } a_n^k = A_k $, for $ k = 1 , \dots , m $. 
\end{description}
\begin{notn}[{Abstract functional analysis}]\label{notn2}
For an abstract Banach space $ X $, we denote by $| \cdot |_X$ the norm of $ X $, and denote by $ \langle \cdot , \cdot \rangle_X $ the duality pairing between $ X $ and its dual $ X^* $. In particular, when $ X $ is a Hilbert space, we denote by $( \cdot , \cdot )_X$ its inner product. 

For Banach spaces $ X_1 , \dots ,X_d $ with $ 1 < d \in \N$, let $ X_1 \times \dots \times X_d $ be the product Banach space which has the norm 
\[ | \cdot |_{ X_1\times \dots \times X_d } : = \left(| \cdot |_{X_1}^2 + \dots + | \cdot |_{X_d}^2\right)^\frac{1}{2} .\] 
In addition, in the case where the domain is the whole space we shall denote the norm in the following:
\begin{align*}
  \|\cdot\|_{W^{p.q}}:=|\cdot|_{W^{p.q}(\R^d;\R^k)}, \mbox{ for all $ d,k\in\N$, $ p \in \N\cup \{0\}, $ and $q \in[1,\infty]$.}
\end{align*}
Furthermore, for normed linear spaces $E$ and $F$, we denote by
\begin{align*}
  \mathcal{L}(E,F)
\end{align*}
the space of all continuous linear maps from $E$ into $F$. 
\end{notn}

\begin{notn}[{Convex analysis}]\label{notn3}
For any proper lower semi-continuous (l.s.c.) and convex function $ \Psi : X \rightarrow (-\infty,\infty]$ on a Hilbert space $ X $, we denote by $ D( \Psi ) $ the effective domain of $ \Psi $, and denote by $ \partial \Psi $ the subdifferential of $ \Psi $. The subdifferential $ \partial \Psi $ is a set-valued map corresponding to a weak differential of $ \Psi $, and it is known as a maximal monotone graph in the product space $ X \times X $. More precisely, for each $ w \in X $, the value $ \partial \Psi( w ) $ is defined as the set of all elements $ w^* \in X $ that satisfy the variational inequality
\[( w^*, x - w )_X \leq \Psi (x) - \Psi (w), \mbox{ for any } x \in D ( \Psi ), \]
and the set $ D ( \partial \Psi ) := \{ x \in X \,|\, \partial \Psi (x) \neq \emptyset \}$ is called the domain of $ \partial \Psi $. We often use the notation $`` [w, w^*] \in \partial \Psi $ in $ X \times X "$ to mean that $`` w^* \in \partial \Psi ( w ) $ in $ X $ for $ w \in D ( \partial \Psi ) "$, by identifying the operator $ \partial \Psi $ with its graph in $ X \times X $.
\end{notn}
\begin{ex}\label{ex1}
  Let $X$ be a Hilbert space. Let $\gamma: \R^{M\times N} \longrightarrow [0,+\infty)$ be a convex function such that it belongs to $C^{0,1}(\R^{M\times N})$. Also, the function $\{\gamma_\varepsilon\}_{\varepsilon\geq0}$ is defined as follows: 
  \begin{equation*}
    \gamma_\varepsilon :=\left\{
      \begin{aligned}
        &\gamma, &&\mbox{ if }\varepsilon=0,
        \\
        &\rho_\varepsilon * \gamma, &&\mbox{ otherwise},
      \end{aligned}\right. \mbox{ on }\R^{M\times N},
  \end{equation*}
  where $\rho_\varepsilon$ is the standard mollifier. Then, the following two items hold.
  \begin{description}
    \item[(I)] Let $\{\Phi_\varepsilon\}_{\varepsilon\geq0}$ be a sequence of functionals on $[X]^{M\times N}$, defined as:
  \begin{align*}
    &\Phi_\varepsilon:W\in[X]^{M\times N}\mapsto\Phi_\varepsilon(W):=\int_\Omega\gamma_\varepsilon(W)\,dx\in[0,\infty).
  \end{align*}
  Then, for every $\varepsilon\in[0,\infty)$, $\Phi_\varepsilon$ is the proper, l.s.c., and convex function, such that 
  \[D(\Phi_\varepsilon)=D(\partial\Phi_\varepsilon)=[X]^{M\times N},\]
  and 
  \begin{equation*}
    \partial\Phi_\varepsilon(W):=
    \left\{
      \begin{aligned}
        &\{\nabla\gamma_\varepsilon(W)\}, \mbox{ if }\varepsilon>0,
        \\
        &\{W^*\in[X]^{M\times N}~|~W^*\in\partial\gamma(W)\mbox{ a.e. in } \Omega,\}, \mbox{ if }\varepsilon=0,
      \end{aligned}
    \right.
  \end{equation*}
  \[\mbox{ in }[X]^{M\times N}, \mbox{ for any }W\in[X]^{M\times N}.\]
    \item[(II)] Let any open interval $I\subset(0,T)$, and let $\{\widehat{\Phi}_\varepsilon^I\}_{\varepsilon\geq0}$ be a sequence of functionals on $L^2(I;[X]^{M\times N})\,(=[L^2(I;H)]^{M\times N})$, defined as:
  \begin{align*}
    &\widehat{\Phi}_\varepsilon^I:W\in L^2(I;[X]^{M\times N})\mapsto\widehat{\Phi}_\varepsilon^I(W):=\int_I\Phi_\varepsilon(W(t))\,dt\in[0,\infty).
  \end{align*}
  Then, for every $\varepsilon\in[0,\infty)$, $\widehat{\Phi}^I_\varepsilon$ is the proper, l.s.c., and convex function, such that 
  \[D(\widehat{\Phi}_\varepsilon^I)=D(\partial\widehat{\Phi}_\varepsilon^I)=L^2(I;[X]^{M\times N}),\]
  and 
  \begin{align*}
    \partial\widehat{\Phi}_\varepsilon^I(W)
    &=\{\tilde{W}^*\in L^2(I;[X]^{M\times N})~|~\tilde{W}^*(t)\in \partial \Phi_\varepsilon(W(t))\mbox{ in }[X]^{M\times N}, \mbox{ a.e. }t\in I\}
    \\
    &=
    \left\{
      \begin{aligned}
        &\{\nabla\gamma_\varepsilon(W)\}, \mbox{ if }\varepsilon>0,
        \\
        &\{W^*\in L^2(I;[X]^{M\times N})~|~W^*\in\partial\gamma(W)\mbox{ a.e. in } I\times \Omega,\}, \mbox{ if }\varepsilon=0,
      \end{aligned}
    \right.
  \end{align*}
  \[\mbox{ in }L^2(I;[X]^{M\times N}), \mbox{ for any }W\in L^2(I;[X]^{M\times N}).\]
  \end{description}
\end{ex}

\begin{definition}[{Mosco-convergence}: cf.\cite{MR0298508}]\label{dfnmosco}
    Let $ X $ be a Hilbert space. Let $ \Psi : X \rightarrow ( -\infty , \infty ] $ be a proper, l.s.c., and convex function, and let $ \{ \Psi_n \}_{ n \in \N } $ be a sequence of proper, l.s.c., and convex functions $ \Psi_n : X \rightarrow ( -\infty , \infty ] $, $ n \in \N $. Then, we say that $ \Psi_n \to  \Psi $ on $ X $ in the sense of Mosco, iff. the following two conditions are fulfilled:
  \begin{description}
    \item[(M1) (Optimality)] For any $w_0 \in D ( \Psi )$, there exists a sequence $ \{w_n\}_{ n \in \N } \subset X $ such that $ w_n \rightarrow w_0 $ in $ X $ and $ \Psi_n ( w_n ) \rightarrow \Psi ( w_0 ) $ as $ n \rightarrow \infty $, 
    \item[(M2) (Lower-bound)] $\liminf_{ n \rightarrow \infty } \Psi_n ( w_n ) \geq \Psi ( w_0 )$ if $w_0 \in X, \{ w_n \}_{ n \in \N } \subset X $, and $w_n \rightarrow w_0 $ weakly in $ X $ as $ n \rightarrow \infty $.
  \end{description}
\end{definition}
\begin{definition}[{$ \Gamma $-convergence}; cf.\cite{MR1201152}]\label{dfngamma}
  Let $ X $ be a Hilbert space. Let $ \Psi : X \rightarrow ( -\infty , \infty ] $ be a proper and l.s.c. function, and let $ \{ \Psi_n \}_{ n \in \N } $ be a sequence of proper and l.s.c. functions $ \Psi_n : X \rightarrow ( -\infty , \infty ] $, $ n \in \N $. Then, we say that {$ \Psi_n \to \Psi $} on $ X $ in the sense of $ \Gamma $-convergence, iff. the following two conditions are fulfilled:
  \begin{description}
      \item[{($\mathbf{\Gamma}$1) (Optimality)}] For any $w_0 \in D ( \Psi )$, there exists a sequence $ \{w_n\}_{ n \in \N } \subset X $ such that $ w_n \rightarrow w_0 $ in $ X $ and $ \Psi_n ( w_n ) \rightarrow \Psi ( w_0 ) $ as $ n \rightarrow \infty $, 
      \item[{($\mathbf{\Gamma}$2) (Lower-bound)}] $\liminf_{ n \rightarrow \infty } \Psi_n ( w_n ) \geq \Psi ( w_0 )$ if $w_0 \in X, \{ w_n \}_{ n \in \N } \subset X $, and $w_n \rightarrow w_0 $ in $ X $ as $ n \rightarrow \infty $.
  \end{description}
\end{definition}

\begin{rem}\label{rem2}
  We note that under the condition of convexity of functionals, Mosco convergence implies $ \Gamma $-convergence, i.e., the  $ \Gamma $-convergence of convex functions can be regarded as a weak version of Mosco convergence. Furthermore,  in the $ \Gamma $-convergence of convex functions, we can see the following:
\begin{description}
  \item[(Fact 2)] (cf.\cite[Theorem 3.66]{MR0773850} and \cite[Chapter 2]{Kenmochi81}) Let $ X $ be a Hilbert space. Let $ \Psi : X \rightarrow ( -\infty , \infty ] $ and $ \Psi_n : X \rightarrow ( -\infty , \infty ] $, $ n \in \N $, be proper, l.s.c., and convex functions on a Hilbert space $ X $ such that $ \Psi_n \rightarrow \Psi $ on $ X $, in the sense of $ \Gamma $-convergence, as $ n \rightarrow \infty $. Let us assume that
  \begin{equation*}
    \left\{
    \begin{aligned}
      &[z, z^*] \in X \times X ,~ [z_n, z_n^*] \in \partial \Psi_n \mbox{ in } X \times X,~n \in \N,
      \\
      &z_n \rightarrow z^* \mbox{ in } X \mbox{ and } z_n^* \rightarrow z^* \mbox{ weakly in } X \mbox{ as } n \rightarrow \infty.  
    \end{aligned}
    \right.
  \end{equation*}
  Then, it holds that:
  \[ [ z , z^* ] \in \partial \Psi \mbox{ in } X \times X , \mbox{ and } \Psi_n ( z_n ) \rightarrow \Psi ( z ) \mbox{ as } n \rightarrow \infty. \]
  \item[(Fact 3)](cf.\cite[Lemma 4.1]{MR3661429} and \cite[Appendix]{MR2096945}) Let $ X $ be a Hilbert space, $ d \in \N $ be dimension constant, and $ A \subset \R^d $ be a bounded open set. Let $ \Psi : X \rightarrow ( -\infty , \infty ] $ and $ \Psi_n : X \rightarrow ( -\infty , \infty ] $, $ n \in \N $, be proper, l.s.c., and convex functions on $ X $ such that $ \Psi_n \rightarrow \Psi $ on $ X $, in the sense of $ \Gamma $-convergence, as $ n \rightarrow \infty $. Then, a sequence $ \{ \widehat{ \Psi }_n^A \}_{ n \in \N } $ of proper, l.s.c., and convex functions on $ L^2 ( A ; X ) $, defined as:
  \begin{equation*}
    z \in L^2 ( A ; X ) \mapsto \widehat{ \Psi }^A_n ( z ) : = \left\{
      \begin{aligned}
        & \int_A \Psi _n ( z ( t ) ) \,dt, \mbox{ if } \Psi_n ( z ) \in L^1 ( A ), 
        \\
        & \infty, \mbox{ otherwise}, 
      \end{aligned}
    \right.\mbox{ for }n \in \N;
  \end{equation*}
  converges to a proper, l.s.c., and convex function $ \widehat{ \Psi }^A $ on $ L^2 ( A ; X ) $, defined as:
  \begin{equation*}
    z \in L^2 ( A ; X ) \mapsto \widehat{ \Psi }^A ( z ) : = \left\{
      \begin{aligned}
        & \int_A \Psi ( z ( t ) ) \,dt, \mbox{ if } \Psi ( z ) \in L^1 ( A ), 
        \\
        & \infty, \mbox{ otherwise}, 
      \end{aligned}
    \right.
  \end{equation*}
  on $ L^2 ( A ; X ) $, in the sense of $ \Gamma $-convergence, as $ n \rightarrow \infty $.
\end{description}
\end{rem}
\begin{ex}[Examples of Mosco-convergence]\label{ex2}
  Let $X$ be a Hilbert space. Let $\varepsilon_0\geq0$ be arbitrary fixed constant, and let $\gamma$ and $\{\gamma_\varepsilon\}_{\varepsilon\geq0}$ be as in Example \ref{ex1}, respectively. Then, the following three items hold.
  \begin{description}
    \item[(I)] $\gamma_\varepsilon\rightarrow\gamma_{\varepsilon_0}$ on $\R^{M\times N}$, in the sense of Mosco, as $\varepsilon\rightarrow\varepsilon_0$.
    \item[(II)] Let $\{\Phi_\varepsilon\}_{\varepsilon\geq0}$ be the sequence of proper, l.s.c., and convex functions on $[X]^{M\times N}$, as in Example \ref{ex1}(I). Then,
    \begin{align}
    \Phi_\varepsilon\rightarrow\Phi_{\varepsilon_0} \mbox{ on }[X]^{M\times N} \mbox{ in the sense of Mosco, as }\varepsilon\rightarrow\varepsilon_0.
    \end{align} 
    \item[(III)] Let $I\subset(0,T)$ be an open interval, and let $\{\widehat{\Phi}_{\varepsilon}^I\}_{\varepsilon\geq0}$ be the sequence of proper, l.s.c., and convex functions on $L^2(I;[X]^{M\times N})$, as in Example \ref{ex1}(II). Then, 
    \begin{align}
    \widehat{\Phi}^I_\varepsilon\rightarrow\widehat{\Phi}^I_{\varepsilon_0} \mbox{ on }L^2(I;[X]^{M\times N}), \mbox{ in the sense of Mosco, as }\varepsilon\rightarrow\varepsilon_0.
  \end{align}
  \end{description}
\end{ex}
\begin{notn}\label{deftseq}
  Let $ \tau >0 $ be a constant of time-step size, and $\{t_i\}_{i=0}^\infty$ be a time-sequence defined as 
  \[t_i := i\tau, \mbox{ for any } i = 0, 1, 2, \dots . \]
  Let $ X $ be a Banach space. For any sequence $ \{ [t_i , u_i] \}_{i=0}^\infty \subset [0,\infty) \times X $, we define three types of interpolations: $ [ \overline{u} ]_\tau \in L^\infty_{\mathrm{loc}}([0,\infty);X) $, $ [ \underline{u} ]_\tau \in L^\infty_{\mathrm{loc}}([0,\infty);X) $, and $ [ u ]_\tau \in W^{1,2}_{\mathrm{loc}}([0,\infty);X) $, as follows:
\begin{equation*}
  \left\{
  \begin{aligned}
    &[\overline{u}]_\tau(t):=\chi_{(-\infty,0]}u_{0}+\sum_{i=1}^{\infty}\chi_{(t_{i-1},t_i]}(t)u_{i},
    \\
    &[\underline{u}]_\tau(t):=\sum_{i=0}^{\infty}\chi_{(t_i,t_{i+1}]}(t)u_{i},\hspace*{25ex}\mbox{in }X,\mbox{ for any } t\geq0,
    \\
    &[u]_\tau(t):=\sum_{i=1}^{\infty}\chi_{(t_{i-1},t_i]}(t)\biggl(\frac{t-t_{i-1}}{\tau}u_{i}+
    \frac{t_i-t}{\tau}u_{i-1}\biggr),
  \end{aligned}
  \right.
\end{equation*}
where $ \chi_E : \R \rightarrow \{0,1\} $ represents the characteristic function of the set $ E \subset \R $.

In the meantime, for any $q\in[1,\infty)$ and any $w \in L^q_{\mathrm{loc}}([0,\infty);X)$, we denoted by $\{w_i\}_{i=0}^\infty\subset X$ the sequence of time-discretization data of $w$, defined as 
\begin{align*}
  w_0=0 \mbox{ in $X$ and } w_i :=\frac{1}{\tau}\int_{t_{i-1}}^{t_i}w(\sigma)\,d\sigma \mbox{ in $X$ for }i=1,2,3,\dots.
\end{align*}
As is easily checked, the time-interpolations $[\overline{w}]_\tau$, $[\underline{w}]_\tau\in L^q_{\mathrm{loc}}([0,\infty);X)$, for the above $\{w_i\}_{i=0}^\infty$ fulfill 
\begin{align*}
  [\overline{w}]_\tau \rightarrow w \mbox{ and }[\underline{w}]_\tau\rightarrow w\mbox{ in }L^q_{\mathrm{loc}}([0,\infty);X) \mbox{ as }\tau \downarrow 0.
\end{align*}
\end{notn}

\begin{notn}\label{tensordef}
  Let $A = [a_{ij}]$ and $B = [b_{ij}]$ be arbitrary matrices in $\mathbb{R}^{M \times N}$, 
with entries $a_{ij}, b_{ij} \in \mathbb{R}$ for $i = 1,\ldots,M$ and $j = 1,\ldots,N$. 
We define the Frobenius inner product and norm by
\[
A : B = \sum_{i=1}^M \sum_{j=1}^N a_{ij} b_{ij}, 
\qquad 
\|A\| = \sqrt{A : A}, \mbox{ respectively.}
\]

For $M > 1$, and a vector-valued function $\bm{z} = [z_i] \in [L^1_{\mathrm{loc}}(\Omega)]^M$. 
Its distributional gradient is given by
\[
\nabla \bm{z} = {}^\top(\nabla z_1, \ldots, \nabla z_M) 
= \begin{bmatrix}
\partial_1 z_1 & \cdots & \partial_N z_1 \\
\vdots & \ddots & \vdots \\
\partial_1 z_m & \cdots & \partial_N z_M
\end{bmatrix}
\in \mathcal{D}'(\Omega)^{M \times N}.
\]
We also denoted by $\Delta \bm{z} := [\Delta z_i] \in [\mathscr{D}'(\Omega)]^M$ as the usual Laplace operator, in the distributional sense. 

\noindent
In particular, let $\Delta_N$ denote the Laplace operator with homogeneous Neumann boundary condition, defined as follows:
\[
D(\Delta_N) := 
\left\{ \tilde{\bm{z}} \in [H^2(\Omega)]^M \;\middle|\; 
\nabla \tilde{\bm{z}}|_\Gamma \cdot n_\Gamma = 0 
\text{ in } [H^{1/2}(\Gamma)]^m \right\},
\]
and for $\bm{z} \in D(\Delta_N)$ we set
\[
\Delta_N \bm{z} := [\Delta_N z_i] = \Delta \bm{z} \in [H]^M.
\]

The operator $-\Delta_N$ is identified as a linear isomorphism via the Green-type formula (cf. \cite[Propositon 5.6.2]{MR2192832}):
\[
- \int_\Omega \Delta_N \bm{z} \cdot \bm{w} \, dx 
= \int_\Omega \nabla \bm{z} : \nabla \bm{w} \, dx, 
\qquad \forall (\bm{z},\bm{w}) \in D(\Delta_N) \times [H]^M.
\]

Finally, for a matrix-valued function $\bm{Z} = [z_{ik}] \in [H]^{M \times N}$, 
we define its distributional divergence by
\[
\diver \bm{Z} := 
\left[ \sum_{k=1}^N \partial_k z_{ik} \right]
\in \mathcal{D}'(\Omega)^M.
\]
\end{notn}

\section{Main results}\label{secmain}
\begin{notn}\label{notnsp}
  Throughout this paper, we fix a finite time $T>0$ and a spatial dimension $N \in \mathbb{N}$. 
Let $\Omega \subset \mathbb{R}^N$ be a bounded domain, $ \Gamma := \partial \Omega $ with the unit outer normal $ \bm{n}_\Gamma $, and $ \Gamma $ has the Lipschitz regularity when $ N > 1 $. We further define the time-space cylinders $Q := (0,T)\times \Omega$ and $\Sigma := (0,T)\times \Gamma$, and introduce the basic functional spaces as follows.
\begin{align*}
  H:=L^2(\Omega),~V:=H^1(\Omega),~\mathscr{H}:=L^2(0,T;H),\mbox{ and }\mathscr{V}:=L^2(0,T;V)
\end{align*}
\end{notn}
Based on the above notation, the main results of this paper are established under the following assumptions.
\begin{description}
  \item[(A0)]$ M \in \N $, $p \in (2,\infty) \cap [N,\infty)$, $ \mu > 0 $, and $ \nu \geq 0 $ are fixed constant. 
  \item[(A1)] $\alpha : \R^M \longrightarrow [0,\infty)$ is a fixed function such that $\alpha \in W^{2,\infty}(\R^M) \cap C^2(\R^M)$.
  \item[(A2)] $G : (x, \bm{u}) \in \Omega \times \R^M \longrightarrow G(x, \bm{u}) \in [0,\infty)$ is a fixed function, with variables $ x \in \Omega $ and $ \bm{u} = {^\top}[u_1, \dots, u_M] \in \R^M $, fulfilling the following conditions:
      \begin{description}
          \item[$\bullet$] for any $ \bm{u} \in \R^M $, $ G(\cdot, \bm{u}) \in L^1(\Omega) $,
          \item[$\bullet$] for a.e. $ x \in \Omega $, $G(x, \cdot) \in C^{1, 1}(\R^M)$, so that, for a.e. $ x \in \Omega $, its gradient $ \nabla_{\bm{u}}G(x, \cdot) = \bigl[ \partial_{u_i} G(x, \cdot) \bigr]_{1 \leq i \leq M} $ and Hessian  $ \nabla_{\bm{u}}^2 G(x, \cdot) = \bigl[ \partial_{u_i} \partial_{u_j} G(x, \cdot) \bigr]_{1 \leq i, j \leq M} $ are bounded on $ \R^M $; 
          \item[$\bullet$] the functions $ (x, \bm{u}) \in \Omega \times \R^M \mapsto |\nabla_{\bm{u}} G(x, \bm{u})| $ and $ (x, \bm{u}) \in \Omega \times \R^M \mapsto |\nabla_{\bm{u}}^2 G(x, \bm{u})| $ belong to $ L^\infty(\Omega \times \R^M) $.
      \end{description}
  \item[(A3)] $B : \R^M \longrightarrow \mathcal{L}(\R^{M \times N}; \R^{M \times N}) $ is a nonlinear operator which is defined as:
      \begin{gather}
          B: \bm{u} \in \R^M \mapsto B(\bm{u}) W := W B_0(\bm{u}) \in \R^{M \times N}
      \end{gather}
        with use of a matrix-valued function:
  \begin{align*}
      B_0 \in [W^{2,\infty}(\R^M) \cap C^2(\R^M)]^{N \times N}.
  \end{align*}
  Also, $[\nabla B] : \R^M \longrightarrow \mathcal{L}(\R^M;\mathcal{L}(\R^{M \times N}; \R^{M \times N})) $ is the differential of $B$, i.e. 
  \begin{align*}
    [\nabla B]:\bm{u}\in\R^M \mapsto [\nabla B](\bm{u})W := W [\nabla B_0](\bm{u}) \in (\R^{M \times N})^M.
\end{align*}
Furthermore, the following operator will be used throughout this paper:
\begin{align*}
  &W[\nabla B_0](\bm{v}):={}^\top\left[W[\partial_{v_1}B_0](\bm{v}),\dots,W[\partial_{v_M}B_0](\bm{v})\right]\in (\R^{M\times N})^M,
  \\
  &Z:W[\nabla B_0](\bm{v}):={}^\top\left[Z:W[\partial_{v_1}B_0](\bm{v}),\dots,Z:W[\partial_{v_M}B_0](\bm{v})\right]\in\R^M,
  \\
  &\hspace{15ex}\mbox{for any }\bm{v}\in\R^M, \mbox{ for any }W,Z\in \R^{M\times N}.
\end{align*}
  
  \noindent
        Moreover, $ B^* : \R^M \longrightarrow \mathcal{L}(\R^{M \times N}; \R^{M \times N}) $ is a conjugate operator defined as:
      \begin{gather}
          B^*: \bm{u} \in \R^M \mapsto B^*(\bm{u}) W := W \,{}^\top B_0(\bm{u}) \in \R^{M \times N}.
      \end{gather}
  \item[(A4)] $A : \R^M \longrightarrow \R^{M \times M}$ is a given function such that $A\in W^{1,\infty}(\R^M;\R^{M \times M})$, and for any $\bm{v} \in \R^M$, $A(\bm{v})$ is symmetric matrix, and there exists a constant $C_A > 0$ such that
  \begin{gather*}
    ^\top \bm{w}A(\bm{v})\bm{w}\geq C_A|\bm{w}|^2, \mbox{ for all }\bm{v},\bm{w}\in\R^M.
  \end{gather*}  
  \item[(A5)] $\Upsilon_p : \R^{M \times N} \longrightarrow [0,\infty)$ is a $C^1$-class convex function such that there exists a constant $C_\Upsilon > 0$ satisfying
  \begin{gather*}
    \frac{1}{C_\Upsilon}(|W|^p-1)\leq \Upsilon_p(W)\leq C_\Upsilon(|W|^p+1), \mbox{ for any }W\in\R^{M\times N}, \mbox{ and }
    \\
    (\nabla \Upsilon_p(W_1)-\nabla \Upsilon_p(W_2)):(W_1-W_2)\geq C_\Upsilon|W_1-W_2|^p, \mbox{ for all }W_1,W_2\in\R^{M\times N}.
  \end{gather*}
  \item[(A6)] $\gamma: \R^{M \times N} \longrightarrow [0,\infty)$ is a fixed convex function satisfying the following conditions: there exists a constant $C_\gamma > 0$ such that
\begin{gather*}
   \gamma(W) \leq C_\gamma(|W| + 1), \mbox{ and } \nabla \gamma \in L^\infty(\R^{M \times N}; \R^{M \times N}).
\end{gather*}
  \item[(A7)] $\bm{u}_0$ is a fixed initial data such that 
  \begin{align*}
    \bm{u}_0\in [W^{1,p}(\Omega)]^M, \mbox{ if }\nu\in(0,1), \mbox{ and } \bm{u}_0\in [V]^M, \mbox{ if }\nu=0.
  \end{align*}
\end{description}
\begin{rem}\label{Rem.(A2)}
  Note that the assumption (A2) implies the following condition.
    \begin{description}
      \item[(A2)$'$]
        The function $ x \in \Omega \mapsto  G(x, \bm{0}) $ belongs to $ L^1(\Omega) $, and there exists a constant $ L_G $, independent of variables $ x \in \Omega $ and $ \bm{u} \in \R^M $, such that
          \begin{gather}
            \bigl|G(x, \bm{u}) -G(x, \bm{\tilde{u}})\bigr| +\bigl|\nabla_{\bm{u}}G(x, \bm{u}) -\nabla_{\bm{u}}G(x, \bm{\tilde{u}})\bigr| \leq L_G |\bm{u} -\bm{\tilde{u}}|, \\
            \mbox{for a.e. $ x\in \Omega $, and all $ \bm{u}, \bm{\tilde{u}} \in \R^M $.}
          \end{gather}
    \end{description}
\end{rem}

Next, let us give the definition of the solution to the system (S)$_\nu$.
\medskip

\begin{definition}
  A function $\bm{u} \in [\mathscr{H}]^M$ is called a solution to the system (S)$_\nu$ if and only if $\bm{u}$ fulfills the following conditions:
  \begin{description}
      \item[(S0)] The function $\bm{u}$ belongs to $W^{1,2}(0,T;[V]^M)$, and if $\nu > 0$, then $\bm{u}$ belongs to $L^\infty(0,T; [W^{1,p}(\Omega)]^M)$.
    \item[(S1)] There exists a vector-value function $\bm{w}^*\in [\mathscr{H}]^{M\times N}$ such that:
      \begin{gather*}
        \bm{w}^*(t,x)\in\partial\gamma(B(\bm{u}(t,x))\nabla \bm{u}(t,x)) \mbox{ in } \R^{M\times N}, \mbox{ for a.e. }(t,x) \in Q,
      \end{gather*}
      and $\bm{u}$ solves the following variational inequality: 
      \begin{gather*}
        (A(\bm{u}(t))\partial_t\bm{u}(t),\bm{u}(t)-\bm{\varphi})_{[H]^M}+\mu(\nabla\partial_t\bm{u}(t),\nabla(\bm{u}(t)-\bm{\varphi}))_{[H]^{M\times N}}
    \\
    +\nu\int_\Omega \nabla \Upsilon_p(\nabla \bm{u}(t)):\nabla (\bm{u}(t)-\bm{\varphi})\,dx+(\nabla G(x,\bm{u}(t)),\bm{u}(t)-\bm{\varphi})_{[H]^M}
    \\
    +(\alpha(\bm{u}(t))\bm{w}^*(t):[\nabla{B}](\bm{u}(t))\nabla\bm{u}(t),\bm{u}(t)-\bm{\varphi})_{[H]^M}
    \\
    +([\nabla\alpha](\bm{u}(t))\gamma({B}(\bm{u}(t))\nabla\bm{u}(t)),\bm{u}(t)-\bm{\varphi})_{[H]^M}
    \\
    +\int_{\Omega}\alpha(\bm{u}(t))\gamma(B(\bm{u}(t))\nabla \bm{u}(t))\,dx\leq\int_{\Omega}\alpha(\bm{u}(t))\gamma (B(\bm{u}(t))\nabla\bm{\varphi})\,dx,
    \\
    \end{gather*}
    for a.e. $t \in (0,T)$ and for any $\bm{\varphi} \in [V]^M$, where $\bm{\varphi} \in [W^{1,p}(\Omega)]^M$ if $\nu > 0$.
      \item[(S2)] $ \bm{u}(0) = \bm{u}_0 $ in $ [H]^M $.
  \end{description}
\end{definition}

\begin{mainThm}(Existence and uniqueness of solution with energy-dissipation)\label{mainThm1}
  Under the assumptions (A0)--(A7), the system (S)$_\nu$ admits a solution $\bm{u}$. Moreover, the solution $\bm{u}$ satisfies the following energy-inequality:
  \begin{gather}
    \frac{C_A}{4}\int_{s}^{t}|\partial_t\bm{u}(\sigma)|^2_{[H]^M}\,d\sigma
      +\frac{\mu}{2}\int_{s}^{t}|\nabla\partial_t\bm{u}(\sigma)|^2_{[H]^{M\times N}}\,d\sigma
      +E_\nu(\bm{u}(t))\leq E_\nu(\bm{u}(s)),
      \\
      \mbox{for a.e. $ s \in [0, T) $ including $ s=0 $, and any } t \in [s, T].
    \label{ene-inq1}
  \end{gather}
\end{mainThm}

  \begin{mainThm}(Uniqueness and continuous dependence of solution)\label{mainThm2}
  Under the assumptions (A0)--(A7), let $ \bm{u}_{0, k} \in [W^{1, p}(\Omega)]^M $, $ k = 1, 2 $, be two initial data, and let $ \bm{u}_k \in [\mathscr{H}]^M $, $ k = 1, 2 $, be two solutions to the system (S)$_\nu$ with initial data $ \bm{u}_0 = \bm{u}_{0, k} $, $ k = 1, 2 $. Additionally, we assume that $ \nu>0 $, $N\leq 6$, $A\in C^1(\R^M;\R^{M\times M})$ and $\gamma \in C^{1,1}(\R^{M\times N})\cap C^2(\R^{M\times N})$. Let $ J(t) $ be the function of time defined as:
  \begin{gather}\label{defOfJ}
      J(t) := |\sqrt{A(\bm{u}_2(t))}(\bm{u}_1 -\bm{u}_2)(t)|^2_{[H]^M} +\mu |\nabla (\bm{u}_1 -\bm{u}_2)(t)|_{[H]^{M\times N}}^2, \mbox{ for all } t \in [0,T].
  \end{gather}
  Then, there exists a positive constant $ C_* >0 $ such that 
  \begin{gather}\label{cncl_J}
        J(t)\leq \exp(C_*(1+T)((1+|\bm{u}_1|_{L^\infty(0,T;[W^{1,p}(\Omega)]^M)})^2+|\partial_t\bm{u}_1|_{[\mathscr{V}]^M}+|\partial_t\bm{u}_2|_{[\mathscr{V}]^M}))J(0), 
        \\
        \mbox{ for all } t \in [0,T].
  \end{gather}
  Moreover, we can replace the phrase ``a.e. $ s \in [0, T) $'' to ``for any $ s \in [0, T] $'' the energy-inequality \eqref{ene-inq1}.
\end{mainThm}

\section{Time-discretization scheme}\label{timedis003}
In this section, we focus on the time-discretization scheme of our system (S)$_\nu$. Let $ \nu\in(0,1) $ be a fixed positive constant. Let $ m \in \N $ be the division number of the time-interval $ (0, T) $. Let $\tau := \frac{T}{m}$ be the time-step size, and let $\{t^i\}_{i=1}^m$ be a time sequence defined by $t^i := i\tau$, for $i = 1, 2, \dots, m$. Moreover, throughout this section, we suppose that the function $\gamma$ belongs to $C^{1,1}(\R^{M\times N})$ and $C^2(\R^{M\times N})$.

Based on these, we set up the time-discretization scheme (AP)$^\tau_\nu$ as an approximating problem of (S)$_\nu$:
\\

(AP)$^\tau_\nu$: \, For any initial value $\bm{u}^0_\nu \in [W^{1,p}(\Omega)]^M$, find a sequence of $ \{\bm{u}^i_\nu\}_{i=1}^m \subset [W^{1,p}(\Omega)]^M $ such that 
  \begin{align*}
  &\frac{A(\bm{u}^{i-1}_\nu)}{\tau}(\bm{u}^i_\nu-\bm{u}^{i-1}_\nu)-\mathrm{div}\Big(\alpha(\bm{u}^i_\nu) B^*(\bm{u}^i_\nu) \nabla\gamma({B}(\bm{u}^i_\nu)\nabla\bm{u}^i_\nu)+\nu \nabla\Upsilon_p(\nabla \bm{u}^i_\nu) 
  \\
  &\qquad+\frac{\mu}{\tau}\nabla(\bm{u}^i_\nu-\bm{u}^{i-1}_\nu)\Big)+\nabla_{\bm{u}} G(x,\bm{u}^i_\nu)+[\nabla\alpha](\bm{u}^{i-1}_\nu)\gamma({B}(\bm{u}^{i-1}_\nu)\nabla\bm{u}^{i-1}_\nu)
  \\
  &\qquad+\alpha(\bm{u}^{i-1}_\nu)\nabla\gamma({B}(\bm{u}^{i-1}_\nu)\nabla\bm{u}^{i-1}_\nu):[\nabla{B}](\bm{u}^{i-1}_\nu)\nabla\bm{u}^{i-1}_\nu=\bm{0}~\mathrm{in}~[H]^M,
  \\
  &\Big(\alpha(\bm{u}^i_\nu) B^*(\bm{u}^i_\nu) \nabla\gamma({B}(\bm{u}^i_\nu)\nabla\bm{u}^i_\nu)+\nu \nabla\Upsilon_p(\nabla \bm{u}^i_\nu)+\frac{\mu}{\tau}\nabla(\bm{u}^i_\nu-\bm{u}^{i-1}_\nu)\Big)\bm{n}_\Gamma=\bm{0}~\mbox{on}~\Gamma,
  \\
  &\hspace{20ex}~~\mbox{for }i=1,2,\dots,m.
  \end{align*}

The solution to (AP)$^\tau_\nu$ is given as follows.
\begin{definition}
  A sequence of function $\{\bm{u}^{i}_\nu\}_{i=1}^m$ is called a solution to (AP)$^\tau_\nu$ with the initial data $\bm{u}_\nu^0 \in [W^{1,p}(\Omega)]^M$, if $\{\bm{u}^{i}_\nu\}_{i=1}^m \subset [W^{1,p}(\Omega)]^M$, and $\bm{u}^{i}_\nu$ fulfills the following variational identity for any $i = 1,2,\dots,m$:
    \begin{align}
    &\frac{1}{\tau}(A(\bm{u}^{i-1}_\nu)(\bm{u}^i_\nu-\bm{u}^{i-1}_\nu),\bm{\varphi})_{[H]^M}+\frac{\mu}{\tau}(\nabla(\bm{u}^i_\nu-\bm{u}^{i-1}_\nu),\nabla\bm{\varphi})_{[H]^{M\times N}}
    \\
    &\quad+(\alpha(\bm{u}^i_\nu) B^*(\bm{u}^i_\nu) \nabla\gamma({B}(\bm{u}^i_\nu)\nabla\bm{u}^i_\nu),\nabla\bm{\varphi})_{[H]^{M\times N}}
    +(\nabla_{\bm{u}} G(x,\bm{u}^i_\nu),\bm{\varphi})_{[H]^M}
    \\
    &\quad +\nu\int_\Omega \nabla \Upsilon_p(\nabla \bm{u}^i_\nu):\nabla\bm{\varphi}\,dx
    +([\nabla\alpha](\bm{u}^{i-1}_\nu)\gamma({B}(\bm{u}^{i-1}_\nu)\nabla\bm{u}^{i-1}_\nu),\bm{\varphi})_{[H]^M}
    \\
    &\quad+(\alpha(\bm{u}^{i-1}_\nu)\nabla\gamma({B}(\bm{u}^{i-1}_\nu)\nabla\bm{u}^{i-1}_\nu):[\nabla{B}](\bm{u}^{i-1}_\nu)\nabla\bm{u}^{i-1}_\nu,\bm{\varphi})_{[H]^M}=0,\label{3TimeDis-01}
    \\
    &\hspace*{20ex}\mbox{ for  any }\bm{\varphi}\in [W^{1,p}(\Omega)]^M.
  \end{align}
\end{definition}

\begin{theorem}(Existence, uniqueness of solution with energy-dissipation)\label{003Thm1}
  There exists a sufficiently small constant $\tau_0:=\tau_0(\nu,\|\nabla\gamma\|_{W^{1,\infty}})\in (0,1)$, possibly dependent on $\nu$ and $\gamma$, such that for any $ \tau \in (0,\tau_0) $, $\mathrm{(AP)}^\tau_\nu$ admits a unique solution $\{\bm{u}^i_\nu\}_{i=1}^m$. Moreover, the solution $\{\bm{u}^i_\nu\}_{i=1}^m$ fulfills the following energy-inequality:
  \begin{gather}
    \frac{C_A}{4\tau}|\bm{u}^{i}_\nu-\bm{u}^{i-1}_\nu|^2_{[H]^M}+\frac{\mu}{2\tau}|\nabla(\bm{u}^{i}_\nu-\bm{u}^{i-1}_\nu)|^2_{[H]^{M\times N}} +E_\nu(\bm{u}^{i}_\nu) \leq E_\nu(\bm{u}^{i-1}_\nu),
    \nonumber
    \\
    \mbox{ for any } i=1,2,\dots,m.\label{f-ene0}
  \end{gather}
\end{theorem}

We prepare some lemmas for the proof of Theorem \ref{003Thm1}. 
\begin{lemma}\label{lem002}
For any $\bm{w}_0\in [V]^M$, there exists a sequence $ \{\bm{w}_\nu\}_{\nu\in(0,1)}\subset[W^{1,p}(\Omega)]^M $ such that 
\begin{align*}
  \bm{w}_\nu \rightarrow \bm{w}_0 \mbox{ in } [V]^M, \mbox{ and } E_\nu(\bm{w}_\nu)\rightarrow E(\bm{w}_0) \mbox{ as } \nu\downarrow 0.
\end{align*}
  \begin{proof}
    Since $C^\infty(\overline{\Omega})$ is dense in $V$, for any $\bm{w}_0\in[V]^M$, there exists a sequence $\{\widehat{\bm{w}}_{n}\}_{n\in\N}\subset [C^\infty(\overline{\Omega})]^M\subset [W^{1,p}(\Omega)]^M$ such that 
    \begin{align*}
      \begin{aligned}
        &\widehat{\bm{w}}_{n}\rightarrow \bm{w}_0 \mbox{ in } [V]^M,
        \\
        \widehat{\bm{w}}_n(x)\rightarrow \bm{w}_0(x) &\mbox{ and }\nabla \widehat{\bm{w}}_n(x)\rightarrow \nabla \bm{w}_0(x) \mbox{ for a.e. } x\in\Omega,
      \end{aligned}~~\mbox{ as } n\rightarrow\infty.
    \end{align*}
    Here, we set 
    \begin{align*}
      \widehat{\nu}_0:=1,~\widehat{\nu}_n\in (0,\widehat{\nu}_{n-1}), \mbox{ for } n\in\N, \mbox{ and } \nu\int_\Omega |\nabla \widehat{\bm{w}}_{n}|^p\,dx<2^{-n}, \mbox{ for } \nu\in(0,\widehat{\nu}_n),~n\in\N.
    \end{align*}
    Moreover, we define a sequence $\{\bm{w}_{\nu}\}_{\nu\in(0,1)}$ as follows:
    \begin{align*}
      \bm{w}_\nu:=\left\{
        \begin{aligned}
          & \widehat{\bm{w}}_n \mbox{ if } \widehat{\nu}_{n+1}\leq \nu<\widehat{\nu}_n,~n=0,1,2,\dots,
          \\
          & \widehat{\bm{w}}_1 \mbox{ if } \widehat{\nu}_0\leq \nu.
        \end{aligned}
      \right.
    \end{align*}
    Then, we can see that
    \begin{align}\label{lem0021}
      \bm{w}_\nu\rightarrow \bm{w}_0 \mbox{ in } [V]^M,\mbox{ and }\nu\int_\Omega |\nabla \bm{w}_\nu|^p\,dx\rightarrow 0 \mbox{ as } \nu\downarrow 0.
    \end{align}
    Also, from (A5), \eqref{lem0021} and Lebesgue's dominated convergence theorem, we can derive the following convergences as $\nu\downarrow 0$:
    \begin{align*}
      &\nu\int_\Omega \Upsilon_p(\nabla \bm{w}_\nu)\,dx\rightarrow 0,
      \\
      &\int_\Omega G(x,\bm{w}_\nu)\,dx\rightarrow \int_\Omega G(x,\bm{w}_0)\,dx.
    \end{align*}
    Moreover, by using (A3), \eqref{lem0021} and Lebesgue's dominated convergence theorem, we obtain
    \begin{align*}
      \nabla\bm{w}_\nu B_0(\bm{w}_\nu) \rightarrow \nabla \bm{w}_0 B_0(\bm{w}_0)\mbox{ in }[H]^{M\times N} \mbox{ as }\nu \downarrow 0.
    \end{align*}
    Therefore, from the Lipschitz continuity of $\alpha$ and $\gamma$, it follows that
    \begin{align*}
      \int_\Omega \alpha(\bm{w}_\nu)\gamma(\nabla\bm{w}_\nu B_0(\bm{w}_\nu))\,dx\rightarrow \int_\Omega \alpha(\bm{w}_0)\gamma(\nabla\bm{w}_0B_0(\bm{w}_0))\,dx \mbox{ as } \nu\downarrow 0.
    \end{align*}
    Thus, the proof is complete.
  \end{proof}
\end{lemma}

\begin{rem}\label{rem001}
    From the proof of Lemma \ref{lem002}, there exists a small constant $\nu_0\in(0,1)$ such that for any $\nu\in(0,\nu_0)$,
    \begin{align*}
      |E_\nu(\bm{w}_\nu)-E(\bm{w}_0)|<1.
    \end{align*}
\end{rem}

First, for an arbitrary $\bm{u}^\dagger\in[W^{1,p}(\Omega)]^M$ and $\overline{\bm{u}}\in[W^{1,p}(\Omega)]^M$, we consider the following elliptic problem $\mathrm{(E1)}_{\overline{\bm{u}}}$:
  \begin{align*}
  &\frac{A(\bm{u}^\dagger)}{\tau}(\bm{u}-\bm{u}^\dagger)-\mathrm{div}\Big(\alpha(\overline{\bm{u}}) B^*(\overline{\bm{u}}) \nabla\gamma({B}(\overline{\bm{u}})\nabla\bm{u})+\nu \nabla\Upsilon_p(\nabla \bm{u})
  \\
  &\qquad+\frac{\mu}{\tau}\nabla(\bm{u}-\bm{u}^\dagger)\Big)+\nabla_{\bm{u}} G(x,\overline{\bm{u}})+[\nabla\alpha](\bm{u}^\dagger)\gamma({B}(\bm{u}^\dagger)\nabla\bm{u}^\dagger)
  \\
  &\qquad+\alpha(\bm{u}^\dagger)\nabla\gamma({B}(\bm{u}^\dagger)\nabla\bm{u}^\dagger):[\nabla{B}](\bm{u}^\dagger)\nabla\bm{u}^\dagger=\bm{0}~\mathrm{in}~[H]^M.
  \end{align*}

\begin{lemma}\label{lem001u}
  For any $\tau\in(0,1)$, (E1)$_{\overline{\bm{u}}}$ admits a unique solution $\bm{u}\in[W^{1,p}(\Omega)]$ in the following variational sense:
  \begin{align*}
    &\frac{1}{\tau}(A(\bm{u}^\dagger)(\bm{u}-\bm{u}^\dagger),\bm{\psi})_{[H]^M}+\frac{\mu}{\tau}(\nabla(\bm{u}-\bm{u}^\dagger),\nabla\bm{\psi})_{[H]^{M\times N}}
    \\
    &\quad+(\alpha(\overline{\bm{u}}) B^*(\overline{\bm{u}}) \nabla\gamma({B}(\overline{\bm{u}})\nabla\bm{u}),\nabla\bm{\psi})_{[H]^{M\times N}}
    +(\nabla_{\bm{u}} G(x,\overline{\bm{u}}),\bm{\psi})_{[H]^M}
    \\
    &\quad +\nu\int_\Omega \nabla \Upsilon_p(\nabla \bm{u}):\nabla\bm{\psi}\,dx
    +([\nabla\alpha](\bm{u}^\dagger)\gamma({B}(\bm{u}^\dagger)\nabla\bm{u}^\dagger),\bm{\psi})_{[H]^M}
    \\
    &\quad+(\alpha(\bm{u}^\dagger)\nabla\gamma({B}(\bm{u}^\dagger)\nabla\bm{u}^\dagger):[\nabla{B}](\bm{u}^\dagger)\nabla\bm{u}^\dagger,\bm{\psi})_{[H]^M}=0,
    \mbox{ for  any }\bm{\psi}\in [W^{1,p}(\Omega)]^M.
  \end{align*}
  Moreover, there exist a small time-step constant $\tau_1\in(0,1)$, and a positive constant $C_0>0$, independent of $\overline{\bm{u}}\in [W^{1,p}(\Omega)]^M$, such that the unique solution $\bm{u}$ fulfill the following estimate:
    \begin{align}\label{lem0010}
    |\nabla \bm{u}|_{[L^p(\Omega)]^{M\times N}}^p\leq \frac{C_0}{\nu} (1+|\nabla \bm{u}^\dagger|_{[L^p(\Omega)]^{M\times N}}^p).
  \end{align}
  \begin{proof}
    First, for any $\bm{\overline{u}}\in[V]^M$, we consider a proper, l.s.c., strictly convex, and coercive function $\Psi:[H]^M\rightarrow(-\infty,\infty]$ defined as follows:
    \begin{align*}
      \Psi:\bm{z}\in[H]^M\mapsto \Psi(\bm{z}):=\left\{
        \begin{aligned}
          &\frac{1}{2\tau}\int_\Omega |\sqrt{A(\bm{u}^\dagger)}(\bm{z}-\bm{u}^\dagger)|^2\,dx+\frac{\mu}{2\tau}\int_\Omega|\nabla(\bm{z}-\bm{u}^\dagger)|^2\,dx
          \\
          &\quad+\int_\Omega\alpha(\bm{\overline{u}})\gamma(B(\bm{\overline{u}})\nabla\bm{z})\,dx+\int\nabla_{\bm{u}} G(x,\bm{\overline{u}})\cdot \bm{z}\,dx
          \\
          &\quad+\nu\int_\Omega\Upsilon_p(\nabla\bm{z})\,dx+\int_\Omega\big([\nabla\alpha](\bm{u}^\dagger)\gamma({B}(\bm{u}^\dagger)\nabla\bm{u}^\dagger)\big)\cdot\bm{z}\,dx
          \\
          &\quad+\int_\Omega\big(\alpha(\bm{u}^\dagger)\nabla\gamma({B}(\bm{u}^\dagger)\nabla\bm{u}^\dagger):[\nabla{B}](\bm{u}^\dagger)\nabla\bm{u}^\dagger\big)\cdot\bm{z}\,dx
        \end{aligned}
      \right.
    \end{align*} 
    Since $\Psi$ is strictly convex, the minimizer of $\Psi$ is unique, and it directly leads to the existence and uniqueness of solution $\bm{u}\in [W^{1,p}(\Omega)]^M$ to problem (E1)$_{\overline{\bm{u}}}$.

    Let $\bm{u}\in \argmin_{\bm{z}\in[H]^M}\Psi(\bm{z})$. Then, noting $\Psi(\bm{u})\leq \Psi(\bm{u}^\dagger)$, we can derive:
  \begin{align*}
    &\frac{C_A\land \mu}{2\tau}|\bm{u}-\bm{u}^\dagger|^2_{[V]^M}+\nu\int_\Omega \Upsilon_p(\nabla \bm{u})\,dx
    \\
    &\leq \Psi(\bm{u}^\dagger)-\int \nabla_{\bm{u}} G(x,\bm{\overline{u}})\cdot \bm{z}\,dx-\int_\Omega\big([\nabla\alpha](\bm{u}^\dagger)\gamma(\nabla\bm{u}^\dagger{B}_0(\bm{u}^\dagger))\big)\cdot\bm{u}\,dx
    \\
    &\quad-\int_\Omega\big(\alpha(\bm{u}^\dagger)\nabla\gamma(\nabla\bm{u}^\dagger{B}_0(\bm{u}^\dagger)):\nabla\bm{u}^\dagger[\nabla{B}_0](\bm{u}^\dagger)\big)\cdot\bm{u}\,dx
    \\
    &\leq \nu \int_\Omega\Upsilon_p(\nabla \bm{u}^\dagger)\,dx +\int_\Omega\alpha(\bm{\overline{u}})\gamma(\nabla\bm{u}^\dagger{B}_0(\bm{\overline{u}}))\,dx
    \\
    &\quad +\int \nabla_{\bm{u}} G(x,\bm{\overline{u}})\cdot (\bm{u}^\dagger-\bm{u})\,dx+\int_\Omega\big([\nabla\alpha](\bm{u}^\dagger)\gamma(\nabla\bm{u}^\dagger{B}_0(\bm{u}^\dagger))\big)\cdot(\bm{u}^\dagger-\bm{u})\,dx
    \\
    &\quad+\int_\Omega\big(\alpha(\bm{u}^\dagger)\nabla\gamma(\nabla\bm{u}^\dagger{B}_0(\bm{u}^\dagger)):\nabla\bm{u}^\dagger[\nabla{B}_0](\bm{u}^\dagger)\big)\cdot(\bm{u}^\dagger-\bm{u})\,dx
    \\
    &\leq C_\Upsilon\int_\Omega(|\nabla\bm{u}^\dagger|^p+1)\,dx+\|\alpha\|_{L^\infty} C_\gamma\int_\Omega(|\nabla\bm{u}^\dagger B_0(\overline{\bm{u}})|+1)\,dx
    \\
    &\quad +\frac{C_A\land \mu}{8\tau}|\bm{u}-\bm{u}^\dagger|^2_{[H]^M}+\frac{2\tau}{C_A\land \mu}L_G|\Omega|+\frac{C_A\land \mu}{8\tau}|\bm{u}-\bm{u}^\dagger|^2_{[H]^M}
    \\
    &\quad +\frac{2\tau}{C_A\land \mu}\|\nabla\alpha\|^2_{L^\infty} C_\gamma^2\int_\Omega(|\nabla\bm{u}^\dagger B_0(\overline{\bm{u}})|+1)^2\,dx+\frac{C_A\land \mu}{8\tau}|\bm{u}-\bm{u}^\dagger|^2_{[H]^M}
    \\
    &\quad +\frac{2\tau}{C_A\land \mu}\|\alpha\|_{L^\infty}\|\nabla\gamma\|_{L^\infty}\|\nabla B_0\|_{L^\infty}|\nabla\bm{u}^\dagger|^2_{[H]^{M\times N}}
    \\
    &\leq C_\Upsilon\int_\Omega(|\nabla\bm{u}^\dagger|^p+1)\,dx +C_\gamma\|\alpha\|_{W^{2,\infty}}(\|B_0\|_{W^{2,\infty}}+1)\int_\Omega(|\nabla\bm{u}^\dagger|+1)\,dx
    \\
    &\quad +\frac{3(C_A\land \mu)}{8\tau}|\bm{u}-\bm{u}^\dagger|^2_{[V]^M}+\frac{2\tau}{C_A\land \mu}L_G|\Omega|
    \\
    &\quad+\frac{2\tau}{C_A\land \mu}\|\alpha\|_{W^{2,\infty}} C_\gamma^2(\|B_0\|_{W^{2,\infty}}+1)^2\int_\Omega(|\nabla\bm{u}^\dagger|+1)^2\,dx
    \\
    &\quad +\frac{2\tau}{C_A\land \mu}\|\alpha\|_{W^{2,\infty}}^2\|\nabla\gamma\|_{L^\infty}^2\|B_0\|_{W^{2,\infty}}^2|\nabla\bm{u}^\dagger|^2_{[H]^{M\times N}},
  \end{align*}
  and therefore, 
  \begin{align}
    &\frac{\nu}{C_\Upsilon}\int_\Omega(|\nabla\bm{u}|^p-1)\,dx
    \\
    &\quad\leq \left((C_\Upsilon+|\Omega|) +C_\gamma\|\alpha\|_{W^{2,\infty}}(\|B_0\|_{W^{2,\infty}}+1)\right)(|\nabla\bm{u}|_{[L^p(\Omega)]^{M\times N}}+1)^p
    \\
    &\quad +\frac{2\tau}{C_A \land \mu}(|\nabla\bm{u}|_{[L^p(\Omega)]^{M\times N}}+1)^p\left(L_G|\Omega|+\|\alpha\|_{W^{2,\infty}}^2\|\nabla\gamma\|_{L^\infty}^2\|B_0\|_{W^{2,\infty}}^2\right.
    \\
    &\qquad \left. +\|\alpha\|_{W^{2,\infty}} C_\gamma^2(\|B_0\|_{W^{2,\infty}}+1)^2 \right).\label{dm01}
  \end{align}
  Now, we take a small constant $\tau_1\in(0,1)$ such that:
  \begin{align*}
    \tau_1:=\min\left\{1,\frac{C_A\land \mu}{2L_G|\Omega|+\|\alpha\|_{W^{2,\infty}}^2\|\nabla\gamma\|_{L^\infty}^2\|B_0\|_{W^{2,\infty}}^2+\|\alpha\|_{W^{2,\infty}} C_\gamma^2(\|B_0\|_{W^{2,\infty}}+1)^2}\right\}
  \end{align*}
  so that 
  \begin{align*}
    &\tau \cdot \frac{2}{C_A}\left(L_G|\Omega|+\|\alpha\|_{W^{2,\infty}}^2\|\nabla\gamma\|_{L^\infty}^2\|B_0\|_{W^{2,\infty}}^2+\|\alpha\|_{W^{2,\infty}} C_\gamma^2(\|B_0\|_{W^{2,\infty}}+1)^2\right)
    \\
    &\qquad \leq1, \mbox{ for any }\tau \in (0,\tau_1).
  \end{align*}
  Then, \eqref{dm01} can be reduced to:
  \begin{align*}
    \frac{\nu}{C_\Upsilon}\int_\Omega(|\nabla\bm{u}|^p-1)\,dx&\leq \left(C_\Upsilon +C_\gamma\|\alpha\|_{W^{2,\infty}}(\|B_0\|_{W^{2,\infty}}+1)+1\right)\cdot 
    \\
    &\qquad \cdot (|\nabla\bm{u}|_{[L^p(\Omega)]^{M\times N}}+1)^p
    \\
    &\leq 2^{p-1}\left(C_\Upsilon +C_\gamma\|\alpha\|_{W^{2,\infty}}(\|B_0\|_{W^{2,\infty}}+1)+1\right)\cdot 
    \\
    &\qquad \cdot (|\nabla\bm{u}|_{[L^p(\Omega)]^{M\times N}}^p+1).
  \end{align*}

  Thus, Lemma \ref{lem001u} is conclude with the constant:
  \begin{align*}
    C_0:=2^{p-1}(C_\Upsilon+|\Omega|)\left(C_\Upsilon +C_\gamma\|\alpha\|_{W^{2,\infty}}(\|B_0\|_{W^{2,\infty}}+1)+2\right).
  \end{align*}
  \end{proof}
\end{lemma}

\begin{rem}
  In what follows, the estimate \eqref{lem0010} is used in the form of 
  \begin{align}\label{nu001}
    |\nabla\bm{u}|^2_{[L^p(\Omega)]^{M\times N}}\leq \frac{\widetilde{C}_0}{\nu^\frac{4}{p}}(1+E_\nu(\bm{u}^\dagger)),
  \end{align}
  by setting $\widetilde{C}_0:=C_0(1+|\Omega|+C_\Upsilon)$.
\end{rem}

  We next consider the following elliptic problem $\mathrm{(E2)}$:
  \begin{align*}
  &\frac{A(\bm{u}^\dagger)}{\tau}(\bm{u}-\bm{u}^\dagger)-\mathrm{div}\Big(\alpha({\bm{u}}) B^*({\bm{u}}) \nabla\gamma({B}({\bm{u}})\nabla\bm{u})+\nu \nabla\Upsilon_p(\nabla \bm{u})
  \\
  &\qquad+\frac{\mu}{\tau}\nabla(\bm{u}-\bm{u}^\dagger)\Big)+\nabla_{\bm{u}} G(x,{\bm{u}})+[\nabla\alpha](\bm{u}^\dagger)\gamma({B}(\bm{u}^\dagger)\nabla\bm{u}^\dagger)
  \\
  &\qquad+\alpha(\bm{u}^\dagger)\nabla\gamma({B}(\bm{u}^\dagger)\nabla\bm{u}^\dagger):[\nabla{B}](\bm{u}^\dagger)\nabla\bm{u}^\dagger=\bm{0}~\mathrm{in}~[H]^M.
  \end{align*}

  \begin{lemma}\label{lem001uka}
    Let $\tau_1\in(0,1)$ be as in Lemma \ref{lem001u}. Then, there exists a small time-step $\tau_2 \in (0,\tau_1)$ such that (E2) admits a unique solution $\bm{u}\in[W^{1,p}(\Omega)]^M$ in the following variational sense:
    \begin{align}
    &\frac{1}{\tau}(A(\bm{u}^\dagger)(\bm{u}-\bm{u}^\dagger),\bm{\psi})_{[H]^M}+\frac{\mu}{\tau}(\nabla(\bm{u}-\bm{u}^\dagger),\nabla\bm{\psi})_{[H]^{M\times N}}
    \\
    &\quad+(\alpha({\bm{u}}) B^*({\bm{u}}) \nabla\gamma({B}({\bm{u}})\nabla\bm{u}),\nabla\bm{\psi})_{[H]^{M\times N}}
    +(\nabla_{\bm{u}} G(x,{\bm{u}}),\bm{\psi})_{[H]^M}
    \\
    &\quad +\nu\int_\Omega \nabla \Upsilon_p(\nabla \bm{u}):\nabla\bm{\psi}\,dx
    +([\nabla\alpha](\bm{u}^\dagger)\gamma({B}(\bm{u}^\dagger)\nabla\bm{u}^\dagger),\bm{\psi})_{[H]^M} \label{dm03}
    \\
    &\quad+(\alpha(\bm{u}^\dagger)\nabla\gamma({B}(\bm{u}^\dagger)\nabla\bm{u}^\dagger):[\nabla{B}](\bm{u}^\dagger)\nabla\bm{u}^\dagger,\bm{\psi})_{[H]^M}=0,
    \mbox{ for  any }\bm{\psi}\in [W^{1,p}(\Omega)]^M.
  \end{align}
  \begin{proof}
  We define an operator $S_\tau:[V]^M\rightarrow[V]^M$ as the mapping that assigns to each $\bm{\overline{u}}\in[V]^M$ the unique solution of the problem (E2). To ensure that $S_\tau$ is contractive, we investigate the condition under which the parameter $\tau$ is sufficiently small. For the simple computation, we take a small time-step constant $\tau_1^*\in(0,\tau_1)$ satisfying:
  \begin{align}\label{dm02}
    \tau_1^*:=\min\left\{\tau_1,\frac{C_A\land \mu}{L_G+\|\alpha\|^2_{W^{1,\infty}}\|B_0\|_{W^{2,\infty}}^2(\|\nabla\gamma\|_{L^\infty}^2+1)}\right\}
  \end{align}
  so that 
  \begin{align*}
    \frac{\tau}{C_A\land\mu}(L_G+\|\alpha\|^2_{W^{1,\infty}}\|B\|_{W^{2,\infty}}^2\|\nabla\gamma\|_{L^\infty}^2)\leq 1, \mbox{ for any }\tau \in (0,\tau_1^*).
  \end{align*}
  Let $\bm{u}_k:=S_\tau\bm{\overline{u}}_k$, $k=1,2$. By taking difference of (E2), multiplying both sides by $(\bm{u}_1-\bm{u}_2)$, and summing the two equations, we deduce from (A1)--(A5) that 
  \begin{align}
    &\frac{C_A\land\mu}{\tau}|\bm{u}_1-\bm{u}_2|^2_{[V]^M}
    \\
    &\leq L_G\int_\Omega|\bm{\overline{u}}_1-\bm{\overline{u}}_2||\bm{u}_1-\bm{u}_2|\,dx
    \\
    &\quad +\|\nabla\alpha\|_{L^\infty}\|B_0\|_{L^\infty}\|\nabla\gamma\|_{L^\infty}\int_\Omega|\bm{\overline{u}}_1-\bm{\overline{u}_2}||\nabla(\bm{u}_1-\bm{u}_2)|\,dx
    \\
    &\quad +\|\alpha\|_{L^\infty}\|\nabla B_0\|_{L^\infty}\|\nabla\gamma\|_{L^\infty}\int_\Omega|\bm{\overline{u}}_1-\bm{\overline{u}_2}||\nabla(\bm{u}_1-\bm{u}_2)|\,dx
    \\
    &\quad +\|\alpha\|_{L^\infty}\|B_0\|_{W^{1,\infty}}\|\nabla^2\gamma\|_{L^\infty}\int_\Omega|\bm{\overline{u}}_1-\bm{\overline{u}_2}||\nabla \bm{u}_1||\nabla(\bm{u}_1-\bm{u}_2)|\,dx
    \\
    &\quad -\int_\Omega \alpha(\overline{\bm{u}}_2)\bigl([\nabla\gamma](\nabla\bm{u}_1B_0(\overline{\bm{u}}_1))-[\nabla\gamma](\nabla\bm{u}_2B_0(\overline{\bm{u}}_1))\bigr): \nabla(\bm{u}_1-\bm{u}_2)B_0(\overline{\bm{u}}_1)\,dx
    \\
    &\leq \frac{C_A\land\mu}{2\tau}|\bm{u}_1-\bm{u}_2|^2_{[V]^M}
    \\
    &\quad \frac{2\tau}{C_A\land\mu}(L_G+\|\alpha\|^2_{W^{1,\infty}}\|B_0\|_{W^{1,\infty}}^2\|\nabla\gamma\|_{L^\infty}^2)|\overline{\bm{u}}_1-\overline{\bm{u}}_2|^2_{[H]^M}
    \\
    &\quad +\frac{2\tau}{C_A\land\mu}\|\alpha\|^2_{W^{1,\infty}}\|B_0\|_{W^{1,\infty}}^2\|\nabla^2\gamma\|_{L^\infty}^2\int_\Omega |\bm{\overline{u}}_1-\bm{\overline{u}}_2|^2|\nabla \bm{u}_1|^2\,dx.\label{lem0013}
  \end{align}
  Having \eqref{dm02} in mind, we can reduce \eqref{lem0013} by:
  \begin{gather}
    \frac{C_A\land\mu}{2\tau}|\bm{u}_1-\bm{u}_2|^2_{[V]^M}\leq 2|\overline{\bm{u}}_1-\overline{\bm{u}}_2|^2_{[H]^M}+2|\nabla^2\gamma|^2_{L^\infty}\int_\Omega |\bm{\overline{u}}_1-\bm{\overline{u}}_2|^2|\nabla \bm{u}_1|^2\,dx,
    \\
    \mbox{ for any }\tau \in (0,\tau_1^*).\label{lem0014}
  \end{gather}
The estimate of the final term of \eqref{lem0014} will depend on the spatial dimension. Indeed, when $N=1,2$, the continuous embedding from $V$ to $L^\frac{2p}{p-2}(\Omega)$ is valid, and hence, on account of \eqref{nu001}, it is estimated by:
\begin{align}
  &2|\nabla^2\gamma|^2_{L^\infty}\int_\Omega |\bm{\overline{u}}_1-\bm{\overline{u}}_2|^2|\nabla \bm{u}_1|^2\,dx\leq2|\nabla^2\gamma|^2_{L^\infty} |\bm{\overline{u}}_1-\bm{\overline{u}}_2|^2_{[L^\frac{2p}{p-2}(\Omega)]^M}|\nabla \bm{u}_1|^2_{[L^p(\Omega)]^{M\times N}}
  \\
  &\quad\leq \frac{1}{\nu^\frac{4}{p}}\cdot2\widetilde{C}_0\left(C_V^{L^\frac{2p}{p-2}}\right)^2|\nabla^2\gamma|^2_{L^\infty}(1+E_\nu(\bm{u}^\dagger))|\bm{\overline{u}}_1-\bm{\overline{u}}_2|^2_{[V]^M}.\label{lem0015}
\end{align}
On the other hand, if $N\geq3$, then by using the continuous embedding from $V$ to $L^\frac{2N}{N-2}$ and by recalling the assumption $N\leq p$, we can compute:
\begin{align}
  &2|\nabla^2\gamma|^2_{L^\infty}\int_\Omega |\bm{\overline{u}}_1-\bm{\overline{u}}_2|^2|\nabla \bm{u}_1|^2\,dx\leq2|\nabla^2\gamma|^2_{L^\infty} |\bm{\overline{u}}_1-\bm{\overline{u}}_2|^2_{[L^\frac{2N}{N-2}(\Omega)]^M}|\nabla \bm{u}_1|^2_{[L^N(\Omega)]^{M\times N}}
  \\
  &\leq 2\widetilde{C}_0\left(C_V^{L^\frac{2N}{N-2}}\right)^2|\nabla^2\gamma|^2_{L^\infty}|\bm{\overline{u}}_1-\bm{\overline{u}}_2|^2_{[V]^M}\cdot|\Omega|^\frac{2(p-N)}{Np}|\nabla \bm{u}_1|^2_{[L^p(\Omega)]^{M\times N}}
  \\
  &\leq \frac{1}{\nu^\frac{4}{p}}\cdot2\widetilde{C}_0\left(C_V^{L^\frac{2N}{N-2}}\right)^2|\Omega|^\frac{2(p-N)}{Np}|\nabla^2\gamma|^2_{L^\infty}(1+E_\nu(\bm{u}^\dagger))|\bm{\overline{u}}_1-\bm{\overline{u}}_2|^2_{[V]^M}.\label{lem0016}
\end{align}\noeqref{lem0015}
Therefore, according to \eqref{lem0014}--\eqref{lem0016}, we arrive at the smallness condition of $\tau$, with the constant $\tau_2:=\tau_2(\bm{u}^\dagger,\nu,|\nabla^2\gamma|_{L^\infty})$, given by:
\begin{align*}
  \tau_2:=\min\left\{ \tau_1^*,\frac{\nu^\frac{4}{p}(C_A\land \mu)}{8\widetilde{C}_0(1+E_\nu(\bm{u}^\dagger))(1+|\nabla^2\gamma|^2_{L^\infty})\Bigl(\Bigl(C_V^{L^\frac{2p}{p-2}}\Bigr)^2+\Bigl(C_V^{L^\frac{2N}{N-2}}\Bigr)^2|\Omega|^\frac{2(p-N)}{Np}+1  \Bigr)} \right\}
\end{align*}
so that 
\begin{align*}
  &\left\{
    \begin{aligned}
      &\frac{\tau}{\nu^\frac{4}{p}}(1+E_\nu(\bm{u}^\dagger))(1+|\nabla^2\gamma|^2_{L^\infty})\cdot\frac{4\Bigl(\widetilde{C}_0\Bigl(C_V^{L^\frac{2p}{p-2}}\Bigr)^2+1\Bigr)}{C_A\land\mu}\leq\frac{1}{2},
      \\
      &\frac{\tau}{\nu^\frac{4}{p}}(1+E_\nu(\bm{u}^\dagger))(1+|\nabla^2\gamma|^2_{L^\infty})\cdot\frac{4\Bigl(\widetilde{C}_0\Bigl(C_V^{L^\frac{2N}{N-2}}|\Omega|^\frac{2(p-N)}{Np}\Bigr)^2+1\Bigr)}{C_A\land\mu}\leq\frac{1}{2},
    \end{aligned}
  \right.
  \\
  &\hspace{25ex} \mbox{ for any }\tau \in (0,\tau_2).
\end{align*}
Under the smallness condition $\tau\in(0,\tau_2)$, $S_\tau$ is contractive, and thus, we conclude Lemma \ref{lem001uka}.
\end{proof}
\end{lemma}

Next, for the guarantee of the energy-inequality as in \eqref{f-ene0}, we handle fundamental computations.
\begin{lemma} \label{lem001dm}
  Let $\tau_2 \in (0,1)$ be as in Lemma \ref{lem002}. Let $\bm{u} \in [W^{1,p}(\Omega)]^M$ be the unique solution to (E2) for an arbitrary $\bm{u}^\dagger \in [W^{1,p}(\Omega)]^M$. Then, there exists a small time-step constant $\tau_3 = \tau_3(\bm{u}^\dagger, \nu, \|\nabla^2\gamma\|_{L^\infty}) \in (0,\tau_2)$, depending on $|\nabla \bm{u}^\dagger|_{[L^p(\Omega)]^{M \times N}}$, $\nu \in (0,1)$ and $\| \nabla^2 \gamma \|_{L^\infty}$, such that for any $\tau \in (0,\tau_3)$, the following energy-inequality at each step holds:
  \begin{gather}
    \frac{C_A}{4\tau} |\bm{u} - \bm{u}^\dagger|_{[H]^M}^2 + \frac{\mu}{2 \tau} |\nabla (\bm{u} - \bm{u}^\dagger)|_{[H]^{M \times N}}^2 + E_\nu(\bm{u}) \leq E_\nu(\bm{u}^\dagger). \label{dm04}
  \end{gather}
\end{lemma}

\begin{proof}
  For simplicity, let us define
  \begin{align*}
    f(\bm{v},W):=\alpha(\bm{v})\gamma(B(\bm{v})W), \mbox{ for } [\bm{v},W] \in \R^M \times \R^{M\times N}.
  \end{align*}
  Let $\bm{\psi}:=\bm{u}-\bm{u}^\dagger$ in \eqref{dm03}. By Young's inequality, assumption (A4), and the convexity of $\gamma$ and $\Upsilon_p$, we have
  \begin{align}
    &\frac{C_A}{\tau}|\bm{u}-\bm{u}^\dagger|_{[H]^M}^2+\frac{\mu}{\tau}|\nabla(\bm{u}-\bm{u}^\dagger)|_{[H]^{M\times N}}^2+\nu\int_\Omega \Upsilon_p(\nabla \bm{u})\,dx-\nu\int_\Omega \Upsilon_p(\nabla \bm{u}^\dagger)\,dx
    \\
    &(\nabla_{\bm{u}} G(x,\bm{u}),\bm{u}-\bm{u}^\dagger)_{[H]^M}+ \int_\Omega f(\bm{u},\nabla \bm{u})\,dx-\int_\Omega f(\bm{u},\nabla \bm{u}^\dagger)\,dx
    \\
    &+\int_\Omega[\partial_{\bm{v}}f](\bm{u}^\dagger,\nabla \bm{u}^\dagger)\cdot(\bm{u}-\bm{u}^\dagger)\,dx\leq 0. \label{thm0010}
  \end{align}
  By using (A2), it is observed that
  \begin{align}
    &(\nabla_{\bm{u}} G(x,\bm{u}),\bm{u}-\bm{u}^\dagger)_{[H]^M}\geq \int_\Omega G(x,\bm{u})\,dx-\int_\Omega G(x,\bm{u}^\dagger)\,dx
    \\
    &\quad +(\nabla_{\bm{u}} G(x,\bm{u})-\nabla_{\bm{u}} G(x,\bm{u}^\dagger),\bm{u}-\bm{u}^\dagger)_{[H]^M}-\frac{1}{2}L_G|\bm{u}-\bm{u}^\dagger|_{[H]^M}^2
    \\
    &\geq \int_\Omega G(x,\bm{u})\,dx-\int_\Omega G(x,\bm{u}^\dagger)\,dx -\frac{3}{2}L_G|\bm{u}-\bm{u}^\dagger|_{[H]^M}^2. \label{thm0011}
  \end{align}
  Thanks to \eqref{thm0010} and \eqref{thm0011}, and by applying Taylor's theorem for $ f \in C^{1,1}(\R^M\times \R^{M\times N})\cap C^2(\R^M\times \R^{M\times N})$, we obtain
  \begin{align}
    &\left(\frac{C_A}{\tau}-\frac{3}{2}L_G\right)|\bm{u}-\bm{u}^\dagger|_{[H]^M}^2+\frac{\mu}{\tau}|\nabla(\bm{u}-\bm{u}^\dagger)|_{[H]^{M\times N}}^2
    +E_\nu(\bm{u})-E_\nu(\bm{u}^\dagger)
    \\
    &\leq \int_\Omega \Big[ f(\bm{u},\nabla \bm{u}^\dagger)-f(\bm{u}^\dagger,\nabla \bm{u}^\dagger)-[\partial_{\bm{v}}f](\bm{u}^\dagger,\nabla \bm{u}^\dagger)(\bm{u}-\bm{u}^\dagger)\Big]\,dx
    \\
    &\leq \int_\Omega \int_0^1(1-\sigma){}^\top(\bm{u}-\bm{u}^\dagger)\frac{\partial^2}{\partial \bm{v}^2}\Big[f(\bm{u}^\dagger+\sigma(\bm{u}-\bm{u}^\dagger),\nabla \bm{u}^\dagger)\Big](\bm{u}-\bm{u}^\dagger)\,d\sigma dx.\label{thm0012}
  \end{align}
  Here, by using the chain-rule, we easily check that
  \begin{gather}
    \left| \frac{\partial^2}{\partial \bm{v}^2}\Big[f(\bm{v},W)\Big] \right|\leq C_2(1+\|\nabla\gamma\|_{W^{1,\infty}})^2(1+|W|^2), 
    \\
    \mbox{ for all }[\bm{v},W]\in \R^M\times \R^{M\times N},\label{thm0013}
  \end{gather}
  where 
  \begin{align*}
    C_2:=2(1+\|\alpha\|_{W^{2,\infty}})^2(1+\|B_0\|_{W^{2,\infty}})^2.
  \end{align*}
  So as a consequence of \eqref{thm0012} and \eqref{thm0013}, it is inferred that
\begin{align}
    &\left(\frac{C_A}{\tau}-\frac{3}{2}L_G\right)|\bm{u}-\bm{u}^\dagger|_{[H]^M}^2+\frac{\mu}{\tau}|\nabla(\bm{u}-\bm{u}^\dagger)|_{[H]^{M\times N}}^2+E_\nu(\bm{u})-E_\nu(\bm{u}^\dagger)
    \\
    &\leq C_2(1+\|\nabla\gamma\|_{W^{1,\infty}})^2\int_\Omega(1+|\nabla\bm{u}^\dagger|^2)(\bm{u}-\bm{u}^\dagger)^2\,dx.
    \label{thm0014}
  \end{align}
  However, we analyze the principal integral appearing on the right-hand side of \eqref{thm0014}, by distinguishing cases according to the spatial dimension and evaluating it under each corresponding condition. First, we consider the case when $N = 1,2$. By using (A5) and the Sobolev embedding $H^1(\Omega) \subset L^{\frac{2p}{p-2}}(\Omega)$, we can compute the principal integral as follows:
  \begin{align}
    &\int_\Omega(1+|\nabla\bm{u}^\dagger|^2)(\bm{u}-\bm{u}^\dagger)^2\,dx
    \leq |\bm{u}-\bm{u}^\dagger|_{[L^{\frac{2p}{p-2}}(\Omega)]^M}^2|\nabla \bm{u}^\dagger|_{[L^p(\Omega)]^{M\times N}}^2+|\bm{u}-\bm{u}^\dagger|_{[H]^M}^2
    \\
    &\leq (C_{H^1}^{L^{\frac{2p}{p-2}}})^2|\nabla \bm{u}^\dagger|_{[L^p(\Omega)]^{M\times N}}^2|\bm{u}-\bm{u}^\dagger|_{[V]^M}^2+|\bm{u}-\bm{u}^\dagger|_{[H]^M}^2
    \\
    &\leq \frac{(C_{H^1}^{L^{\frac{2p}{p-2}}})^2((C_\Upsilon)^\frac{2}{p}+1)+1}{(C_A \land \mu)\nu^\frac{4}{p}}(1+E_\nu(\bm{u}^\dagger)^\frac{2}{p})\cdot 
    \\
    &\qquad\cdot(C_A|\bm{u}-\bm{u}^\dagger|_{[H]^M}^2+\mu|\nabla(\bm{u}-\bm{u}^\dagger)|_{[H]^M}^2).\label{thm00152}
  \end{align}\noeqref{thm00152}
  Here, $C_{H^1}^{L^{\frac{2p}{p-2}}}$ denotes the constant associated with the Sobolev embedding $H^1(\Omega) \subset L^{\frac{2p}{p-2}}(\Omega)$. Lastly, we turn to the case $N \geq 3$. We estimate the principal integral by using (A5) and the Sobolev embedding $H^1(\Omega) \subset L^{\frac{2N}{N-2}}(\Omega)$.
\begin{align}
    &\int_\Omega(1+|\nabla\bm{u}^\dagger|^2)(\bm{u}-\bm{u}^\dagger)^2\,dx
    \\
    &\leq |\bm{u}-\bm{u}^\dagger|_{[L^{\frac{2N}{N-2}}(\Omega)]^M}^2|\nabla \bm{u}^\dagger|_{[L^p(\Omega)]^{M\times N}}^2|\Omega|^\frac{2(p-N)}{pN}+|\bm{u}-\bm{u}^\dagger|_{[H]^M}^2
    \\
    &\leq |\Omega|^\frac{2(p-N)}{pN}(C_{H^1}^{L^{\frac{2N}{N-2}}})^2|\nabla \bm{u}^\dagger|_{[L^p(\Omega)]^{M\times N}}^2|\bm{u}-\bm{u}^\dagger|_{[V]^M}^2+|\bm{u}-\bm{u}^\dagger|_{[H]^M}^2
    \\
    &\leq \frac{|\Omega|^\frac{2(p-N)}{pN}(C_{H^1}^{L^{\frac{2N}{N-2}}})^2((C_\Upsilon)^\frac{2}{p}+1)+1}{(C_A \land \mu)\nu^\frac{2}{p}}(1+E_\nu(\bm{u}^\dagger)^\frac{2}{p})\cdot 
    \\
    &\qquad\cdot(C_A|\bm{u}-\bm{u}^\dagger|_{[H]^M}^2+\mu|\nabla(\bm{u}-\bm{u}^\dagger)|_{[H]^M}^2).\label{thm00153}
  \end{align}
  Here, $C_{H^1}^{L^{\frac{2N}{N-2}}}$ represents the constant in the Sobolev embedding $H^1(\Omega) \subset L^{\frac{2N}{N-2}}(\Omega)$. From \eqref{thm00152} and \eqref{thm00153}, we derive the following uniform estimate for the principal integral appearing on the right-hand side of \eqref{thm0014}, applicable to any spatial dimension $N \in \N$:
  \begin{align}
    &\int_\Omega(1+|\nabla\bm{u}^\dagger|^2)(\bm{u}-\bm{u}^\dagger)^2\,dx
    \\
    &\leq \frac{\widetilde{C}}{\nu^\frac{2}{p}}(1+E_\nu(\bm{u}^\dagger)^\frac{2}{p})(C_A|\bm{u}-\bm{u}^\dagger|_{[H]^M}^2+\mu|\nabla(\bm{u}-\bm{u}^\dagger)|_{[H]^M}^2),\label{thm00154}
  \end{align}
  where $\widetilde{C}>0$ is a positive constant defined by 
  \begin{align*}
    \widetilde{C}&:=\max\left\{\frac{(C_{H^1}^{L^{\frac{2p}{p-2}}})^2((C_\Upsilon)^\frac{2}{p}+1)+1}{C_A \land \mu},\frac{|\Omega|^\frac{2(p-N)}{pN}(C_{H^1}^{L^{\frac{2N}{N-2}}})^2((C_\Upsilon)^\frac{2}{p}+1)+1}{C_A \land \mu}\right\}.
  \end{align*}
  Due to \eqref{thm0012}, \eqref{thm0014}, and \eqref{thm00154}, we will infer that 
\begin{align}
    &\left(\frac{C_A}{\tau}-\frac{3}{2}L_G\right)|\bm{u}-\bm{u}^\dagger|_{[H]^M}^2+\frac{\mu}{\tau}|\nabla(\bm{u}-\bm{u}^\dagger)|_{[H]^{M\times N}}^2+E_\nu(\bm{u}^i_\nu)-E_\nu(\bm{u}^\dagger)
    \\
    &\leq \frac{C_2\widetilde{C}}{\nu^\frac{2}{p}}(1+\|\nabla\gamma\|_{W^{1,\infty}})^2(1+E_\nu(\bm{u}^\dagger)^\frac{2}{p})(C_A|\bm{u}-\bm{u}^\dagger|_{[H]^M}^2+\mu|\nabla(\bm{u}-\bm{u}^\dagger)|_{[H]^M}^2).
    \label{thm0016}
  \end{align}

  Finally, let us set $\tau_3 \in (0, \tau_2)$ by:
  \begin{align}
    \tau_3 &= \tau_3(\bm{u}^\dagger, \nu, \| \nabla^2 \gamma \|_{L^\infty})
    \\
    &:=\min \left\{ \tau_2,\frac{C_A}{6L_G},\frac{\nu^\frac{4}{p}}{4C_2\widetilde{C}(1+\|\nabla\gamma\|_{W^{1,\infty}})^2(1+E(\bm{u}_0)^\frac{2}{p})} \right\},
  \end{align}
  so that: 
  \begin{align}\label{thm0017}
    \left\{
      \begin{aligned}
        &\frac{2\tau C_2\widetilde{C}(1+\|\nabla\gamma\|_{W^{1,\infty}})^2(1+E(\bm{u}_0)^\frac{2}{p})}{\nu^\frac{4}{p}}\leq \frac{1}{2}, 
        \\
        &\frac{C_A}{2\tau}-\frac{3}{2}L_G^2\geq \frac{C_A}{4\tau},\mbox{ for } \tau \in (0,\tau_3).
      \end{aligned}
    \right.
  \end{align}
  Then, for $\tau \in (0,\tau_3)$, we can verify the each step energy-inequality \eqref{dm04}.
\end{proof}

\begin{proof}[Proof of Theorem \ref{003Thm1}]
  For $\bar{\bm{u}}, \hat{\bm{u}} \in [W^{1,p}(\Omega)]^M$, we note that:
  \begin{equation}
    \mbox{if } E(\bar{\bm{u}}) \leq E(\hat{\bm{u}}), \mbox{ then } \tau_3(\bar{\bm{u}},\nu,\| \nabla^2 \gamma \|_{L^\infty}) \geq \tau_3(\hat{\bm{u}},\nu,\| \nabla^2 \gamma \|_{L^\infty}), \label{dm05}
  \end{equation}
  where $\tau_3$ is as in Lemma \ref{lem001dm}. We set $\tau_0$ by:
  \begin{equation}
    \tau_0 = \tau_0(\nu,\| \nabla^2 \gamma\|_{L^\infty}) := \inf_{\nu \in (0, \nu_0)} \tau_3({\bm{u}_\nu^0, \nu, \| \nabla^2 \gamma \|_{L^\infty}}) > 0. \label{dm06}
  \end{equation}
  We give the proof by induction argument, i.e.: we show the following statement for each step $i = 1,2,\dots, m$:
  \begin{itemize}
    \item[($*$)]we can find a unique solution $\bm{u}_\nu^i$ to (AP)$_\nu^\tau$, satisfying the energy-inequality \eqref{f-ene0}.
  \end{itemize}
  At first, for $i = 1$, the existence and uniqueness of the solution to (AP)$^\tau_\nu$ is a direct consequence of Lemma \ref{lem001uka}. Indeed, by applying lemma \ref{lem001uka} when the case $\bm{u}^\dagger = \bm{u}_\nu^0$ and noting $\tau_0 < \tau_3({\bm{u}_\nu^0, \nu,\|\nabla^2 \gamma\|_{L^\infty}})$, we immediately obtain the unique solution $\bm{u}^1_\nu\in [W^{1,p}(\Omega)]^M$ to (AP)$^\tau_\nu$. Besides, by virtue of Lemma \ref{lem001dm}, we can confirm that $\bm{u}^1_\nu$ satisfies the energy-dissipation inequality.

  Next, we suppose that the statement ($*$) holds for all $i = 1,2,\dots, k$, with $k \leq m-1$, and consider the case when $i = k+1$. Thanks to the energy-inequality, we immediately see that
  \begin{align*}
    E_\nu(\bm{u}^k_\nu)\leq E_\nu(\bm{u}^{k-1}_\nu)\leq \cdots \leq E_\nu(\bm{u}^1_\nu)\leq E_\nu(\bm{u}^0_\nu), \mbox{ for all } \nu \in (0,\nu_0).
  \end{align*}
  Hence, on account of \eqref{dm05} and \eqref{dm06}, we have:
  \begin{equation}
    \tau_0 := \inf_{\nu \in (0,\nu_0)} \tau_3(\bm{u}_\nu^0, \nu, \| \nabla^2 \gamma \|_{L^\infty}) \leq \tau_3(\bm{u}_\nu^k, \nu, \| \nabla^2 \gamma \|_{L^\infty}).
  \end{equation}
  In this light, by applying Lemma \ref{lem001dm} again, we obtain the unique solution $\bm{u}_\nu^{k+1}$ satisfying the energy-dissipation \eqref{f-ene0}, and thus, we complete the proof of Theorem \ref{003Thm1}.
\end{proof}

\section{Proof of Main Theorem 1.}\label{sec:proof1}
First, we introduce a lemma that is essential for the proof of the Main Theorem \ref{mainThm1}. 
  \begin{lemma}\label{lem003}
    For any $\bm{\psi}\in [\mathscr{V}]^M$, and $\{\nu_k\}_{k\in\N}\subset (0,1);\nu_k\downarrow0$, there exists a sequence $\{ {\bm{\psi}_k} \}_{k\in\N} \subset [C^\infty(\overline{Q})]^M$ satisfying the following properties.
    \begin{align*}
      &\bm{\psi}_k \rightarrow\bm{\psi} \mbox{ in }[\mathscr{V}]^M \mbox{ as }k\rightarrow\infty,
      \\
      &\nu^{\frac{1}{p}}_k\nabla\bm{\psi}_k \rightarrow 0 \mbox{ in }L^p(0,T;[L^p(\Omega)]^{M\times N})\mbox{ as }k\rightarrow\infty.
    \end{align*}
    \begin{proof}
      Let $\bm{\psi}\in [\mathscr{V}]^M$ be fixed. Then, there exists a sequence $\{\widehat{\bm{\psi}}_n\}_{n\in\N}\subset [C^\infty(\overline{Q})]^M$ such that
      \begin{align*}
        \widehat{\bm{\psi}}_n \rightarrow \bm{\psi} \mbox{ in } [\mathscr{V}]^M \mbox{ as }m \rightarrow\infty.
      \end{align*}
      Moreover, we take the sequence $\{\widehat{\nu}_n\}_{n\in\N}\subset (0,1)$ such that 
      \begin{gather*}
        \widehat{\nu}_{n+1}<\widehat{\nu}_n<2^{-n},~~~\nu|\nabla \widehat{\bm{\psi}}_n|^p_{L^p(0,T;[L^p(\Omega)]^{M\times N})}<2^{-n},
        \\
        \mbox{ for any }n\in\N, \mbox{ and } \nu\in(0,\widehat{\nu}_n].
      \end{gather*}
         Thus, we construct a sequence $\{ {\bm{\psi}_k} \}_{k\in\N} \subset [C^\infty(\overline{Q})]^M$ in the following manner:
    \begin{align*}
      \bm{\psi}_k:=\left\{
        \begin{aligned}
        &\widehat{\bm{\psi}}_n &&\mbox{ if }\widehat{\nu}_{n+1} <\nu_k\leq \widehat{\nu}_n,
        \\
        &1 &&\mbox{ if }\widehat{\nu}_1 < \nu_k < 1.
      \end{aligned}\right.
    \end{align*}
    \end{proof}
  \end{lemma}

  \begin{proof}[Proof of Main Theorem 1]
    First, we define a regularization sequence $ \{ \gamma_\varepsilon \}_{\varepsilon \in (0, 1)} \subset C^{1, 1}(\R^{M\times N}) \cap C^2(\R^{M\times N}) $, by setting:
  \begin{gather*}
    \gamma_\varepsilon:=\left\{
      \begin{aligned}
          &\rho_\varepsilon*\gamma,&\mbox{ if }\gamma\notin C^{1,1}(\R^{M\times N}) \cap C^2(\R^{M\times N}),
        \\
          &\gamma,&\mbox{ if } \gamma\in C^{1,1}(\R^{M\times N}) \cap C^2(\R^{M\times N}),
      \end{aligned}\right.
      \\
      {\leftline{\mbox{with use of the standard mollifier $ \rho_\varepsilon \in C_\mathrm{c}^\infty(\R^{M\times N}) $, for any $ \varepsilon \in (0, 1) $.}}}
  \end{gather*}
    In addition, we define an approximating energy $E_{\nu,\varepsilon}$ on $[H]^M$, as follows:
    \begin{gather}
    E_{\nu,\varepsilon} :\bm{u} \in [V]^M \mapsto E_{\nu,\varepsilon}(\bm{u}):=\int_\Omega\alpha(\bm{u}){\gamma}_\varepsilon(B(\bm{u})\nabla \bm{u})\,dx+\int_\Omega G(x,\bm{u})\,dx\nonumber
    \\
    +{\nu}\int_\Omega\Upsilon_p(\nabla \bm{u})\,dx
    \in[0,\infty).\label{E_eps}
  \end{gather}
  As is easily seen, 
    \begin{gather}\label{gamma_ep}
        \begin{cases}
            \gamma_\varepsilon \to \gamma \mbox{ uniformly on $ \R^{M\times N} $ as $ \varepsilon \downarrow 0 $,}
            \\
            \|\nabla \gamma_\varepsilon\|_{L^\infty} \leq \|\nabla \gamma\|_{L^\infty}, \mbox{ for all $ \varepsilon \in (0, 1) $.}
        \end{cases}
    \end{gather}
    The uniform convergence of $ \{ \gamma_\varepsilon \}_{\varepsilon \in (0, 1)} $ implies
\begin{gather}\label{gammaconv}
    E_{\nu,\varepsilon}\rightarrow E_{\nu}\mbox{ on }[H]^M,\mbox{ in the sense of $\Gamma$-convergence, as }\varepsilon\downarrow0.
\end{gather}
Let $\{\bm{u}^0_\nu\}_{\nu\in(0,\nu_0)}$ be a sequence of initial value for Theorem \ref{003Thm1}, constructed by Lemma \ref{lem002} with $\bm{w}_0=\bm{u}_0$. Then, owing to and the uniform estimate for $ \{ \|\nabla \gamma_\varepsilon\|_{L^\infty} \}_{\varepsilon \in (0, 1)} $ and Remark \ref{rem001}, we have
   \begin{align}\label{constc_E}
       c_E:=\sup_{\substack{\nu\in(0,\nu_0)\\ \varepsilon\in(0,1)}}E_{\nu,\varepsilon}(\bm{u}_{0,\nu}) \leq E(\bm{u}_0)+1+\|\nabla \gamma\|_{L^\infty}|\Omega|<\infty,
   \end{align}
    {where $ |\Omega| $ denotes the Lebesgue measure of $ \Omega \subset \R^N $. }

    Now, for all $\nu \in (0,\nu_0)$ and $ \varepsilon \in (0, 1) $, we define a constant $ \tau_{*}(\nu,\varepsilon) $, which depends on both $\nu$ and $\varepsilon$, as follows:
 \begin{align*}
    \tau_{*}(\nu,\varepsilon)&:=\tau_0(\nu,\|\nabla\gamma_\varepsilon\|_{L^\infty}),
  \end{align*}
  where $\tau_0$ is as in Theorem \ref{003Thm1}. By applying Theorem \ref{003Thm1} to the case when $\gamma=\gamma_\varepsilon$, we find a sequence of functions $\{\bm{u}^i_{\nu,\varepsilon}\}_{i=1}^m$, for $\nu\in(0,\nu_0)$, $\varepsilon\in(0,1)$, and $\tau\in(0,\tau_{*}(\nu,\varepsilon))$, such that:
      \begin{align}
    &\frac{1}{\tau}(A(\bm{u}^{i-1}_{\nu,\varepsilon})(\bm{u}^i_{\nu,\varepsilon}-\bm{u}^{i-1}_{\nu,\varepsilon}),\bm{\varphi})_{[H]^M}+\frac{\mu}{\tau}(\nabla(\bm{u}^i_{\nu,\varepsilon}-\bm{u}^{i-1}_{\nu,\varepsilon}),\nabla\bm{\varphi})_{[H]^{M\times N}}
    \\
    &\quad+(\alpha(\bm{u}^i_{\nu,\varepsilon}) B^*(\bm{u}^i_{\nu,\varepsilon}) \nabla\gamma_\varepsilon({B}(\bm{u}^i_{\nu,\varepsilon})\nabla\bm{u}^i_{\nu,\varepsilon}),\nabla\bm{\varphi})_{[H]^{M\times N}}
    +(\nabla_{\bm{u}} G(x,\bm{u}^i_{\nu,\varepsilon}),\bm{\varphi})_{[H]^M}
    \\
    &\quad +\nu\int_\Omega \nabla \Upsilon_p(\nabla \bm{u}^i_{\nu,\varepsilon}):\nabla\bm{\varphi}\,dx
    +([\nabla\alpha](\bm{u}^{i-1}_{\nu,\varepsilon})\gamma_\varepsilon({B}(\bm{u}^{i-1}_{\nu,\varepsilon})\nabla\bm{u}^{i-1}_{\nu,\varepsilon}),\bm{\varphi})_{[H]^M}
    \\
    &\quad+(\alpha(\bm{u}^{i-1}_{\nu,\varepsilon})\nabla\gamma_\varepsilon({B}(\bm{u}^{i-1}_{\nu,\varepsilon})\nabla\bm{u}^{i-1}_{\nu,\varepsilon}):[\nabla{B}](\bm{u}^{i-1}_{\nu,\varepsilon})\nabla\bm{u}^{i-1}_{\nu,\varepsilon},\bm{\varphi})_{[H]^M}=0,\label{3TimeDis-02}
    \\
    &\hspace{15ex}\mbox{ for all } \bm{\varphi}\in [W^{1,p}(\Omega)]^M, \mbox{ and }i=1,\dots,m.
    \end{align}
    Moreover, $\{\bm{u}^i_{\nu,\varepsilon}\}$ fulfill the following energy-inequality:
    \begin{gather}
    \frac{C_A}{4\tau}|\bm{u}^{i}_{\nu,\varepsilon}-\bm{u}^{i-1}_{\nu,\varepsilon}|^2_{[H]^M}+\frac{\mu}{2\tau}|\nabla(\bm{u}^{i}_{\nu,\varepsilon}-\bm{u}^{i-1}_{\nu,\varepsilon})|^2_{[H]^{M\times N}} +E_{\nu,\varepsilon}(\bm{u}^{i}_{\nu,\varepsilon}) \leq E_{\nu,\varepsilon}(\bm{u}^{i-1}_{\nu,\varepsilon}),\label{ukai_ene}
    \\
    \mbox{ for all } i=1,\dots,m.
    \end{gather}
  Additionally, taking into account Notation \ref{deftseq} and Example \ref{ex1}, we observe that
  \begin{align}
    &\nabla\gamma_{\varepsilon}(B( [\underline{\bm{u}}_{\nu,\varepsilon}]_\tau(t))\nabla [\underline{\bm{u}}_{\nu,\varepsilon}]_\tau(t))\in\partial\Phi_{\varepsilon}(B( [\underline{\bm{u}}_{\nu,\varepsilon}]_\tau(t))\nabla [\underline{\bm{u}}_{\nu,\varepsilon}]_\tau(t))\mbox{ in }[H]^{M\times N},
    \label{subdig1}
    \\
    &\hspace{12ex}\mbox{ for all } \nu\in (0,\nu_0),~\varepsilon\in(0,1), \mbox{ and a.e. }t\in(0,T),
    \nonumber
  \end{align}
  and hence, 
  \begin{gather}
    \nabla\gamma_{\varepsilon}(B( [\underline{\bm{u}}_{\nu,\varepsilon}]_\tau)\nabla [\underline{\bm{u}}_{\nu,\varepsilon}]_\tau)\in\partial\widehat{\Phi}^I_{\varepsilon}(B( [\underline{\bm{u}}_{\nu,\varepsilon}]_\tau)\nabla [\underline{\bm{u}}_{\nu,\varepsilon}]_\tau)\mbox{ in }L^2(I;[H]^{M\times N}),
    \label{subdig2}
    \\
    \mbox{ for all } \nu\in (0,\nu_0),~\varepsilon\in(0,1), \mbox{ and any open interval }I\subset(0,T),
    \nonumber
  \end{gather}\noeqref{ukai_ene,subdig1}
  where  $[\bm{u}_{\nu,\varepsilon}]_\tau$, $ [\overline{\bm{u}}_{\nu,\varepsilon}]_\tau $, and $ [\underline{\bm{u}}_{\nu,\varepsilon}]_\tau $  are the time-interpolation of $ \{ \bm{u}_{\nu,\varepsilon}^i \}_{i = 1}^m $ as in Notation \ref{deftseq}, for $ \nu\in(0,\nu_0),~\varepsilon \in (0, 1) $ and $ \tau \in (0, \tau_{*}(\nu,\varepsilon)) $.

  \subsection{Case: Analysis in the Limit $\nu\downarrow0$}\label{subnu}

From \eqref{gamma_ep}--\eqref{subdig2} and (A5), we can see the following boundedness:
\begin{description}
    \item[(B-1)]$\{[\bm{u}_{\nu,\varepsilon}]_\tau~|~\tau\in(0,\tau_*(\nu,\varepsilon)),~\nu\in(0,\nu_0),~\varepsilon\in(0,1)\}$ is bounded in $L^\infty(0,T;[V]^M)$ and in $W^{1,2}(0,T;[V]^M)$,
    \item[(B-2)]$\{[\overline{\bm{u}}_{\nu,\varepsilon}]_\tau~|~\tau\in(0,\tau_*(\nu,\varepsilon)),~\nu\in(0,\nu_0),~\varepsilon\in(0,1)\}$ and $\{[\underline{\bm{u}}_{\nu,\varepsilon}]_{\tau}~|~\tau\in(0,\tau_*(\nu,\varepsilon)),~\nu\in(0,\nu_0),~\varepsilon\in(0,1)\}$ are bounded in $L^\infty(0,T;[V]^M)$,
    \item[(B-3)]$\{\nabla\gamma_{\varepsilon}(B( [\underline{\bm{u}}_{\nu,\varepsilon}]_\tau(t))\nabla [\underline{\bm{u}}_{\nu,\varepsilon}]_\tau(t))~|~\tau\in(0,\tau_*(\nu,\varepsilon)),~\nu\in(0,\nu_0),~\varepsilon\in(0,1)\}$ is bounded in $L^\infty(Q;\R^{M\times N})$,
    \item[(B-4)]The function of time $t\in[0,T]\mapsto E_{\nu,\varepsilon}([\overline{\bm{u}}_{\nu,\varepsilon}]_\tau(t))\in[0,\infty)$ and $t\in [0,T]\mapsto E_{\nu,\varepsilon}([\underline{\bm{u}}_{\nu,\varepsilon}]_\tau(t))\in[0,\infty)$ are nonincreasing for every $\nu\in (0,\nu_0)$, $0<\varepsilon<1$ and $0<\tau<\tau_*(\varepsilon)$. Moreover, $\{E_{\nu,\varepsilon}(\bm{u}_{\nu}^0)~|~\nu\in(0,\nu_0),~\varepsilon\in(0,1)\}$ is bounded, and hence, the class $\{E_{\nu,\varepsilon}([\overline{\bm{u}}_{\nu,\varepsilon}]_{\tau})~|~\tau\in(0,\tau_*(\nu,\varepsilon)),~\nu\in(0,\nu_0),~\varepsilon\in(0,1)\} $ and $\{E_{\nu,\varepsilon}([\underline{\bm{u}}_{\nu,\varepsilon}]_{\tau})~|~\tau\in(0,\tau_*(\nu,\varepsilon)),~\nu\in(0,\nu_0),~\varepsilon\in(0,1)\} $ are bounded in $BV(0,T)$,
    \item[(B-5)] $\{\nu\Upsilon_p(\nabla[\overline{\bm{u}}_{\nu,\varepsilon}]_\tau)~|~\tau\in(0,\tau_*(\nu,\varepsilon)),~\nu\in(0,\nu_0),~\varepsilon\in(0,1)\}$ and $\{\nu\Upsilon_p(\nabla[\underline{\bm{u}}_{\nu,\varepsilon}]_{\tau})~|~\tau\in(0,\tau_*(\nu,\varepsilon)),~\nu\in(0,\nu_0),~\varepsilon\in(0,1)\}$ are bounded in $L^\infty(0,T;L^1(\Omega))$.
  \end{description}
On account of (B-1)--(B-3), we can apply the compactness theory of Aubin's type \cite[Corollary 4]{MR0916688}, and find a sequence $\{\nu_n\}_{n\in\N}\subset (0,\nu_0)$ and $\{\varepsilon_n\}_{n\in\N}\subset(0,1)$ with $ \{ \tau_n \}_{n\in\N} \subset (0, 1) $, and a function $\bm{u}\in[\mathscr{H}]^M$ with $\bm{w}^*\in L^\infty(Q;\R^{M\times N})$ such that: 
\begin{gather*}
    \nu_n \downarrow0,~\varepsilon_n \downarrow 0, ~ \tau_n := \frac{1}{2} \bigl( \tau_*(\nu_n,\varepsilon_n) \wedge \nu_n \wedge \varepsilon_n \wedge 1 \bigr) \downarrow 0, \mbox{ as $ n \to \infty $,}
\end{gather*}
 \begin{align}
    & \bm{u}_n : =  [\overline{\bm{u}}_{\nu_n,\varepsilon_n}]_{\tau_n}\rightarrow \bm{u}\mbox{ in }C([0,T];[H]^M),\mbox{ weakly in } W^{ 1 , 2 }( 0,T ; [V]^M),
    \nonumber 
    \\
      &\hspace*{10ex}\mbox{ weakly-}* \mbox{ in } L^\infty( 0,T ; [V]^M) , \mbox{ as $ n \to \infty $,} 
    \label{conv_1}
  \end{align}
  and
  \begin{align}
      &\nabla\gamma_{\varepsilon_n}(B( [\underline{\bm{u}}_{\nu_n,\varepsilon_n}]_{\tau_n}(t))\nabla [\underline{\bm{u}}_{\nu_n,\varepsilon_n}]_{\tau_n})~\rightarrow\bm{w}^*\mbox{ weakly-}*\mbox{in }L^\infty(Q;\R^{M\times N}), \mbox{ as $ n \to \infty $,}
    \label{conv_2}
  \end{align}
  in particular, from Lemma \ref{lem002} 
  \begin{align}\label{conv_3}
    \bm{u}(0)=\lim_{n\rightarrow\infty}\bm{u}_n(0)=\lim_{n\rightarrow\infty}\bm{u}^0_{\nu_n}=\bm{u}_0\mbox{ in }[H]^M.
  \end{align}
    Here, since
  \begin{align*}
    &\left\{
    \begin{aligned}
      &\max\{|([\overline{\bm{u}}_{\nu,\varepsilon}]_{\tau}-\bm{u}_{n})(t)|_{[V]^M},|([\underline{\bm{u}}_{\nu,\varepsilon}]_{\tau}-\bm{u}_{n})(t)|_{[V]^M}\} 
      \\
      &\quad\leq \int_{t_{i-1}}^{t_i}|\partial_t\bm{u}_{n}(t)|_{[V]^M}\,dt\leq \tau^{\frac{1}{2}}|\partial_t\bm{u}_{n}|_{[\mathscr{V}]^M},
    \end{aligned}
    \right.
    \\
    &\mbox{ for }t\in[t_{i-1},t_i),~i=1,\dots,m,~\tau\in(0,\tau_*(\nu,\varepsilon)),~\nu\in(0,\nu_0),\mbox{ and }~\varepsilon\in(0,1),
    \end{align*}
  one can also see from \eqref{conv_1} that:
  \begin{align}
    &\overline{\bm{u}}_{n}:=[\overline{\bm{u}}_{\nu_n,\varepsilon_n}]_{\tau_n}\rightarrow \bm{u},~\underline{\bm{u}}_{n}:=[\underline{\bm{u}}_{\nu_n,\varepsilon_n}]_{\tau_n}\rightarrow \bm{u} \mbox{ in }L^\infty(0,T;[H]^M),
    \nonumber
    \\
    &\hspace*{10ex}\mbox{ weakly-}* \mbox{ in }L^\infty(0,T;[V]^M),
    \label{conv_4}
  \end{align}
and in particular,
\begin{align}\label{conv_5}
    &\overline{\bm{u}}_{n}(t)\rightarrow \bm{u}(t), \underline{\bm{u}}_{n}(t)\rightarrow \bm{u}(t)\mbox{ in }[H]^M \mbox{ and weakly in }[V]^M,
    \\
    &\hspace{15ex}\mbox{ as } n\rightarrow\infty,\mbox{ for any }t\in(0,T).
    \nonumber
\end{align}
Moreover, \eqref{constc_E} and (B-4) enable us to see
\begin{gather}\label{conv_6}
    \bigl| E_{\nu_n,\varepsilon_n}(\overline{\bm{u}}_n) -E_{\nu_n,\varepsilon_n}(\underline{\bm{u}}_n) \bigr|_{L^1(0, T)} \leq 2c_E \tau_n \to 0, \mbox{ as $ n \to \infty $}.
\end{gather}
So, applying Helly's selection theorem \cite[Chapter 7, p.167]{rudin1976principles}, we will find a bounded and nonincreasing function $\mathcal{J}_*:[0,T]\mapsto[0,\infty)$, such that 
\begin{align}
    &E_{\nu_n,\varepsilon_n}(\overline{\bm{u}}_{n})\rightarrow\mathcal{J}_* \mbox{ and } E_{\nu_n,\varepsilon_n}(\underline{\bm{u}}_{n})\rightarrow\mathcal{J}_* \nonumber
  \\
  & \qquad \mbox{ weakly-}*\mbox{ in }BV(0,T),\mbox{ and }\mbox{weakly-}* \mbox{ in }L^\infty(0,T),
  \\
    &E_{\nu_n,\varepsilon_n} (\overline{\bm{u}}_{n}(t)) \rightarrow \mathcal{J}_*(t) \mbox{ and } E_{\nu_n,\varepsilon_n} (\underline{\bm{u}}_{n}(t)) \rightarrow \mathcal{J}_*(t), \mbox{ for any }t\in[0,T], \label{conv_7}
\end{align}
as $ n \to \infty $, by taking a subsequence if necessary.

Now, let us show the limit function $\bm{u}$ is a solution to the system (S)$_\nu$. (S2) can be checked by \eqref{conv_3}. Next, let us show $\bm{u}$ satisfies the variational inequalities (S1). Let us take any $ t \in (0, T] $. Then, from \eqref{3TimeDis-02}, the sequences as in \eqref{conv_1}--\eqref{conv_4} satisfy the following inequality:
      \begin{align}
    &\int^t_0(A(\underline{\bm{u}}_n(\sigma))\partial_t\bm{u}_n,(\overline{\bm{u}}_n-\bm{\omega})(\sigma))_{[H]^M}\,d\sigma
    \\
    &\quad+\mu\int^t_0(\nabla\partial_t\bm{u}_n(\sigma),\nabla(\overline{\bm{u}}_n-\bm{\omega})(\sigma))_{[H]^{M\times N}}\,d\sigma
    \\
    &\quad+\int^t_0(\nabla_{\bm{u}} G(x,\overline{\bm{u}}_n(\sigma)),(\overline{\bm{u}}_n-\bm{\omega})(\sigma))_{[H]^M}\,d\sigma+\nu_n\int^t_0\int_\Omega \Upsilon_p(\nabla \overline{\bm{u}}_n(\sigma))\,dxd\sigma
    \\
    &\quad +\int^t_0([\nabla\alpha](\underline{\bm{u}}_n(\sigma))\gamma_{\varepsilon_n}(\underline{\bm{u}}_n(\sigma))\nabla\underline{\bm{u}}_n(\sigma),(\overline{\bm{u}}_n-\bm{\omega})(\sigma))_{[H]^M}\,d\sigma
    \\
    &\quad+\int^t_0(\alpha(\underline{\bm{u}}_n(\sigma))\nabla\gamma_{\varepsilon_n}(\nabla\underline{\bm{u}}_n(\sigma){B}_0(\underline{\bm{u}}_n(\sigma)))
    \\
    &\qquad:\nabla\underline{\bm{u}}_n(\sigma)[\nabla{B}_0](\underline{\bm{u}}_n(\sigma)),(\overline{\bm{u}}_n-\bm{\omega})(\sigma))_{[H]^M}\,d\sigma
    \\
    &\quad +\int^t_0\int_\Omega\alpha(\overline{\bm{u}}_n(\sigma))\gamma_{\varepsilon_n}(\nabla\overline{\bm{u}}_n(\sigma){B}_0(\overline{\bm{u}}_n(\sigma)))\,dxd\sigma
    \\
    &\leq \int^t_0\int_\Omega\alpha(\overline{\bm{u}}_n(\sigma))\gamma_{\varepsilon_n}(\nabla\bm{\omega}(\sigma){B}_0(\overline{\bm{u}}_n(\sigma)))\,dxd\sigma+\nu_n\int^t_0\int_\Omega \Upsilon_p(\nabla \bm{\omega}(\sigma))\,dxd\sigma,\label{conv_8}
    \\
    &\hspace{20ex}\mbox{ for all } \bm{\omega}\in L^2(0,T;[W^{1,p}(\Omega)]^M), \mbox{ and }n\in\N.
  \end{align}
  According to (A5) and (B-5), it is obvious that
  \begin{align}\label{conv_9}
    \liminf_{n\rightarrow\infty}\nu_n\int_0^t\int_\Omega  \Upsilon_p(\nabla\overline{\bm{u}}_n(\sigma))\,dxd\sigma\geq0.
  \end{align}
  Also, by using (A1), (A6), \eqref{gamma_ep} and \eqref{conv_4}, weakly lower-semicontinuity of $\gamma_\varepsilon$ and Fatou's lemma, one can see 
  \begin{align}
    &\liminf_{n\rightarrow\infty} \int^t_0\int_\Omega\alpha(\overline{\bm{u}}_n(\sigma))\gamma_{\varepsilon_n}(\nabla\overline{\bm{u}}_n(\sigma){B}_0(\overline{\bm{u}}_n(\sigma)))\,dxd\sigma
    \\
    &\geq -C_\gamma\|\nabla\alpha\|_{L^\infty}\|B_0\|_{L^\infty}\lim_{n\rightarrow\infty}\int^t_0\int_\Omega|\overline{\bm{u}}_n(\sigma)-{\bm{u}}(\sigma)|(|\nabla\overline{\bm{u}}_n(\sigma)|+1)\,dxd\sigma
    \\
    &\quad -\lim_{n\rightarrow\infty}\|\gamma_{\varepsilon_n}-\gamma\|_{L^\infty}\int^t_0\int_\Omega\alpha(\bm{u}(\sigma))\,dxd\sigma
    \\
    &\quad +\liminf_{n\rightarrow\infty}\int^t_0\int_\Omega\alpha(\bm{u}(\sigma))\gamma(\nabla\overline{\bm{u}}_n(\sigma){B}_0(\overline{\bm{u}}_n(\sigma)))\,dxd\sigma
    \\
    &\geq \int^t_0\int_\Omega\alpha(\bm{u}(\sigma))\gamma(\nabla{\bm{u}}(\sigma)){B}_0({\bm{u}}(\sigma))\,dxd\sigma.\label{conv_10}
  \end{align}\noeqref{conv_10}
  Furthermore, setting $m_t:=([\frac{t}{\tau}]+1)\land \frac{T}{\tau}$ with use of the integral part $[\frac{t}{\tau}]\in\mathbb{Z}$ of $\frac{t}{\tau}\in\R$, we obtain that  
  \begin{align*}
    &\int_0^t(\nabla\partial_t\bm{u}_n(\sigma),\nabla\overline{\bm{u}}_n(\sigma))_{[H]^{M\times N}}\,d\sigma
    \\
    &=\int_{0}^{t_{{m}_t}}(\nabla\partial_t\bm{u}_n(\sigma),\nabla\overline{\bm{u}}_n(\sigma))_{[H]^{M\times N}}\,d\sigma
    -\int_{t}^{t_{m_t}}(\nabla\partial_t\bm{u}_n(\sigma),\nabla\overline{\bm{u}}_n(\sigma))_{[H]^{M\times N}}\,d\sigma
    \\
    &\geq\sum^{m_t}_{i=1}\frac{1}{2}(|\nabla \bm{u}_n(t_i)|_{[H]^{M\times N}}^2-|\nabla \bm{u}_n(t_{i-1})|_{[H]^{M\times N}}^2)
    \\
    &\qquad-|\Omega|^{\frac{p-2}{2p}}|\nabla\overline{\bm{u}}_n|_{L^\infty(0,T;[L^p(\Omega)]^{M\times N})}\int_{t}^{t_{m_t}}|\nabla\partial_t\bm{u}_n(\sigma)|_{[H]^{M\times N}}\,d\sigma
    \\
    &\geq\frac{1}{2}(|\nabla \bm{u}_n(t_{m_t})|_{[H]^{M\times N}}^2-|\nabla \bm{u}_n(0)|_{[H]^{M\times N}}^2)
    \\
    &\quad-\tau_n^\frac{1}{2}|\Omega|^{\frac{p-2}{2p}}|\nabla\overline{\bm{u}}_n|_{L^\infty(0,T;[L^p(\Omega)]^{M\times N})}|\nabla\partial_t\bm{u}_n|_{[\mathscr{H}]^{M\times N}}
    \\
    &=\frac{1}{2}(|\nabla\overline{\bm{u}}_n(t)|^2_{[H]^{M\times N}}-|\nabla \bm{u}_0|_{[H]^{M\times N}}^2)
      \\
      & \quad -\tau_n^\frac{1}{2}|\Omega|^{\frac{p-2}{2p}}|\nabla\overline{\bm{u}}_n|_{L^\infty(0,T;[L^p(\Omega)]^{M\times N})}|\nabla\partial_t\bm{u}_n|_{[\mathscr{H}]^{M\times N}}.
  \end{align*}
  Hence, by \eqref{conv_5}, we obtain
  \begin{align}
    &\liminf_{n\rightarrow\infty}\int_0^t(\nabla\partial_t\bm{u}_n(\sigma),\nabla\overline{\bm{u}}_n(\sigma))_{[H]^{M\times N}}\,d\sigma
    \label{conv_11}
    \\
    &\quad\geq \frac{1}{2}(|\nabla{\bm{u}}(t)|^2_{[H]^{M\times N}}-|\nabla \bm{u}_0|_{[H]^{M\times N}}^2)=\int_{0}^{t}(\nabla\partial_t\bm{u}(\sigma),\nabla \bm{u}(\sigma))_{[H]^{M\times N}}\,d\sigma,
    \nonumber
  \end{align}
  Here, $\bm{\psi}_n$ be the sequence obtained by Lemma \ref{lem003} with $\bm{\psi}=\bm{u}$. substituting $\bm{\psi}_n$ for $\bm{\omega}$ in \eqref{conv_8}, we see that 
  \begin{align}
    &\limsup_{n\rightarrow\infty}\Big(\frac{\nu_n}{C_\Upsilon}\int_0^t|\nabla\overline{\bm{u}}_n(\sigma)|^p_{[L^p(\Omega)]^{M\times N}}\,d\sigma-\frac{\nu_n}{C_\Upsilon}|\Omega|T 
    \\
    &\qquad+\frac{\mu}{2}\big(|\nabla\overline{\bm{u}}_n(t)|^2_{[H]^{M\times N}}-|\nabla\bm{u}_0|^2_{[H]^{M\times N}}\big)
    \\
    &\qquad -\tau_n^\frac{1}{2}|\Omega|^{\frac{p-2}{2p}}|\nabla\overline{\bm{u}}_n|_{L^\infty(0,T;[L^p(\Omega)]^{M\times N})}|\nabla\partial_t\bm{u}_n|_{[\mathscr{H}]^{M\times N}}
    \\
    &\qquad+\int^t_0\int_\Omega\alpha(\overline{\bm{u}}_n(\sigma))\gamma_{\varepsilon_n}(\nabla\overline{\bm{u}}_n(\sigma){B}_0(\overline{\bm{u}}_n(\sigma)))\,dxd\sigma
    \Big)
    \\
    &\limsup_{n\rightarrow\infty}\left(\nu_n\int_0^t\int_\Omega\Upsilon_p(\nabla\overline{\bm{u}}_n(\sigma))\,dxd\sigma+\mu\int_0^t(\nabla\partial_t\bm{u}_n(\sigma),\nabla\overline{\bm{u}}_n(\sigma))_{[H]^{M\times N}}\,d\sigma\right.
    \\
    &\qquad\left.+\int^t_0\int_\Omega\alpha(\overline{\bm{u}}_n(\sigma))\gamma_{\varepsilon_n}(\nabla\overline{\bm{u}}_n(\sigma){B}_0(\overline{\bm{u}}_n(\sigma)))\,dxd\sigma
    \right)
  \\
    &\leq \lim_{n\rightarrow\infty}\nu_nC_\Upsilon\Big(\int_0^t|\nabla\bm{\psi}_n(\sigma)|^p_{L^p(\Omega)}\,d\sigma+|\Omega|T\Big)
    \\
    &\quad+\lim_{n\rightarrow\infty}\mu\int_0^t(\nabla\partial_t\bm{u}_n(\sigma),\nabla{\bm{\psi}}_n(\sigma))_{[H]^{M\times N}}\,d\sigma 
    \end{align}
  \begin{align}
    &\quad-\lim_{n\rightarrow\infty}\int_0^t(A(\underline{\bm{u}}_n(\sigma))\partial_t\bm{u}_n(\sigma)+\nabla_{\bm{u}} G(x,\overline{\bm{u}}_n(\sigma)),(\overline{\bm{u}}_n-\bm{\psi}_n)(\sigma))_{[H]^M}\,d\sigma
    \\
    &\quad-\lim_{n\rightarrow\infty}\int^t_0([\nabla\alpha](\underline{\bm{u}}_n(\sigma))\gamma_{\varepsilon_n}(\nabla\underline{\bm{u}}_n(\sigma)B_0(\underline{\bm{u}}_n(\sigma))),(\overline{\bm{u}}_n-\bm{\psi}_n)(\sigma))_{[H]^M}\,d\sigma\label{conv_12}
    \\
    &\quad-\lim_{n\rightarrow\infty}\int^t_0(\alpha(\underline{\bm{u}}_n(\sigma))\nabla\gamma_{\varepsilon_n}(\nabla\underline{\bm{u}}_n(\sigma){B}_0(\underline{\bm{u}}_n(\sigma))):
    \\
    &\qquad :\nabla\underline{\bm{u}}_n(\sigma)[\nabla{B}_0](\underline{\bm{u}}_n(\sigma)),(\overline{\bm{u}}_n-\bm{\psi}_n)(\sigma))_{[H]^M}d\sigma
    \\
    &=0+\mu\int_0^t(\nabla\partial_t\bm{u}(\sigma),\nabla{\bm{u}}(\sigma))_{[H]^{M\times N}}\,d\sigma+\int^t_0\int_\Omega\alpha({\bm{u}}(\sigma))\gamma(\nabla{\bm{u}}(\sigma){B}_0({\bm{u}}(\sigma)))\,dxd\sigma
    \\
    &=0+\frac{\mu}{2}\left(|\nabla{\bm{u}}(t)|^2_{[H]^{M\times N}}-|\nabla\bm{u}_0|^2_{[H]^{M\times N}}\right)+\int^t_0\int_\Omega\alpha({\bm{u}}(\sigma))\gamma(\nabla{\bm{u}}(\sigma){B}_0({\bm{u}}(\sigma)))\,dxd\sigma
  \end{align}
  From (Fact 1) and \eqref{conv_9}--\eqref{conv_11}, we obtain the following convergences as $n\rightarrow\infty$:
  \begin{align}
    &\nu_n\int_0^t\int_\Omega\Upsilon_p(\nabla\overline{\bm{u}}_n(\sigma))\,dxd\sigma\rightarrow0,\label{conv_13}
    \\
    &\int_0^t(\nabla\partial_t\bm{u}_n(\sigma),\nabla\overline{\bm{u}}_n(\sigma))_{[H]^{M\times N}}\,d\sigma\rightarrow\int_0^t(\nabla\partial_t\bm{u}(\sigma),\nabla{\bm{u}}(\sigma))_{[H]^{M\times N}}\,d\sigma,\label{conv_14}
    \\
    &|\nabla\overline{\bm{u}}_n(t)|^2_{[H]^{M\times N}}\rightarrow|\nabla\bm{u}(t)|^2_{[H]^{M\times N}}, \mbox{ for any }t\in[0,T].\label{conv_15}
  \end{align}
  Additionally, thanks to \eqref{gamma_ep}, \eqref{constc_E}, (B-4), and uniform convexity of $L^2$-based topologies, we can derive that, for any $t\in[0,T]$, as $n\rightarrow\infty$:
  \begin{align}
    &\bullet\overline{\bm{u}}_n(t)\rightarrow \bm{u}(t)\mbox{ and }\underline{\bm{u}}_n(t)\rightarrow \bm{u}(t)\mbox{ in }[V]^M, \mbox{ with }
    \\
    &\qquad |\nabla(\overline{\bm{u}}_n-\underline{\bm{u}}_n)(t)|_{[H]^{M\times N}}\leq \tau_n^\frac{1}{2}|\nabla\partial_t\bm{u}_n|_{[\mathscr{H}]^M}\rightarrow0,\label{conv_16}
    \\[1.0ex]
    &\bullet|\alpha(\overline{\bm{u}}_n(t))\gamma_{\varepsilon_n}(\nabla\overline{\bm{u}}_n(t){B}_0(\overline{\bm{u}}_n(t)))-\alpha({\bm{u}}(t))\gamma(\nabla{\bm{u}}(t){B}_0({\bm{u}}(t)))|_{L^1(\Omega)}
    \\
    &\qquad \leq C_\gamma\|\nabla\alpha\|_{L^\infty}\|B_0\|_{L^\infty}|(\overline{\bm{u}}_n-\bm{u})(t)|_{[H]^M}(|\nabla\overline{\bm{u}}_n(t)|_{[H]^{M\times N}}+|\Omega|^\frac{1}{2})
    \\
    &\qquad +\|\alpha\|_{L^\infty}\|B_0\|_{L^\infty}\|\gamma_{\varepsilon_n}-\gamma\|_{L^\infty}|\nabla\overline{\bm{u}}_n(t)|_{[H]^{M\times N}}|\Omega|^\frac{1}{2}
    \\
    &\qquad +\|\alpha\|_{L^\infty}\|\nabla B_0\|_{L^\infty}\|\nabla\gamma\|_{L^\infty}|(\overline{\bm{u}}_n-\bm{u})(t)|_{[H]^M}|\nabla\overline{\bm{u}}_n(t)|_{[H]^{M\times N}}
    \\
    &\qquad +\|\alpha\|_{L^\infty}\|B_0\|_{L^\infty}\|\nabla\gamma\|_{L^\infty}|\nabla(\overline{\bm{u}}_n-\bm{u})(t)|_{[H]^{M\times N}}\rightarrow0,\label{conv_17}
    \\
    &\mbox{ as well as }|\alpha(\underline{\bm{u}}_n(t))\gamma_{\varepsilon_n}(\nabla\underline{\bm{u}}_n(t){B}_0(\underline{\bm{u}}_n(t)))-\alpha({\bm{u}}(t))\gamma(\nabla{\bm{u}}(t){B}_0({\bm{u}}(t)))|_{L^1(\Omega)}\rightarrow0.
  \end{align}
  
  \eqref{conv_5}--\eqref{conv_7}, \eqref{conv_16} and \eqref{conv_17} imply 
  \begin{gather}\label{convJ}
     \begin{cases}
         E_{\nu_n,\varepsilon_n}(\overline{\bm{u}}_n(t)) \to \mathcal{J}_*(t) = E(\bm{u}(t)),
         \\
         E_{\nu_n,\varepsilon_n}(\underline{\bm{u}}_n(t)) \to \mathcal{J}_*(t) = E(\bm{u}(t)), 
     \end{cases}
     \mbox{a.e. $ t \in (0, T) $, as $ n \to \infty $.}
 \end{gather}
 Now, we take $\omega=\bm{\varphi}$ in $[V]^M$ in \eqref{conv_8}, and consider to pass to the limit $n\rightarrow\infty$. Then, in light of \eqref{conv_1}, \eqref{conv_4}, \eqref{conv_13} and \eqref{conv_14}, and the Lebesgue dominated convergence theorem, we obtain
    \begin{gather}\label{conv_18}
    \int_I(A(\bm{u}(t))\partial_t\bm{u}(t),\bm{u}(t)-\bm{\varphi})_{[H]^M}\,dt+\int_I\mu(\nabla\partial_t\bm{u}(t),\nabla(\bm{u}(t)-\bm{\varphi}))_{[H]^{M\times N}}\,dt
    \\
    +\int_I(\nabla_{\bm{u}} G(x,\bm{u}(t))+[\nabla\alpha](\bm{u}(t))\gamma({B}(\bm{u}(t))\nabla\bm{u}(t)),\bm{u}(t)-\bm{\varphi})_{[H]^M}\,dt
    \\
    +\int_I(\alpha(\bm{u}(t))\bm{w}^*(t):[\nabla{B}](\bm{u}(t))\nabla\bm{u}(t),\bm{u}(t)-\bm{\varphi})_{[H]^M}\,dt
    \\
    +\int_I\int_{\Omega}\alpha(\bm{u}(t))\gamma(B(\bm{u}(t))\nabla \bm{u}(t))\,dx\,dt\leq\int_I\int_{\Omega}\alpha(\bm{u}(t))\gamma (B(\bm{u}(t))\nabla\bm{\varphi})\,dx\,dt,
    \end{gather}
    for any open interval $I\subset(0,T)$. Moreover, by involving \eqref{subdig1}, \eqref{subdig2}, \eqref{conv_2}, together with Example \ref{ex2}, and (Fact 3) in Remark \ref{rem2}, one can conclude that 
    \begin{equation*}
     \bm{w}^*\in\partial\widehat{\Phi}_0^I(B(\bm{u})\nabla \bm{u})\mbox{ in }L^2(I;[H]^{M\times N}),
    \end{equation*}
 and hence, 
 \begin{equation}\label{subdig3}
  \begin{aligned}
    &\bm{w}^*\in\partial\Phi_0(B(\bm{u})\nabla \bm{u}) \mbox{ in } [H]^{M\times N}, \mbox{ for a.e. }t\in(0,T),\mbox{ and }
    \\
    &\bm{w}^*\in\partial\gamma(B(\bm{u})\nabla \bm{u}) \mbox{ in }\R^{M\times N}, \mbox{ a.e. in } Q.
  \end{aligned}
 \end{equation}
\eqref{conv_18}--\eqref{subdig3} imply that the limit $\bm{u}$ satisfies (S0) and (S1).

Next, we consider the energy inequality \eqref{ene-inq1}. From \eqref{ukai_ene}, it is derived that 
\begin{align}
      &\int_{t_{i-1}}^{t_i}\biggl(\frac{C_A}{4}|\partial_t\bm{u}_n(\sigma)|^2_{[H]^M}+\frac{\mu}{2}|\nabla\partial_t\bm{u}_n(\sigma)|^2_{[H]^{M\times N}}\biggr)\,d\sigma+E_{\nu_n,\varepsilon_n}(\overline{\bm{u}}_n(t))
      \nonumber
      \\
      &\quad\leq E_{\nu_n,\varepsilon_n}(\underline{\bm{u}}_n(t)),\mbox{ for }t\in[t_{i-1},t_i),~i=1,2,\dots,{\ts\frac{T}{\tau_n}},\mbox{ and }n\in\N.
      \label{energy1}
\end{align}
Here, setting $m^s:=[\frac{s}{\tau}]$ and $m_t:=\bigl([\frac{t}{\tau}]+1\bigr)\land \frac{T}{\tau}$ for $0\leq s< t\leq T$, and summing both sides of \eqref{energy1} for $i=m^s+1, m^s+2,\dots,m_t$, we obtain that
\begin{align}
      &\frac{C_A}{4}\int_s^t \left(|\partial_t\bm{u}_n(\sigma)|^2_{[H]^M}+2\mu|\nabla\partial_t\bm{u}_n(\sigma)|^2_{[H]^{M\times N}}\right)\,d\sigma+E_{\nu_n,\varepsilon_n}(\overline{\bm{u}}_n(t))
      \nonumber
      \\
      &\leq \frac{C_A}{4}\int_{m^s\tau_n}^{m_t\tau_n} \left(|\partial_t\bm{u}_n(\sigma)|^2_{[H]^M}+2\mu|\nabla\partial_t\bm{u}_n(\sigma)|^2_{[H]^{M\times N}}\right)\,d\sigma+E_{\nu_n,\varepsilon_n}(\overline{\bm{u}}_n(t))
      \nonumber
      \\
      &\leq E_{\nu_n,\varepsilon_n}(\underline{\bm{u}}_n(s)), \mbox{ for }s,t\in[0,T];s\leq t,\mbox{ and }n\in\N.\label{energy3}
    \end{align}
    Now, taking the limit $n\rightarrow\infty$ and using \eqref{conv_1} \eqref{conv_2}, \eqref{convJ} and \eqref{energy3}, we see that 
\begin{gather}
    \frac{C_A}{4}\int_{s}^{t}|\partial_t\bm{u}(\sigma)|^2_{[H]^M}\,d\sigma+\frac{\mu}{2}\int_{s}^{t}|\nabla\partial_t\bm{u}(\sigma)|^2_{[H]^{M\times N}}\,d\sigma+E(\bm{u}(t)) \leq E(\bm{u}(s)),
    \label{energy4}
    \\
    \mbox{ for a.e. $ s \in [0, T) $ including $s = 0$, and a.e. $ t \in (s, T) $.}
\end{gather}
In addition, the condition ``a.e. $ t \in (s, T) $'' in \eqref{energy4} can be strengthened to ``for any $ t \in [s, T] $''. Indeed, by taking a sequence $\{t_n\}_{n\in\N}\subset (t,T)$ with $t_n \rightarrow t$, we observe that 
\begin{gather}
  \frac{C_A}{4}\int_{s}^{t_n}|\partial_t\bm{u}(\sigma)|^2_{[H]^M}\,d\sigma+\frac{\mu}{2}\int_{s}^{t_n}|\nabla\partial_t\bm{u}(\sigma)|^2_{[H]^{M\times N}}\,d\sigma+E(\bm{u}(t_n)) \leq E(\bm{u}(s)),
  \label{energy5}
  \\
  \mbox{ for all $ n\in\N $.}
\end{gather}
Having in mind the lower semi-continuity of $E(\bm{u})$ on $[H]^M$ and the convergence $\bm{u}(t_n) \rightarrow \bm{u}(t)$ in $[H]^M$, by taking the lower limit of both sides of \eqref{energy5} yields that 
\begin{align}
  &\frac{C_A}{4}\int_{s}^{t}|\partial_t\bm{u}(\sigma)|^2_{[H]^M}\,d\sigma+\frac{\mu}{2}\int_{s}^{t}|\nabla\partial_t\bm{u}(\sigma)|^2_{[H]^{M\times N}}\,d\sigma +E(\bm{u}(t)) 
    \\
  \leq& \liminf_{n\rightarrow\infty}\Biggl(\frac{C_A}{4}\int_{s}^{t_n}|\partial_t\bm{u}(\sigma)|^2_{[H]^M}\,d\sigma+\frac{\mu}{2}\int_{s}^{t_n}|\nabla\partial_t\bm{u}(\sigma)|^2_{[H]^{M\times N}}\,d\sigma+E(\bm{u}(t_n)) \Biggr)
  \\
  \leq& E(\bm{u}(s)). \label{energy6}
\end{align}
Accordingly, the energy-inequality \eqref{ene-inq1} follows directly from \eqref{energy6}.

\subsection{Case: $\nu>0$ is a fixed positive constant}
From \eqref{gamma_ep}--\eqref{subdig2} and (A5), we can see the following boundedness:
\begin{description}
    \item[(C-1)]$\{[\bm{u}_{\nu,\varepsilon}]_\tau~|~\tau\in(0,\tau_*(\nu,\varepsilon)),~\varepsilon\in(0,1)\}$ is bounded in $L^\infty(0,T;[W^{1,p}(\Omega)]^M)$ and in
    \\
    $W^{1,2}(0,T;[V]^M)$,
    \item[(C-2)]$\{[\overline{\bm{u}}_{\nu,\varepsilon}]_\tau~|~\tau\in(0,\tau_*(\nu,\varepsilon)),~\varepsilon\in(0,1)\}$ and $\{[\underline{\bm{u}}_{\nu,\varepsilon}]_{\tau}~|~\tau\in(0,\tau_*(\nu,\varepsilon)),~\varepsilon\in(0,1)\}$ are bounded in $L^\infty(0,T;[W^{1,p}(\Omega)]^M)$,
    \item[(C-3)]$\{\nabla\gamma_{\varepsilon}(B( [\underline{\bm{u}}_{\nu,\varepsilon}]_\tau(t))\nabla [\underline{\bm{u}}_{\nu,\varepsilon}]_\tau(t))~|~\tau\in(0,\tau_*(\nu,\varepsilon)),~\varepsilon\in(0,1)\}$ is bounded in 
    \\
    $L^\infty(Q;\R^{M\times N})$,
    \item[(C-4)]The function of time $t\in[0,T]\mapsto E_{\nu,\varepsilon}([\overline{\bm{u}}_{\nu,\varepsilon}]_\tau(t))\in[0,\infty)$ and $t\in [0,T]\mapsto E_{\nu,\varepsilon}([\underline{\bm{u}}_{\nu,\varepsilon}]_\tau(t))\in[0,\infty)$ are nonincreasing for every $0<\varepsilon<1$ and $0<\tau<\tau_*(\nu,\varepsilon)$. Moreover, $\{E_{\nu,\varepsilon}(\bm{u}_{\nu,\varepsilon}^0)~|~\varepsilon\in(0,1)\}$ is bounded, and hence, the class $\{E_{\nu,\varepsilon}([\overline{\bm{u}}_{\nu,\varepsilon}]_{\tau})~|~\tau\in(0,\tau_*(\nu,\varepsilon)),~\varepsilon\in(0,1)\} $ and $\{E_{\nu,\varepsilon}([\underline{\bm{u}}_{\nu,\varepsilon}]_{\tau})~|~\tau\in(0,\tau_*(\nu,\varepsilon)),~\varepsilon\in(0,1)\} $ are bounded in $BV(0,T)$.
  \end{description}
  Therefore, applying the compactness theory of Aubin's type \cite[Corollary 4]{MR0916688}, and find a sequence $\{\varepsilon_n\}_{n\in\N}\subset(0,1)$ with $ \{ \tau_n \}_{n\in\N} \subset (0, 1) $, and a function $\bm{u}_\nu\in[\mathscr{H}]^M$ with $\bm{w}^*\in L^\infty(Q;\R^{M\times N})$ such that: 
\begin{gather*}
    \varepsilon_n \downarrow 0, ~ \tau_n := \frac{1}{2} \bigl( \tau_*(\nu,\varepsilon_n) \wedge \varepsilon_n \wedge 1 \bigr) \downarrow 0, \mbox{ as $ n \to \infty $,}
\end{gather*}
and 
 \begin{equation}\label{conv_100}
  \left\{
    \begin{aligned}
    & \bm{u}_n : =  [\overline{\bm{u}}_{\nu,\varepsilon_n}]_{\tau_n}\rightarrow \bm{u}_\nu\mbox{ in }C([0,T];[H]^M),
    \\
    &\mbox{weakly in } W^{ 1 , 2 }( 0,T ; [V]^M),\mbox{ weakly-}* \mbox{ in } L^\infty( 0,T ; [W^{1,p}(\Omega)]^M),
    \\
    &\nabla\gamma_{\varepsilon_n}(B( [\underline{\bm{u}}_{\nu,\varepsilon_n}]_{\tau_n}(t))\nabla [\underline{\bm{u}}_{\nu,\varepsilon_n}]_{\tau_n})~\rightarrow\bm{w}^*\mbox{ weakly-}*\mbox{ in }L^\infty(Q;\R^{M\times N}),
    \\
    &\overline{\bm{u}}_{n}:=[\overline{\bm{u}}_{\nu,\varepsilon_n}]_{\tau_n}\rightarrow \bm{u}_\nu,~\underline{\bm{u}}_{n}:=[\underline{\bm{u}}_{\nu,\varepsilon_n}]_{\tau_n}\rightarrow \bm{u}_\nu \mbox{ in }L^\infty(0,T;[H]^M),
    \\
    &\mbox{weakly-}* \mbox{ in }L^\infty(0,T;[W^{1,p}(\Omega)]^M),
    \end{aligned}\right.
  \end{equation}
  In particular, since Sobolev embedding $W^{1,p}(\Omega)\subset C(\overline{\Omega})$ is compact for $p>N$, the sequence $\bm{u}_n$ converges $\bm{u}$ in $C(\overline{Q})$.
  Moreover, \eqref{gammaconv}, \eqref{constc_E} and (C-4) enable us to see
\begin{gather}\label{conv_101}
    \bigl| E_{\nu,\varepsilon_n}(\overline{\bm{u}}_n) -E_{\nu,\varepsilon_n}(\underline{\bm{u}}_n) \bigr|_{L^1(0, T)} \leq 2c_E \tau_n \to 0, \mbox{ as $ n \to \infty $}.
\end{gather}
So, Using Helly's selection theorem \cite[Chapter 7, p.167]{rudin1976principles}, one can construct a bounded and nonincreasing function $\mathcal{\widetilde{J}}_*:[0,T]\mapsto[0,\infty)$, such that 
\begin{align}
    &E_{\nu,\varepsilon_n}(\overline{\bm{u}}_{n})\rightarrow\mathcal{\widetilde{J}}_* \mbox{ and } E_{\nu,\varepsilon_n}(\underline{\bm{u}}_{n})\rightarrow\mathcal{\widetilde{J}}_* \nonumber
  \\
  & \qquad \mbox{ weakly-}*\mbox{ in }BV(0,T),\mbox{ and }\mbox{weakly-}* \mbox{ in }L^\infty(0,T),
  \\
    &E_{\nu,\varepsilon_n} (\overline{\bm{u}}_{n}(t)) \rightarrow \mathcal{\widetilde{J}}_*(t) \mbox{ and } E_{\nu,\varepsilon_n} (\underline{\bm{u}}_{n}(t)) \rightarrow \mathcal{\widetilde{J}}_*(t), \mbox{ for any }t\in[0,T], \label{conv_102}
\end{align}
as $ n \to \infty $, by taking a subsequence if necessary.

We now verify that $\bm{u}_\nu$ is a solution to the system (S)$_\nu$. (S2) follows directly from \eqref{conv_100}. Next, let us show $\bm{u}_\nu$ satisfies the variational inequalities (S1). Let us take any $ t \in [0, T] $. Then, by \eqref{3TimeDis-02}, the sequences appearing in \eqref{conv_100} obey the following inequality:
  \begin{align}
    &\int^t_0(A(\underline{\bm{u}}_n(\sigma))\partial_t\bm{u}_n,(\overline{\bm{u}}_n-\bm{\omega})(\sigma))_{[H]^M}\,d\sigma+\mu\int^t_0(\nabla\partial_t\bm{u}_n(\sigma),\nabla(\overline{\bm{u}}_n-\bm{\omega}))_{[H]^{M\times N}}\,d\sigma
    \\
    &\quad+\int^t_0(\nabla_{\bm{u}} G(x,\overline{\bm{u}}_n(\sigma)),(\overline{\bm{u}}_n-\bm{\omega})(\sigma))_{[H]^M}\,d\sigma+\nu_n\int^t_0\int_\Omega \Upsilon_p(\nabla \overline{\bm{u}}_n(\sigma))\,dxd\sigma
    \\
    &\quad +\int^t_0([\nabla\alpha](\underline{\bm{u}}_n(\sigma))\gamma_{\varepsilon_n}(\nabla\underline{\bm{u}}_n(\sigma)B_0(\underline{\bm{u}}_n(\sigma))),(\overline{\bm{u}}_n-\bm{\omega})(\sigma))_{[H]^M}\,d\sigma
    \\
    &\quad+\int^t_0(\alpha(\underline{\bm{u}}_n(\sigma))\nabla\gamma_{\varepsilon_n}(\nabla\underline{\bm{u}}_n(\sigma){B}_0(\underline{\bm{u}}_n(\sigma))):\nabla\underline{\bm{u}}_n(\sigma)[\nabla{B}_0](\underline{\bm{u}}_n(\sigma)),(\overline{\bm{u}}_n-\bm{\omega})(\sigma))_{[H]^M}
    \\
    &\quad +\int^t_0\int_\Omega\alpha(\overline{\bm{u}}_n(\sigma))\gamma_{\varepsilon_n}(\nabla\overline{\bm{u}}_n(\sigma){B}_0(\overline{\bm{u}}_n(\sigma)))\,dxd\sigma
    \\
    &\leq \int^t_0\int_\Omega\alpha(\overline{\bm{u}}_n(\sigma))\gamma_{\varepsilon_n}(\nabla\bm{\omega}(\sigma){B}_0(\overline{\bm{u}}_n(\sigma)))\,dxd\sigma+\nu_n\int^t_0\int_\Omega \Upsilon_p(\nabla \bm{\omega}(\sigma))\,dxd\sigma,\label{conv_103}
    \\
    &\hspace{15ex}\mbox{ for all }\bm{\omega}\in L^2(0,T;[W^{1,p}(\Omega)]^M), \mbox{ and }i=1,\dots,m.
  \end{align}
  Here, we define the functional $\Pi$ on $[L^p(\Omega)]^{M\times N}$ as follows.
  \begin{align*}
    \Pi:W\in [L^p(\Omega)]^{M\times N}\mapsto \Pi(W):=\int_\Omega \Upsilon_p(W)\,dx \in[0,\infty).
  \end{align*}
  Then, by (A5), $\Pi$ is weakly lower semi-continuous on $[L^p(\Omega)]^{M\times N}$. In view of \eqref{conv_100} and Fatou's lemma, it follows that 
    \begin{align}\label{conv_1031}
    \liminf_{n\rightarrow\infty}\nu\int_0^t\int_\Omega  \Upsilon_p(\nabla\overline{\bm{u}}_n(\sigma))\,dxd\sigma\geq\nu\int_0^t\int_\Omega  \Upsilon_p(\nabla{\bm{u}}_\nu(\sigma))\,dxd\sigma.
  \end{align}

  Here, by putting $\bm{\omega}=\bm{u}_\nu$ in \eqref{conv_103} and using \eqref{conv_100}, we see that 
    \begin{align}
    &\limsup_{n\rightarrow\infty}\Big(\nu\int_0^t\int_\Omega\Upsilon_p(\nabla\overline{\bm{u}}_n(\sigma))\,dxd\sigma+\frac{\mu}{2}\big(|\nabla\overline{\bm{u}}_n(t)|^2_{[H]^{M\times N}}-|\nabla\bm{u}_0|^2_{[H]^{M\times N}}\big)
    \\
    &\quad -\tau_n^\frac{1}{2}|\Omega|^{\frac{p-2}{2p}}|\nabla\overline{\bm{u}}_n|_{L^\infty(0,T;[L^p(\Omega)]^{M\times N})}|\nabla\partial_t\bm{u}_n|_{[\mathscr{H}]^{M\times N}}
    \\
    &\quad+\int^t_0\int_\Omega\alpha(\overline{\bm{u}}_n(\sigma))\gamma_{\varepsilon_n}(\nabla\overline{\bm{u}}_n(\sigma){B}_0(\overline{\bm{u}}_n(\sigma)))\,dxd\sigma
    \Big)
    \\
    &\leq \limsup_{n\rightarrow\infty}\Big(\nu\int_0^t\int_\Omega\Upsilon_p(\nabla\overline{\bm{u}}_n(\sigma))\,dxd\sigma+\mu\int_0^t(\nabla\partial_t\bm{u}_n(\sigma),\nabla\overline{\bm{u}}_n(\sigma))_{[H]^{M\times N}}\,d\sigma
    \\
    &\quad+\int^t_0\int_\Omega\alpha(\overline{\bm{u}}_n(\sigma))\gamma_{\varepsilon_n}(\nabla\overline{\bm{u}}_n(\sigma){B}_0(\overline{\bm{u}}_n(\sigma)))\,dxd\sigma
    \Big)
    \\
    &\leq \lim_{n\rightarrow\infty}\nu\int_0^t\int_\Omega\Upsilon_p(\nabla{\bm{u}}_\nu(\sigma))\,dxd\sigma+\lim_{n\rightarrow\infty}\mu\int_0^t(\nabla\partial_t\bm{u}_n(\sigma),\nabla{\bm{u}}_\nu(\sigma))_{[H]^{M\times N}}\,d\sigma 
    \\
    &\quad+\lim_{n\rightarrow\infty}\int^t_0\int_\Omega\alpha(\overline{\bm{u}}_n(\sigma))\gamma_{\varepsilon_n}(\nabla{\bm{u}}_\nu(\sigma){B}_0(\overline{\bm{u}}_n(\sigma)))\,dxd\sigma
    \\
    &\quad-\lim_{n\rightarrow\infty}\int_0^t(A(\underline{\bm{u}}_n(\sigma))\partial_t\bm{u}_n(\sigma)+\nabla_{\bm{u}} G(x,\overline{\bm{u}}_n(\sigma)),(\overline{\bm{u}}_n-\bm{u}_\nu)(\sigma))_{[H]^M}\,d\sigma
    \\
    &\quad-\lim_{n\rightarrow\infty}\int^t_0([\nabla\alpha](\underline{\bm{u}}_n(\sigma))\gamma_{\varepsilon_n}(\nabla\underline{\bm{u}}_n(\sigma)B_0(\underline{\bm{u}}_n(\sigma))),(\overline{\bm{u}}_n-\bm{u}_\nu)(\sigma))_{[H]^M}\,d\sigma
    \\
    &\quad-\lim_{n\rightarrow\infty}\int^t_0(\alpha(\underline{\bm{u}}_n(\sigma))\nabla\gamma_{\varepsilon_n}(\nabla\underline{\bm{u}}_n(\sigma){B}_0(\underline{\bm{u}}_n(\sigma))):
    \\
    &\qquad\qquad:\nabla\underline{\bm{u}}_n(\sigma)[\nabla{B}_0](\underline{\bm{u}}_n(\sigma)),(\overline{\bm{u}}_n-\bm{u}_\nu)(\sigma))_{[H]^M}d\sigma
    \\
    &=\nu\int_0^t\int_\Omega\Upsilon_p(\nabla{\bm{u}}_\nu(\sigma))\,dxd\sigma+\mu\int_0^t(\nabla\partial_t\bm{u}_\nu(\sigma),\nabla{\bm{u}_\nu}(\sigma))_{[H]^{M\times N}}\,d\sigma
    \\
    &\quad+\int^t_0\int_\Omega\alpha({\bm{u}}_\nu(\sigma))\gamma(\nabla{\bm{u}_\nu}(\sigma){B}_0({\bm{u}_\nu}(\sigma)))\,dxd\sigma
    \\
    &=\nu\int_0^t\int_\Omega\Upsilon_p(\nabla{\bm{u}}_\nu(\sigma))\,dxd\sigma+\frac{\mu}{2}\big(|\nabla{\bm{u}_\nu}(t)|^2_{[H]^{M\times N}}-|\nabla\bm{u}_0|^2_{[H]^{M\times N}}\big)
    \\
    &\quad+\int^t_0\int_\Omega\alpha({\bm{u}_\nu}(\sigma))\gamma(\nabla{\bm{u}_\nu}(\sigma){B}_0({\bm{u}_\nu}(\sigma)))\,dxd\sigma
  \end{align}
  Combining (Fact 1) in Section \ref{secpre} with \eqref{conv_9}--\eqref{conv_11} and \eqref{conv_1031}, we deduce the following convergence properties as $n\rightarrow\infty$:
    \begin{align}
    &\nu\int_0^t\int_\Omega\Upsilon_p(\nabla\overline{\bm{u}}_n(\sigma))\,dxd\sigma\rightarrow\nu\int_0^t\int_\Omega\Upsilon_p(\nabla{\bm{u}}_\nu(\sigma))\,dxd\sigma,\label{conv_104}
    \\
    &\int_0^t(\nabla\partial_t\bm{u}_n(\sigma),\nabla\overline{\bm{u}}_n(\sigma))_{[H]^{M\times N}}\,d\sigma\rightarrow\int_0^t(\nabla\partial_t\bm{u}(\sigma),\nabla{\bm{u}}_\nu(\sigma))_{[H]^{M\times N}}\,d\sigma,\label{conv_105}
    \\
    &|\nabla\bm{u}_n(t)|^2_{[H]^{M\times N}}\rightarrow|\nabla\bm{u}_\nu(t)|^2_{[H]^{M\times N}}, \mbox{ for any }t\in[0,T].\label{conv_106}
  \end{align}\noeqref{conv_105}
    \eqref{conv_16}, \eqref{conv_17}, \eqref{conv_100} and \eqref{conv_104} imply 
  \begin{gather}\label{convJukai}
     \begin{cases}
         E_{\nu,\varepsilon_n}(\overline{\bm{u}}_n(t)) \to \mathcal{\widetilde{J}}_*(t) = E_\nu(\bm{u}(t)),
         \\
         E_{\nu,\varepsilon_n}(\underline{\bm{u}}_n(t)) \to \mathcal{\widetilde{J}}_*(t) = E_\nu(\bm{u}(t)), 
     \end{cases}
     \mbox{a.e. $ t \in (0, T) $, as $ n \to \infty $.}
 \end{gather}
 Now, we take $\bm{\omega}=\varphi$ in $[W^{1,p}(\Omega)]^M$ in \eqref{conv_103}, and pass to the limit $n\rightarrow\infty$. In light of \eqref{conv_100} and \eqref{conv_104}--\eqref{conv_106}, by the Lebesgue dominated convergence theorem, we observe
     \begin{gather}\label{conv_107}
    \int_I(A(\bm{u}_\nu(t))\partial_t\bm{u}_\nu(t),\bm{u}_\nu(t)-\bm{\varphi})_{[H]^M}\,dt+\int_I\mu(\nabla\partial_t\bm{u}_\nu(t),\nabla(\bm{u}_\nu(t)-\bm{\varphi}))_{[H]^{M\times N}}\,dt
    \\[-0.75ex]
    +\int_I(\nabla_{\bm{u}} G(x,\bm{u}_\nu(t))+[\nabla\alpha](\bm{u}_\nu(t))\gamma({B}(\bm{u}_\nu(t))\nabla\bm{u}_\nu(t)),\bm{u}_\nu(t)-\bm{\varphi})_{[H]^M}\,dt
    \\[-0.75ex]
    +\int_I(\alpha(\bm{u}_\nu(t))\bm{w}^*(t):[\nabla{B}](\bm{u}_\nu(t))\nabla\bm{u}_\nu(t),\bm{u}_\nu(t)-\bm{\varphi})_{[H]^M}\,dt
    \\[-0.75ex]
    +\nu\int_\Omega \nabla \Upsilon_p(\nabla \bm{u}_\nu(t)):\nabla (\bm{u}_\nu(t)-\bm{\varphi})\,dxdt
    \\[-0.75ex]
    +\int_I\int_{\Omega}\alpha(\bm{u}_\nu(t))\gamma(B(\bm{u}_\nu(t))\nabla \bm{u}_\nu(t))\,dxdt\leq\int_I\int_{\Omega}\alpha(\bm{u}_\nu(t))\gamma (B(\bm{u}_\nu(t))\nabla\bm{\varphi})\,dxdt,
    \end{gather}
    for any open interval $I\subset(0,T)$. The rest of the discussion proceeds in the same way as in the subsection \ref{subnu}.

Thus, we complete the proof of Main Theorem \ref{mainThm1}. 
\end{proof}

\section{Proof of Main Theorem 2.}\label{sec:proof2}
Let $\nu>0$ be a fixed positive constant. Let $ \bm{u}_k \in W^{1,2}(0,T;[V]^M)\cap L^\infty(0,T;[W^{1,p}(\Omega)]^M)$, for $k = 1 , 2 $, denote the solution of the system (S)$_\nu$ with the initial condition $  \bm{u}_1 ( 0 ) = \bm{u}_2 ( 0 ) = \bm{u}_0 \in [W^{1,p}(\Omega)]^M$. By putting $ \bm{\varphi} = \bm{u}_2 $ in the variational inequality for $ \bm{u}_1 $, and $ \bm{\varphi} = \bm{u}_1 $ in the one for $ \bm{u}_2 $, and then summing the two inequalities, the following estimate is obtained. Furthermore, under the additional assumptions $A\in C^1(\R^M;\R^{M\times M})$ and $\gamma\in C^{1,1}(\R^{M\times N})\cap C^2(\R^{M\times N})$, it follows that
\begin{align}
  &(A(\bm{u}_1(t))\partial_t\bm{u}_1(t)-A(\bm{u}_2(t))\partial_t\bm{u}_2(t),(\bm{u}_1-\bm{u}_2)(t))_{[H]^M}
  \\
  &+\frac{1}{2}\frac{d}{dt}\left(\mu|\nabla(\bm{u}_1-\bm{u}_2)(t)|^2_{[H]^{M\times N}}\right)
  \\
  &+(\nabla_{\bm{u}} G(x,\bm{u}_1(t))-\nabla_{\bm{u}} G(x,\bm{u}_2(t)),(\bm{u}_1-\bm{u}_2)(t))_{[H]^M}
  \\
  &+\nu\int_\Omega \bigl(\nabla \Upsilon_p(\nabla \bm{u}_1(t))-\nabla\Upsilon_p(\nabla\bm{u}_2(t))\bigr):\nabla(\bm{u}_1-\bm{u}_2)(t)\,dx
  \\
  &-\int_\Omega \big[\alpha(\bm{u}_1(t))\nabla\gamma(\nabla\bm{u}_1(t)B_0(\bm{u}_1(t))){}^\top B_0(\bm{u}_1(t))
  \\
  &\qquad-\alpha(\bm{u}_2(t))\nabla\gamma(\nabla\bm{u}_2(t)B_0(\bm{u}_2(t))){}^\top B_0(\bm{u}_2(t))\big]:\nabla(\bm{u}_1-\bm{u}_2)(t)\,dx
  \\
  &+\int_\Omega \big[[\nabla\alpha](\bm{u}_1(t))\gamma(\nabla\bm{u}_1(t)B_0(\bm{u}_1(t)))
  \\
  &\qquad-[\nabla\alpha](\bm{u}_2(t))\gamma(\nabla\bm{u}_2(t)B_0(\bm{u}_2(t)))\big]\cdot(\bm{u}_1-\bm{u}_2)(t)\,dx
  \\
  &+\int_\Omega \big[\alpha(\bm{u}_1(t))\nabla\gamma(\nabla\bm{u}_1(t)B_0(\bm{u}_1(t))):\nabla\bm{u}_1(t)[\nabla B_0](\bm{u}_1(t))
  \\
  &\qquad-\alpha(\bm{u}_2(t))\nabla\gamma(\nabla\bm{u}_2(t)B_0(\bm{u}_2(t))):\nabla\bm{u}_2(t)[\nabla B_0](\bm{u}_2(t))\big]\cdot (\bm{u}_1-\bm{u}_2)(t)\,dx
  \\
  &\leq 0, \mbox{ a.e. }t\in(0,T).\label{uni001}
\end{align}
Also, by using assumption (A2) and (A5), we see that 
\begin{align}
  &(A(\bm{u}_1(t))\partial_t\bm{u}_1(t)-A(\bm{u}_2(t))\partial_t\bm{u}_2(t),(\bm{u}_1-\bm{u}_2)(t))_{[H]^M}
  \\
  &\qquad+\frac{1}{2}\frac{d}{dt}\left(\mu|\nabla(\bm{u}_1-\bm{u}_2)(t)|^2_{[H]^{M\times N}}\right)
  \\
  &\leq L_G|(\bm{u}_1-\bm{u}_2)(t)|^2_{[H]^M}
  \\
  &+\int_\Omega \big|\alpha(\bm{u}_1(t))\nabla\gamma(\nabla\bm{u}_1(t)B_0(\bm{u}_1(t))){}^\top B_0(\bm{u}_1(t))
  \\
  &\qquad-\alpha(\bm{u}_2(t))\nabla\gamma(\nabla\bm{u}_2(t)B_0(\bm{u}_2(t))){}^\top B_0(\bm{u}_2(t))\big||\nabla(\bm{u}_1-\bm{u}_2)(t)|\,dx
  \\
  &+\int_\Omega \big|[\nabla\alpha](\bm{u}_1(t))\gamma(\nabla\bm{u}_1(t)B_0(\bm{u}_1(t)))
  \\
  &\qquad-[\nabla\alpha](\bm{u}_2(t))\gamma(\nabla\bm{u}_2(t)B_0(\bm{u}_2(t)))\big||(\bm{u}_1-\bm{u}_2)(t)|\,dx
  \\
  &+\int_\Omega \big|\alpha(\bm{u}_1(t))\nabla\gamma(\nabla\bm{u}_1(t)B_0(\bm{u}_1(t))):\nabla\bm{u}_1(t)[\nabla B_0](\bm{u}_1(t))
  \\
  &\qquad-\alpha(\bm{u}_2(t))\nabla\gamma(\nabla\bm{u}_2(t)B_0(\bm{u}_2(t))):\nabla\bm{u}_2(t)[\nabla B_0](\bm{u}_2(t))\big||(\bm{u}_1-\bm{u}_2)(t)|\,dx
  \\
  &=:L_G|(\bm{u}_1-\bm{u}_2)(t)|^2_{[H]^M}+K_1+K_2+K_3.
  \label{uni002}
\end{align}
Now, we estimate $K_1$, $K_2$ and $K_3$. By the assumptions (A1) and (A3), we have
\begin{align}
  K_1&\leq \|\nabla B_0\|_{L^\infty}\|\alpha\|_{L^\infty}\|\nabla\gamma\|_{L^\infty} \int_\Omega |(\bm{u}_1-\bm{u}_2)(t)||\nabla(\bm{u}_1-\bm{u}_2)(t)|\,dx
  \\
  &\quad + \|B_0\|_{L^\infty}\|\nabla\alpha\|_{L^\infty}\|\nabla\gamma\|_{L^\infty} \int_\Omega |(\bm{u}_1-\bm{u}_2)(t)||\nabla(\bm{u}_1-\bm{u}_2)(t)|\,dx
  \\
  &\quad+\|B_0\|_{L^\infty}\|\alpha\|_{L^\infty}\|\nabla^2\gamma\|_{L^\infty}\|\nabla B_0\|_{L^\infty}\cdot 
  \\
  &\qquad\cdot \int_\Omega |(\bm{u}_1-\bm{u}_2)(t)||\nabla(\bm{u}_1-\bm{u}_2)(t)||\nabla\bm{u}_1(t)|\,dx
  \\
  &\quad + \|B_0\|_{L^\infty}^2\|\alpha\|_{L^\infty}\|\nabla^2\gamma\|_{L^\infty} |\nabla(\bm{u}_1-\bm{u}_2)(t)|^2_{[H]^{M\times N}}
  \\
  &\leq 3(\|B_0\|_{W^{2,\infty}}+1)^2\|\alpha\|_{W^{2,\infty}}\|\nabla\gamma\|_{W^{1,\infty}}|(\bm{u}_1-\bm{u}_2)(t)|_{[V]^M}^2
  \\
  &\quad+\|B_0\|_{L^\infty}\|\alpha\|_{L^\infty}\|\nabla^2\gamma\|_{L^\infty}\|\nabla B_0\|_{L^\infty}\cdot 
  \\
  &\qquad\cdot \int_\Omega |(\bm{u}_1-\bm{u}_2)(t)||\nabla(\bm{u}_1-\bm{u}_2)(t)||\nabla\bm{u}_1(t)|\,dx
  \\
  &=:3(\|B_0\|_{W^{2,\infty}}+1)^2\|\alpha\|_{W^{2,\infty}}\|\nabla\gamma\|_{W^{1,\infty}}|(\bm{u}_1-\bm{u}_2)(t)|_{[V]^M}^2+K_4,\label{K1}
  \\
  K_2&\leq \|\nabla^2\alpha\|_{L^\infty}C_\gamma\int_\Omega (\|B_0\|_{L^\infty}|\nabla\bm{u}_1(t)|+1)|(\bm{u}_1-\bm{u}_2)(t)|^2\,dx
  \\
  &\quad +\|\nabla\alpha\|_{L^\infty}\|\nabla\gamma\|_{L^\infty}\|\nabla B_0\|_{L^\infty}\int_\Omega |(\bm{u}_1-\bm{u}_2)(t)|^2|\nabla\bm{u}_1(t)|\,dx
  \\
  &\quad + \|\nabla\alpha\|_{L^\infty}\|\nabla\gamma\|_{L^\infty}\|B_0\|_{L^\infty} \int_\Omega |\nabla(\bm{u}_1-\bm{u}_2)(t)||(\bm{u}_1-\bm{u}_2)(t)|\,dx
  \\
  &\leq\|B_0\|_{W^{2,\infty}}\|\alpha\|_{W^{2,\infty}}(C_\gamma+\|\nabla\gamma\|_{W^{1,\infty}})\int_\Omega |(\bm{u}_1-\bm{u}_2)(t)|^2|\nabla\bm{u}_1(t)|\,dx
  \\
  &\quad +\|\alpha\|_{W^{2,\infty}}(\|B_0\|_{W^{2,\infty}}\|\nabla\gamma\|_{W^{1,\infty}}+1)|(\bm{u}_1-\bm{u}_2)(t)|_{[V]^M}^2
  \\
  &=:K_5+\|\alpha\|_{W^{2,\infty}}(\|B_0\|_{W^{2,\infty}}\|\nabla\gamma\|_{W^{1,\infty}}+1)|(\bm{u}_1-\bm{u}_2)(t)|_{[V]^M}^2,\label{K2}
  \\
  K_3 &\leq \|\nabla\alpha\|_{L^\infty}\|\nabla\gamma\|_{L^\infty}\|\nabla B_0\|_{L^\infty}\int_\Omega |(\bm{u}_1-\bm{u}_2)(t)|^2|\nabla \bm{u}_1(t)|\,dx
  \\
  &\quad + \|\alpha\|_{L^\infty}\|\nabla^2\gamma\|_{L^\infty}\|\nabla B_0\|_{L^\infty}^2 \int_\Omega |(\bm{u}_1-\bm{u}_2)(t)|^2|\nabla \bm{u}_1(t)|^2\,dx
  \\
  &\quad + \|\alpha\|_{L^\infty}\|\nabla^2\gamma\|_{L^\infty}\|B_0\|_{L^\infty}\|\nabla B_0\|_{L^\infty}\cdot 
  \\
  &\qquad \cdot\int_\Omega |\nabla(\bm{u}_1-\bm{u}_2)(t)||\nabla \bm{u}_1(t)||(\bm{u}_1-\bm{u}_2)(t)|\,dx
  \\
  &\quad +\|\alpha\|_{L^\infty}\|\nabla\gamma\|_{L^\infty}\|\nabla^2 B_0\|_{L^\infty} \int_\Omega |(\bm{u}_1-\bm{u}_2)(t)|^2|\nabla \bm{u}_1(t)|\,dx
  \\
  &\quad +\|\alpha\|_{L^\infty}\|\nabla\gamma\|_{L^\infty}\|\nabla B_0\|_{L^\infty}\int_\Omega |\nabla(\bm{u}_1-\bm{u}_2)(t)||(\bm{u}_1-\bm{u}_2)(t)|\,dx
  \\
  &\leq \|B_0\|_{W^{2,\infty}}\|\alpha\|_{W^{2,\infty}}\|\nabla\gamma\|_{W^{1,\infty}}|(\bm{u}_1-\bm{u}_2)(t)|_{[V]^M}^2
  \\
  &\quad +2\|B_0\|_{W^{2,\infty}}\|\alpha\|_{W^{2,\infty}}\|\nabla\gamma\|_{W^{1,\infty}}\int_\Omega |(\bm{u}_1-\bm{u}_2)(t)|^2|\nabla\bm{u}_1(t)|\,dx
  \\
  &\quad +\|B_0\|_{W^{2,\infty}}^2\|\alpha\|_{W^{2,\infty}}\|\nabla\gamma\|_{W^{1,\infty}}\int_\Omega |(\bm{u}_1-\bm{u}_2)(t)|^2|\nabla \bm{u}_1(t)|^2\,dx
  \\
  &\quad +\|B_0\|_{W^{2,\infty}}^2\|\alpha\|_{W^{2,\infty}}\|\nabla\gamma\|_{W^{1,\infty}}\int_\Omega |\nabla(\bm{u}_1-\bm{u}_2)(t)||\nabla \bm{u}_1(t)||(\bm{u}_1-\bm{u}_2)(t)|\,dx
  \\
  &=:\|B_0\|_{W^{2,\infty}}\|\alpha\|_{W^{2,\infty}}\|\nabla\gamma\|_{W^{1,\infty}}|(\bm{u}_1-\bm{u}_2)(t)|_{[V]^M}^2+K_6+K_7+K_8.\label{K3}
\end{align}
Here, we consider the estimates of $K_i$, $i=4,5,6,7,8$ in the following three cases:\vspace{1.0ex}

\textbf{(Case 1)} $N=1,2$,~~~\textbf{(Case 2)} $3\leq N\leq6$.
\vspace{1.0ex}

\noindent
We first consider the case $N=1,2$. By Young's inequality, we obtain
\begin{align*}
  K_4&\leq \|B_0\|_{W^{2,\infty}}^2\|\alpha\|_{W^{2,\infty}}\|\nabla\gamma\|_{W^{1,\infty}}|(\bm{u}_1-\bm{u}_2)(t)|_{[L^{\frac{2p}{p-2}}(\Omega)]^M}|\nabla\bm{u}_1(t)|_{[L^p(\Omega)]^{M\times N}}\cdot 
  \\
  &\quad \cdot|\nabla(\bm{u}_1-\bm{u}_2)(t)|_{[H]^{M\times N}}
  \\
  &\leq\frac{C_{H^1}^{L^{\frac{2p}{p-2}}}\|B_0\|_{W^{2,\infty}}^2\|\alpha\|_{W^{2,\infty}}\|\nabla\gamma\|_{W^{1,\infty}}}{C_A \land \mu}|\nabla\bm{u}_1(t)|_{[L^p(\Omega)]^{M\times N}}\cdot 
  \\
  &\quad \cdot\bigl(|\sqrt{A(\bm{u}_1(t))}(\bm{u}_1-\bm{u}_2)(t)|_{[H]^M}^2+\mu|\nabla(\bm{u}_1-\bm{u}_2)(t)|_{[H]^{M\times N}}^2\bigr),
  \\
  K_5&\leq \|B_0\|_{W^{2,\infty}}\|\alpha\|_{W^{2,\infty}}(C_\gamma+\|\nabla\gamma\|_{W^{1,\infty}})|(\bm{u}_1-\bm{u}_2)(t)|^2_{[L^\frac{2p}{p-1}(\Omega)]^M}\cdot 
  \\
  &\qquad \cdot |\nabla\bm{u}_1(t)|_{[L^p(\Omega)]^{M\times N}}
  \\
  &\leq \frac{(C_{H^1}^{L^\frac{2p}{p-1}})^2\|B_0\|_{W^{2,\infty}}\|\alpha\|_{W^{2,\infty}}(C_\gamma+\|\nabla\gamma\|_{W^{1,\infty}})}{C_A \land \mu}|\nabla\bm{u}_1(t)|_{[L^p(\Omega)]^{M\times N}}\cdot
  \\
  &\quad\cdot \bigl(|\sqrt{A(\bm{u}_1(t))}(\bm{u}_1-\bm{u}_2)(t)|_{[H]^M}^2+\mu|\nabla(\bm{u}_1-\bm{u}_2)(t)|_{[H]^{M\times N}}^2\bigr)
  \\
  K_6&\leq 2\|B_0\|_{W^{2,\infty}}\|\alpha\|_{W^{2,\infty}}\|\nabla\gamma\|_{W^{1,\infty}}|(\bm{u}_1-\bm{u}_2)(t)|^2_{[L^\frac{2p}{p-1}(\Omega)]^M}|\nabla\bm{u}_1(t)|_{[L^p(\Omega)]^{M\times N}}
  \\
  &\leq \frac{2(C_{H^1}^{L^\frac{2p}{p-1}})^2\|B_0\|_{W^{2,\infty}}\|\alpha\|_{W^{2,\infty}}\|\nabla\gamma\|_{W^{1,\infty}}}{C_A \land \mu}|\nabla\bm{u}_1(t)|_{[L^p(\Omega)]^{M\times N}}\cdot
  \\
  &\quad\cdot \bigl(|\sqrt{A(\bm{u}_1(t))}(\bm{u}_1-\bm{u}_2)(t)|_{[H]^M}^2+\mu|\nabla(\bm{u}_1-\bm{u}_2)(t)|_{[H]^{M\times N}}^2\bigr)
  \\
  K_7&\leq \|B_0\|_{W^{2,\infty}}^2\|\alpha\|_{W^{2,\infty}}\|\nabla\gamma\|_{W^{1,\infty}}|(\bm{u}_1-\bm{u}_2)(t)|^2_{[L^\frac{2p}{p-2}(\Omega)]^M}|\nabla \bm{u}_1(t)|^2_{[L^p(\Omega)]^{M\times N}}
  \\
  &\leq \frac{2(C_{H^1}^{L^\frac{2p}{p-2}})^2\|B_0\|_{W^{2,\infty}}^2\|\alpha\|_{W^{2,\infty}}\|\nabla\gamma\|_{W^{1,\infty}}}{C_A \land \mu}|\nabla\bm{u}_1(t)|_{[L^p(\Omega)]^{M\times N}}^2\cdot
  \\
  &\quad\cdot \bigl(|\sqrt{A(\bm{u}_1(t))}(\bm{u}_1-\bm{u}_2)(t)|_{[H]^M}^2+\mu|\nabla(\bm{u}_1-\bm{u}_2)(t)|_{[H]^{M\times N}}^2\bigr)
  \\
  K_8&\leq \|B_0\|_{W^{2,\infty}}^2\|\alpha\|_{W^{2,\infty}}\|\nabla\gamma\|_{W^{1,\infty}}|(\bm{u}_1-\bm{u}_2)(t)|_{[L^{\frac{2p}{p-2}}(\Omega)]^M}\cdot 
  \\
  &\quad \cdot|\nabla\bm{u}_1(t)|_{[L^p(\Omega)]^{M\times N}}|\nabla(\bm{u}_1-\bm{u}_2)(t)|_{[H]^{M\times N}}
  \\
  &\leq\frac{C_{H^1}^{L^{\frac{2p}{p-2}}}\|B_0\|_{W^{2,\infty}}^2\|\alpha\|_{W^{2,\infty}}\|\nabla\gamma\|_{W^{1,\infty}}}{C_A \land \mu}|\nabla\bm{u}_1(t)|_{[L^p(\Omega)]^{M\times N}}\cdot 
  \\
  &\quad \cdot\bigl(|\sqrt{A(\bm{u}_1(t))}(\bm{u}_1-\bm{u}_2)(t)|_{[H]^M}^2+\mu|\nabla(\bm{u}_1-\bm{u}_2)(t)|_{[H]^{M\times N}}^2\bigr),
\end{align*}
where the Sobolev embeddings $H^1(\Omega)\subset L^{\frac{2p}{p-2}}(\Omega)$ and $H^1(\Omega)\subset L^{\frac{2p}{p-1}}(\Omega)$ have been applied, with the corresponding embedding constants denoted by $C_{H^1}^{L^{\frac{2p}{p-2}}}$ and $C_{H^1}^{L^{\frac{2p}{p-1}}}$.

Next, we consider the case $3\leq N\leq 6$. By applying Young's inequality, the integral terms $K_i$, $i=4,5,6,7$ can be estimated as follows:
\begin{align*}
  K_4&\leq \|B_0\|_{W^{2,\infty}}^2\|\alpha\|_{W^{2,\infty}}\|\nabla\gamma\|_{W^{1,\infty}}|(\bm{u}_1-\bm{u}_2)(t)|_{[L^{\frac{2N}{N-2}}(\Omega)]^M}\cdot 
  \\
  &\quad \cdot|\nabla(\bm{u}_1-\bm{u}_2)(t)|_{[H]^{M\times N}}|\nabla\bm{u}_1(t)|_{[L^p(\Omega)]^{M\times N}}|\Omega|^{\frac{p-N}{pN}}
  \\
  &\leq\frac{C_{H^1}^{L^{\frac{2N}{N-2}}}|\Omega|^{\frac{p-N}{pN}}\|B_0\|_{W^{2,\infty}}^2\|\alpha\|_{W^{2,\infty}}\|\nabla\gamma\|_{W^{1,\infty}}}{C_A \land \mu}|\nabla\bm{u}_1(t)|_{[L^p(\Omega)]^{M\times N}}\cdot 
  \\
  &\quad \cdot\bigl(|\sqrt{A(\bm{u}_1(t))}(\bm{u}_1-\bm{u}_2)(t)|_{[H]^M}^2+\mu|\nabla(\bm{u}_1-\bm{u}_2)(t)|_{[H]^{M\times N}}^2\bigr),
    \\
  K_5&\leq \|B_0\|_{W^{2,\infty}}\|\alpha\|_{W^{2,\infty}}(C_\gamma+\|\nabla\gamma\|_{W^{1,\infty}})|\Omega|^\frac{2p-N}{pN}\cdot
  \\
  &\quad\cdot |(\bm{u}_1-\bm{u}_2)(t)|^2_{[L^{\frac{2N}{N-2}}(\Omega)]^M}|\nabla\bm{u}_1(t)|_{[L^p(\Omega)]^{M\times N}}
  \\
  &\leq \frac{(C_{H^1}^{L^{\frac{2N}{N-2}}})^2\|B_0\|_{W^{2,\infty}}\|\alpha\|_{W^{2,\infty}}(C_\gamma+\|\nabla\gamma\|_{W^{1,\infty}})|\Omega|^\frac{2p-N}{pN}}{C_A \land \mu}\cdot
  \\
  &\quad\cdot |\nabla\bm{u}_1(t)|_{[L^p(\Omega)]^{M\times N}}\bigl(|\sqrt{A(\bm{u}_1(t))}(\bm{u}_1-\bm{u}_2)(t)|_{[H]^M}^2+\mu|\nabla(\bm{u}_1-\bm{u}_2)(t)|_{[H]^{M\times N}}^2\bigr)
  \\
  K_6&\leq 2\|B_0\|_{W^{2,\infty}}\|\alpha\|_{W^{2,\infty}}\|\nabla\gamma\|_{W^{1,\infty}}|\Omega|^\frac{2p-N}{pN}|(\bm{u}_1-\bm{u}_2)(t)|^2_{[L^\frac{2N}{N-2}(\Omega)]^M}\cdot 
  \\
  &\qquad\cdot |\nabla\bm{u}_1(t)|_{[L^p(\Omega)]^{M\times N}}
  \\
  &\leq \frac{2(C_{H^1}^{L^\frac{2N}{N-2}})^2\|B_0\|_{W^{2,\infty}}\|\alpha\|_{W^{2,\infty}}\|\nabla\gamma\|_{W^{1,\infty}}|\Omega|^\frac{2p-N}{pN}}{C_A \land \mu}|\nabla\bm{u}_1(t)|_{[L^p(\Omega)]^{M\times N}}\cdot
  \\
  &\quad\cdot \bigl(|\sqrt{A(\bm{u}_1(t))}(\bm{u}_1-\bm{u}_2)(t)|_{[H]^M}^2+\mu|\nabla(\bm{u}_1-\bm{u}_2)(t)|_{[H]^{M\times N}}^2\bigr)
  \\
  K_7&\leq \|B_0\|_{W^{2,\infty}}^2\|\alpha\|_{W^{2,\infty}}\|\nabla\gamma\|_{W^{1,\infty}}|\Omega|^\frac{2(p-N)}{pN}|(\bm{u}_1-\bm{u}_2)(t)|^2_{[L^\frac{2N}{N-2}(\Omega)]^M}\cdot 
  \\
  &\qquad\cdot |\nabla \bm{u}_1(t)|^2_{[L^p(\Omega)]^{M\times N}}
  \\
  &\leq \frac{2(C_{H^1}^{L^\frac{2N}{N-2}})^2\|B_0\|_{W^{2,\infty}}^2\|\alpha\|_{W^{2,\infty}}\|\nabla\gamma\|_{W^{1,\infty}}|\Omega|^\frac{2(p-N)}{pN}}{C_A \land \mu}|\nabla\bm{u}_1(t)|_{[L^p(\Omega)]^{M\times N}}\cdot
  \\
  &\quad\cdot \bigl(|\sqrt{A(\bm{u}_1(t))}(\bm{u}_1-\bm{u}_2)(t)|_{[H]^M}^2+\mu|\nabla(\bm{u}_1-\bm{u}_2)(t)|_{[H]^{M\times N}}^2\bigr)
  \\
  K_8&\leq \|B_0\|_{W^{2,\infty}}^2\|\alpha\|_{W^{2,\infty}}\|\nabla\gamma\|_{W^{1,\infty}}|(\bm{u}_1-\bm{u}_2)(t)|_{[L^{\frac{2N}{N-2}}(\Omega)]^M}\cdot 
  \\
  &\quad \cdot|\nabla(\bm{u}_1-\bm{u}_2)(t)|_{[H]^{M\times N}}|\nabla\bm{u}_1(t)|_{[L^p(\Omega)]^{M\times N}}|\Omega|^{\frac{p-N}{pN}}
  \\
  &\leq\frac{C_{H^1}^{L^{\frac{2N}{N-2}}}|\Omega|^{\frac{p-N}{pN}}\|B_0\|_{W^{2,\infty}}^2\|\alpha\|_{W^{2,\infty}}\|\nabla\gamma\|_{W^{1,\infty}}}{C_A \land \mu}|\nabla\bm{u}_1(t)|_{[L^p(\Omega)]^{M\times N}}\cdot 
  \\
  &\quad \cdot\bigl(|\sqrt{A(\bm{u}_1(t))}(\bm{u}_1-\bm{u}_2)(t)|_{[H]^M}^2+\mu|\nabla(\bm{u}_1-\bm{u}_2)(t)|_{[H]^{M\times N}}^2\bigr),
\end{align*}
where the above inequalities are obtained by employing the Sobolev embedding $H^1(\Omega)\subset L^{\frac{2N}{N-2}}(\Omega)$, with the embedding constant denoted by $C_{H^1}^{L^{\frac{2N}{N-2}}}$.

Combining the estimates for $K_i$, $i=4,5,6,7,8$ in the above two cases, we obtain the following unified estimate.
\begin{align}\label{K4}
  K_4&\leq C_1(|\nabla\bm{u}_1(t)|_{[L^p(\Omega)]^{M\times N}}+1)^2\cdot 
  \\
  &\qquad\cdot \bigl(|\sqrt{A(\bm{u}_1(t))}(\bm{u}_1-\bm{u}_2)(t)|_{[H]^M}^2+\mu|\nabla(\bm{u}_1-\bm{u}_2)(t)|_{[H]^{M\times N}}^2\bigr),
\end{align}
where, 
\begin{align*}
  C_1:=\frac{(C_{H^1}^{L^{\frac{2p}{p-2}}}+C_{H^1}^{L^{\frac{2N}{N-2}}}|\Omega|^{\frac{p-N}{pN}})\|B_0\|_{W^{2,\infty}}^2\|\alpha\|_{W^{2,\infty}}\|\nabla\gamma\|_{W^{1,\infty}}}{C_A \land \mu},
\end{align*}
\begin{align}\label{K5}
  K_5&\leq C_2|\nabla\bm{u}_1(t)|_{[L^p(\Omega)]^{M\times N}}\cdot 
  \\
  &\qquad\cdot\bigl(|\sqrt{A(\bm{u}_1(t))}(\bm{u}_1-\bm{u}_2)(t)|_{[H]^M}^2+\mu|\nabla(\bm{u}_1-\bm{u}_2)(t)|_{[H]^{M\times N}}^2\bigr),
\end{align}
where, 
\begin{align*}
  C_2&:=\frac{((C_{H^1}^{L^\frac{2p}{p-1}})^2+(C_{H^1}^{L^{\frac{2N}{N-2}}})^2|\Omega|^\frac{2p-N}{pN})\|B_0\|_{W^{2,\infty}}\|\alpha\|_{W^{2,\infty}}(C_\gamma+\|\nabla\gamma\|_{W^{1,\infty}})}{C_A \land \mu},
\end{align*}
\begin{align}\label{K6}
  K_6&\leq C_3|\nabla\bm{u}_1(t)|_{[L^p(\Omega)]^{M\times N}}\cdot 
  \\
  &\qquad\cdot\bigl(|\sqrt{A(\bm{u}_1(t))}(\bm{u}_1-\bm{u}_2)(t)|_{[H]^M}^2+\mu|\nabla(\bm{u}_1-\bm{u}_2)(t)|_{[H]^{M\times N}}^2\bigr),
\end{align}
where, 
\begin{align*}
  C_3&:=\frac{2((C_{H^1}^{L^\frac{2p}{p-1}})^2+(C_{H^1}^{L^\frac{2N}{N-2}})^2|\Omega|^\frac{2p-N}{pN})\|B_0\|_{W^{2,\infty}}\|\alpha\|_{W^{2,\infty}}\|\nabla\gamma\|_{W^{1,\infty}}}{C_A \land \mu},
\end{align*}
\begin{align}\label{K7}
  K_7&\leq C_4|\nabla\bm{u}_1(t)|_{[L^p(\Omega)]^{M\times N}}^2\cdot 
  \\
  &\qquad\cdot \bigl(|\sqrt{A(\bm{u}_1(t))}(\bm{u}_1-\bm{u}_2)(t)|_{[H]^M}^2+\mu|\nabla(\bm{u}_1-\bm{u}_2)(t)|_{[H]^{M\times N}}^2\bigr),
\end{align}
where, 
\begin{align*}
  C_4&:=\frac{2((C_{H^1}^{L^\frac{2p}{p-2}})^2+2(C_{H^1}^{L^\frac{2N}{N-2}})^2|\Omega|^\frac{2(p-N)}{pN})\|B_0\|_{W^{2,\infty}}^2\|\alpha\|_{W^{2,\infty}}\|\nabla\gamma\|_{W^{1,\infty}}}{C_A \land \mu}
\end{align*}
\begin{align}\label{K8}
  K_8&\leq C_5|\nabla\bm{u}_1(t)|_{[L^p(\Omega)]^{M\times N}}\cdot 
  \\
  &\qquad \cdot \bigl(|\sqrt{A(\bm{u}_1(t))}(\bm{u}_1-\bm{u}_2)(t)|_{[H]^M}^2+\mu|\nabla(\bm{u}_1-\bm{u}_2)(t)|_{[H]^{M\times N}}^2\bigr),
\end{align}
where, 
\begin{align*}
  C_5:=\frac{(C_{H^1}^{L^{\frac{2p}{p-2}}}+C_{H^1}^{L^{\frac{2N}{N-2}}})\|B_0\|_{W^{2,\infty}}^2\|\alpha\|_{W^{2,\infty}}\|\nabla\gamma\|_{W^{1,\infty}}}{C_A \land \mu}.
\end{align*}\noeqref{K1,K2,K3,K4,K5,K6,K7,K8}
Therefore, setting
\begin{align*}
  C_6:=\frac{3(\|B_0\|_{W^{2,\infty}}\|\nabla\gamma\|_{W^{1,\infty}}+1)^2\|\alpha\|_{W^{2,\infty}}}{C_A\land\mu},
\end{align*}
and using \eqref{K4}--\eqref{K8}, we obtain the following estimate for $K_1+K_2+K_3$:
\begin{align}\label{uni004}
  K_1+K_2+K_3&\leq C_7(|\nabla\bm{u}_1(t)|_{[L^p(\Omega)]^{M\times N}}+1)^2 \cdot 
  \\
  &\qquad \cdot \bigl(|\sqrt{A(\bm{u}_1(t))}(\bm{u}_1-\bm{u}_2)(t)|_{[H]^M}^2+\mu|\nabla(\bm{u}_1-\bm{u}_2)(t)|_{[H]^{M\times N}}^2\bigr),
\end{align}
where $C_7:=C_1+C_2+C_3+C_4+C_5+C_6.$

We next estimate for the first term on the left-hand side of \eqref{uni002} is given as follows.
\begin{align}
  &(A(\bm{u}_1(t))\partial_t\bm{u}_1(t)-A(\bm{u}_2(t))\partial_t\bm{u}_2(t),(\bm{u}_1-\bm{u}_2)(t))_{[H]^M}
  \\
  &=\frac{1}{2}\frac{d}{dt}\bigl( |\sqrt{A(\bm{u}_1(t))}(\bm{u}_1-\bm{u}_2)(t)|^2_{[H]^M} \bigr)
  \\
  &\quad -\frac{1}{2}\int_\Omega\partial_t\bm{u}_1(t)\hspace{-0.5ex}~^\top(\bm{u}_1-\bm{u}_2)(t)[\nabla A](\bm{u}_1(t))(\bm{u}_1-\bm{u}_2)(t)\,dx 
  \\
  &\quad+\int_\Omega (A(\bm{u}_1(t))-A(\bm{u}_2(t)))\partial_t\bm{u}_2(t)(\bm{u}_1-\bm{u}_2)(t)\,dx
  \\
  &\geq \frac{1}{2}\frac{d}{dt}\bigl( |\sqrt{A(\bm{u}_1(t))}(\bm{u}_1-\bm{u}_2)(t)|^2_{[H]^M} \bigr)
  \\
  &\quad-\frac{1}{2}\|\nabla A\|_{L^\infty}\int_\Omega|\partial_t\bm{u}_1(t)||(\bm{u}_1-\bm{u}_2)(t)|^2\,dx 
  \\
  &\quad -\|\nabla A\|_{L^\infty}\int_\Omega |\partial_t\bm{u}_2(t)||(\bm{u}_1-\bm{u}_2)(t)|^2\,dx
  \\
  &=:\frac{1}{2}\frac{d}{dt}\bigl( |\sqrt{A(\bm{u}_1(t))}(\bm{u}_1-\bm{u}_2)(t)|^2_{[H]^M} \bigr)+J_1+J_2.\label{uni003}
\end{align}
Here, by using the continuous embedding from $H^1(\Omega)$ to $L^3(\Omega)$ under $N \leq 6$, the integral terms $J_1$ and $J_2$ in \eqref{uni003} can be estimated as follows.
\begin{align}
  J_1&\geq -\frac{1}{2}\|\nabla A\|_{L^\infty}|\partial_t\bm{u}_1(t)|_{[L^3(\Omega)]^M}|(\bm{u}_1-\bm{u}_2)(t)|_{[L^3(\Omega)]^M}^2
  \\
  &\geq-\frac{(C_{H^1}^{L^3})^3\|\nabla A\|_{L^\infty}}{2}|\partial_t\bm{u}_1(t)|_{[V]^M}|(\bm{u}_1-\bm{u}_2)(t)|_{[V]^M}^2
  \\
  &\geq -\frac{(C_{H^1}^{L^3})^3\|A\|_{W^{1,\infty}}}{2(C_A \land \mu)}|\partial_t\bm{u}_1(t)|_{[V]^M}\cdot 
  \\
  &\qquad \cdot \bigl(|\sqrt{A(\bm{u}_1(t))}(\bm{u}_1-\bm{u}_2)(t)|^2_{[H]^M}+\mu|\nabla(\bm{u}_1-\bm{u}_2)(t)|^2_{[H]^{M\times N}} \bigr),\label{J_1}
  \\
  J_2&\geq-\|\nabla A\|_{L^\infty}|\partial_t\bm{u}_2(t)|_{[L^3(\Omega)]^M}|(\bm{u}_1-\bm{u}_2)(t)|_{[L^3(\Omega)]^M}^2
  \\
  &\geq-{(C_{H^1}^{L^3})^3\|\nabla A\|_{L^\infty}}|\partial_t\bm{u}_2(t)|_{[V]^M}|(\bm{u}_1-\bm{u}_2)(t)|_{[V]^M}^2
  \\
  &\geq -\frac{(C_{H^1}^{L^3})^3\|A\|_{W^{1,\infty}}}{C_A \land \mu}|\partial_t\bm{u}_2(t)|_{[V]^M}\cdot 
  \\
  &\qquad \cdot \bigl(|\sqrt{A(\bm{u}_1(t))}(\bm{u}_1-\bm{u}_2)(t)|^2_{[H]^M}+\mu|\nabla(\bm{u}_1-\bm{u}_2)(t)|^2_{[H]^{M\times N}} \bigr).\label{J_2}
\end{align} 
Therefore, by virtue of \eqref{J_1} and \eqref{J_2}, \eqref{uni003} can be estimated as follows:
\begin{align}
  &(A(\bm{u}_1(t))\partial_t\bm{u}_1(t)-A(\bm{u}_2(t))\partial_t\bm{u}_2(t),(\bm{u}_1-\bm{u}_2)(t))_{[H]^M}
  \\
  &\geq \frac{1}{2}\frac{d}{dt}\bigl( |\sqrt{A(\bm{u}_1(t))}(\bm{u}_1-\bm{u}_2)(t)|^2_{[H]^M} \bigr)
  \\
  &-\frac{(C_{H^1}^{L^3})^3\|A\|_{W^{1,\infty}}}{C_A \land \mu}(|\partial_t\bm{u}_1(t)|_{[V]^M}+|\partial_t\bm{u}_2(t)|_{[V]^M})\cdot
  \\
  &\qquad \cdot\bigl(|\sqrt{A(\bm{u}_1(t))}(\bm{u}_1-\bm{u}_2)(t)|^2_{[H]^M}+\mu|\nabla(\bm{u}_1-\bm{u}_2)(t)|^2_{[H]^{M\times N}} \bigr).\label{uni005}
\end{align}

Combining \eqref{uni002}, \eqref{uni004}, \eqref{uni003} and \eqref{uni005}, we arrive at
\begin{align*}
  &\frac{d}{dt}J(t)\leq C_*((1+|\nabla\bm{u}_1(t)|_{[L^p(\Omega)]^{M\times N}})^2+|\partial_t\bm{u}_1(t)|_{[V]^M}+|\partial_t\bm{u}_2(t)|_{[V]^M})J(t),
\end{align*}
where
\begin{gather*}
  C_*:=\frac{2(C_{H^1}^{L^3})^3\|A\|_{W^{1,\infty}}+2L_G}{C_A \land \mu}+2C_7,
  \\
  J(t)=|\sqrt{A(\bm{u}_1(t))}(\bm{u}_1-\bm{u}_2)(t)|^2_{[H]^M}+\mu|\nabla(\bm{u}_1-\bm{u}_2)(t)|^2_{[H]^{M\times N}}, 
  \\
  \mbox{ for all }t\in[0,T].
\end{gather*}
By applying Gronwall's lemma to \eqref{uni005}, we conclude that 
\begin{gather*}
  J(t)\leq \exp (C_*(1+T)((1+|\bm{u}_1|_{[L^\infty(0,T;[W^{1,p}(\Omega)]^M)]})^2+|\partial_t\bm{u}_1|_{[\mathscr{V}]^M}+|\partial_t\bm{u}_2|_{[\mathscr{V}]^M}))J(0),
  \\
  \mbox{ for all }t\in[0,T].
\end{gather*}

In addition, since \eqref{uni005} ensures the uniqueness of the solution to (S)$_\nu$, it follows that the energy inequality \eqref{ene-inq1} holds for all $ 0 \leq s \leq t \leq T $. Once the unique solution $ \bm{u}_\nu \in [\mathscr{H}]^M $ to (S)$_\nu$ has been obtained and an arbitrary $ s \in [0, T) $ has been fixed, the uniqueness property enables us to verify \eqref{energy3} ``for a. e. $ \widetilde{s} \in (s, T) $, including the case $ \widetilde{s} = s $, and for any $ t \in [\widetilde{s}, T] $.'' This is established by applying the time-discretization method described in Section \ref{timedis003} to the case where the initial data of (S)$_\nu$ is given by $ \bm{u}_\nu(s) \in [W^{1, p}(\Omega)]^M$. $\square$

\section{Applications}

Throughout this section, we assume (A0) and use Notation \ref{notnsp}. In addition, we focus on the two-dimensional setting $M = N = 2$, and demonstrate several applications of our Main Theorems.

\subsection{Pseudo-parabolic image denoising process with anisotropy}

In this subsection, we assume that:
\begin{itemize}
    \item $ \lambda \in  C^{0, 1}(\R) \cap L^\infty(\R) $ is a fixed function, which has a convex primitive $ 0 \leq \widehat{\lambda} \in C^{1, 1}(\R)  $ satisfying $ \widehat{\lambda}(0) = 0 $,
    \item $u_\mathrm{org} \in L^\infty(\Omega)$ is a fixed function,
    \item $ 0 \leq \gamma_0 : \R^2 \longrightarrow \R $ is a fixed nonnegative, even, Lipschitz, and convex function,  i.e. $0 \leq \gamma_0 \in C^{0,1}(\R^2)$ and $\gamma_0(-\bm{w})=\gamma_0(\bm{w})$ for all $\bm{w}\in\R^2$,
    \item $R \in C^\infty(\R; \R^{2\times 2})$ is the matrix-valued function defined by
        \begin{gather}\label{rot}
            R : \vartheta \in \R \mapsto 
            R(\vartheta)
            :=
            \left[
                \begin{matrix}
                    \cos \vartheta & -\sin \vartheta \\
                    \sin \vartheta & \cos \vartheta
                \end{matrix}
            \right]
            \in SO(2) \subset \R^{2 \times 2}.
        \end{gather}
\end{itemize}

Under these assumptions, we consider the following pseudo-parabolic system:
\begin{align}
    &
    \label{alpha00001}
    \begin{cases}
        \partial_t \theta
        - \mathrm{div} \bigl(
            \nabla \theta
            + \mu \nabla \partial_t \theta
            + \nu |\nabla \theta|^{p-2} \nabla \theta
        \bigr)
        \\
        \qquad
        + \partial \gamma_0(R(\theta)\nabla u)
            \cdot R(\theta + {\ts \frac{\pi}{2}})\nabla u
        \ni 0
        \quad \mbox{in $Q$,}
        \\[1ex]
        \bigl(
            \nabla \theta
            + \nu |\nabla \theta|^{p-2} \nabla \theta
            + \mu \nabla \partial_t \theta
        \bigr)
        \cdot\bm{n}_\Gamma = 0 \quad \mbox{on $\Sigma$,}
        \\[1ex]
        \theta(0,x)=\theta_0(x),
        \quad x \in \Omega,
    \end{cases}
    \\
    &
    \label{u00001}
    \begin{cases}
        \partial_t u
        - \mathrm{div}\bigl(
            R(-\theta)\,
            \partial\gamma_0(R(\theta)\nabla u)
            + \mu \nabla \partial_t u
            + \nu |\nabla u|^{p-2}\nabla u
        \bigr)
        \\
        \hspace{4ex}
        + \lambda (u - u_\mathrm{org})
        \ni 0
        \quad \mbox{in $Q$,}
        \\[1ex]
        \bigl(
            R(-\theta)\,
            \partial\gamma_0(R(\theta)\nabla u)
            + \mu \nabla \partial_t u
            + \nu |\nabla u|^{p-2}\nabla u
        \bigr)\cdot\bm{n}_\Gamma = 0
        \quad \mbox{on $\Sigma$,}
        \\[1ex]
        u(0,x)=u_0(x),
        \quad x \in \Omega.
    \end{cases}
\end{align}
This system is motivated by the minimization process of the energy functional
\begin{gather}\label{fe61}
    E : 
    [u,\theta]
    \in D_\nu := \bigl\{ [v, w] \in [V]^2 \,\bigl|\, \nu [v, w] \in [W^{1, p}(\Omega)]^2 \bigr\}
    \nonumber
    \\
    \mapsto
    E(u,\theta)
    :=
    \frac{1}{2}\int_\Omega |\nabla\theta|^2\,dx
    + \int_\Omega \gamma_0(R(\theta)\nabla u)\,dx
    \\
    + \frac{1}{p}\int_\Omega \nu \bigl(|\nabla\theta|^p + |\nabla u|^p\bigr)\,dx
    + \int_\Omega \widehat{\lambda}(u - u_\mathrm{org})\,dx
    \in [0,\infty),
    \label{energyimage}
\end{gather}
which is based on the free-energy used in anisotropic monochrome image denoising problems,
see \cite{berkels2006cartoon,AMSU202411}.
\medskip

In this context, $\Omega \subset \R^2$ represents the spatial domain of the monochrome image, and the unknown $u$ denotes its gray-scale intensity distribution.  
The second unknown $\theta$ encodes locally preferred orientations of structural features, with anisotropy determined by a function~$\gamma_0$.

Based on these observations, we set:
\begin{itemize}
    \item $\bm{u} := {}^\top[u,\theta]$ in $ \R^2 $ to unify the variables,
    \item $A(\bm{u}) := \left[
        \begin{matrix}
            1 & 0 \\ 0 & 1
        \end{matrix}
        \right] \in \R^{2\times 2}$, for all $ \bm{u} = {^\top} [u, \theta] \in \R^2 $,
\item $ \alpha(\bm{u}) := 1 $, for all $ \bm{u} = {^\top}[u, \theta] \in \R^2 $, 
\item $B(\bm{u})W := W\,R(\theta)$ for all $\bm{u} = {^\top} [u, \theta] \in \R^2$ and $W\in\R^{2\times 2}$,
\item $\gamma(W) := \gamma_0(w_{11}, w_{12})$ for all $W=[w_{ij}]_{1 \leq i, j \leq 2} \in \R^{2 \times 2} $,  
    with the convex function $\gamma_0$, 
\item $ \ds G(x, \bm{u}) := \widehat{\lambda}(u -u_\mathrm{org}(x)) $, for a.e. $ x \in \Omega $, and all $ \bm{u} = {^\top}[u, \theta] \in \R^2 $.
\end{itemize}
Under the above setting, we can verify the assumptions (A1)--(A6) in Main Theorems \ref{mainThm1} and \ref{mainThm2}.  
Therefore, the system \eqref{alpha00001}--\eqref{u00001} can be rewritten in the form of (S)$_\nu$, and the well-posedness of \eqref{alpha00001}--\eqref{u00001} follows directly from the mathematical framework established in this paper.

\subsection{Pseudo-parabolic KWC type system with anisotropy}

In this subsection, we assume that:
\begin{itemize}
    \item $ 0 < \alpha_0 \in W^{1, \infty}(\R) \cap C^1(\R) $ and $ 0 \leq \alpha_1 \in W^{1, \infty}(\R) \cap C^2(\R) $ are fixed functions,
    \item $ g \in C^{0, 1}(\R) \cap L^\infty(\R) $ is a fixed function, having a nonnegative primitive $ \widehat{g}  $,
    \item $ \gamma_0 $ is  the nonnegative, even, Lipschitz, and convex function as in the previous subsection, 
        $ R $ is the $ SO(2) $-valued operator given in \eqref{rot}, and $ D_\nu \subset [V]^2 $ is the domain of energy introduced in \eqref{fe61}.
\end{itemize}
On this basis, we consider the following pseudo-parabolic system:
\begin{align}
    & 
    \label{eta00001}
    \begin{cases}
      \partial_t \eta -\mathrm{div} \bigl(
            \nabla \eta
            + \mu \nabla \partial_t \eta
            + \nu |\nabla \eta|^{p-2} \nabla \eta
        \bigr) +g(\eta)+\alpha_1'(\eta)\gamma_0(\nabla\theta)= 0 \mbox{ in $  Q $,}
      \\[1ex]
      \bigl(
            \nabla \eta
            + \mu \nabla \partial_t \eta
            + \nu |\nabla \eta|^{p-2} \nabla \eta
        \bigr)\cdot \bm{n}_\Gamma = 0 \mbox{ on }\Sigma,
      \\[1ex]
      \eta(0, x) = \eta_0(x) ,~ x\in \Omega,
    \end{cases}
    \\
    & 
    \label{theta00001}
    \begin{cases}
        \alpha_0(\eta)\partial_t \theta -\mathrm{div} \bigl( \alpha_1(\eta) R(-\theta) \partial \gamma_0(R(\theta) \nabla \theta)
        +\mu \nabla \partial_t \theta
        +\nu|\nabla \theta|^{p -2} \nabla \theta 
        \bigr) 
        \\
        \quad +\alpha_1(\eta) \partial \gamma_0(R(\theta) \nabla \theta) \cdot R(\theta +\frac{\pi}{2}) \nabla \theta \ni 0 \mbox{ in $  Q $,}
        \\[1ex]
        \bigl( \alpha_1(\eta) R(-\theta) \partial \gamma_0(R(\theta) \nabla \theta)
        +\mu \nabla \partial_t \theta
        +\nu|\nabla \theta|^{p -2} \nabla \theta 
        \bigr) \cdot \bm{n}_\Gamma \ni 0 \mbox{ on }\Sigma,
        \\[1ex]
        \theta(0, x) = \theta_0(x) ,~ x\in \Omega,
    \end{cases}
\end{align}
The system \{\eqref{eta00001}, \eqref{theta00001}\} is a new anisotropic version of a phase-field 
model of grain boundary motion, known as the ``KWC model'', originally proposed by Kobayashi et al. 
\cite{MR1752970,MR1794359}. Related anisotropic formulations have also been discussed in \cite{Oberwolfach2018}. In accordance with the modelling framework of \cite{MR1752970,MR1794359}, the present system is derived 
as a gradient flow of the following free energy:
\begin{gather}
    E: [\eta, \theta] \in D_\nu 
    \mapsto E(\eta, \theta) := \frac{1}{2} \int_\Omega |\nabla \eta|^2 \, dx  +\frac{1}{p} \int_\Omega \nu \bigl( |\nabla \eta|^p +|\nabla \theta|^p \bigr) \, dx  
    \nonumber
    \\
    +\int_\Omega \alpha_1(\eta) \gamma_0(R(\theta) \nabla \theta) \, dx
    +\int_\Omega \widehat{g}(\eta) \, dx \in [0, \infty).
    \label{energykwc}
\end{gather}
In the anisotropic KWC-type model \{\eqref{eta00001}, \eqref{theta00001}\}, the unknowns $ \eta $ and $ \theta $ represent, respectively, the orientation order parameter and the crystalline orientation angle, in a polycrystalline body. The convex function $\gamma_0$ describes the structural unit of grain, while the $SO(2)$-valued map $R$ dynamically adjusts the orientation of the structural unit in response to the time evolution of the state variable $\theta$. Moreover, the functions $ \alpha_0 $ and $ \alpha_1 $ are state-dependent mobility coefficients with respect to the variable $ \eta $, and these play the role of driving forces in the evolution of grain boundaries. 
\medskip

Now, let us set:
\begin{itemize}
    \item $\bm{u} := {}^\top[\eta,\theta]$ in $ \R^2 $ to unify the variables,
    \item $A(\bm{u}) := \left[
        \begin{matrix}
            1 & 0 \\ 0 & \alpha_0(\eta)
        \end{matrix}
        \right] \in \R^{2\times 2}$, for all $ \bm{u} = {^\top} [\eta, \theta] \in \R^2 $,
\item $ \alpha(\bm{u}) = \alpha_1(\eta) $, for all $ \bm{u} = {^\top} [\eta, \theta] \in \R^2 $,
    \item $B(\bm{u})W := W\,R(\theta)$ for all $\bm{u} = {^\top} [\eta, \theta] \in \R^2$ and $W\in\R^{2\times 2}$,
    \item $\gamma(W) := \gamma_0(w_{21}, w_{22})$ for all $W=[w_{ij}]_{1 \leq i, j \leq 2} \in \R^{2 \times 2} $,  
\item $ \ds G(x,\bm{u}) := \widehat{g}(\eta) $, 
    for all $ \bm{u} = {^\top}[\eta, \theta] \in \R^2 $.
\end{itemize}
Then, by verifying assumptions \textnormal{(A1)}--\textnormal{(A6)}, one can immediately deduce the well-posedness of the anisotropic model \{\eqref{eta00001}, \eqref{theta00001}\} as a straightforward consequence of Main Theorems~\ref{mainThm1} and~\ref{mainThm2}.

\begin{rem}[On the relation to existing isotropic KWC-type models]
The KWC model was originally formulated in two spatial dimensions, and the existing studies 
(cf.\ \cite{MR2469586,MR2548486,MR3038131,MR3082861,MR3203495,MR2836555,MR2668289,%
MR3888633,MR3155454,MR4395725}) have mainly focused on the setting in which 
$\gamma_0$ is the Euclidean norm and $\nu = 0$. 
Additionally, in these works, the convexity of the mobility coefficient $\alpha_1$ is also usually imposed as a structural assumption. 
In contrast, our anisotropic KWC-type system \{\eqref{eta00001},\eqref{theta00001}\} does not rely on the convexity of $\alpha_1$. 
Therefore, the abstract framework developed in this paper not only enables us to incorporate crystalline anisotropy into the KWC-type model, but also provides new mathematical insight even for the isotropic models. 
In this sense, our results extend the mathematical theory of KWC-type phase-field systems beyond the conventional settings that rely on the convexity of $\alpha_1$.
\end{rem}

\providecommand{\href}[2]{#2}
\providecommand{\arxiv}[1]{\href{http://arxiv.org/abs/#1}{arXiv:#1}}
\providecommand{\url}[1]{\texttt{#1}}
\providecommand{\urlprefix}{URL }


\begin{thebibliography}{10}

\bibitem{MR1857292}
\newblock L.~Ambrosio, N.~Fusco and D.~Pallara,
\newblock \emph{Functions of Bounded Variation and Free Discontinuity
  Problems},
\newblock Oxford Mathematical Monographs, The Clarendon Press, Oxford
  University Press, New York, 2000.

  

\bibitem{MR4395725}
\newblock H.~Antil, S.~Kubota, K.~Shirakawa and N.~Yamazaki,
\newblock Constrained optimization problems governed by {PDE} models of grain
  boundary motions,
\newblock \emph{Adv. Nonlinear Anal.}, \textbf{11} (2022), 1249--1286,
\newblock \urlprefix\url{https://doi.org/10.1515/anona-2022-0242}.

  

\bibitem{AMSU202411}
\newblock H.~Antil, D.~Mizuno, K.~Shirakawa and N.~Ukai,
\newblock A gradient system based on anisotropic monochrome image processing
  with orientation auto-adjustment,
\newblock \emph{Adv. Math. Sci. Appl.}, \textbf{33} (2024), 755--782,
\newblock
  \urlprefix\url{https://mcm-www.jwu.ac.jp/~aikit/AMSA/pdf/abstract/2024/Top_2024_042.pdf}.

  

\bibitem{MR3888633}
\newblock H.~Antil, K.~Shirakawa and N.~Yamazaki,
\newblock A class of parabolic systems associated with optimal controls of
  grain boundary motions,
\newblock \emph{Adv. Math. Sci. Appl.}, \textbf{27} (2018), 299--336.

  

\bibitem{MR0773850}
\newblock H.~Attouch,
\newblock \emph{Variational Convergence for Functions and Operators},
\newblock Applicable Mathematics Series, Pitman (Advanced Publishing Program),
  Boston, MA, 1984.

  

\bibitem{MR2192832}
\newblock H.~Attouch, G.~Buttazzo and G.~Michaille,
\newblock \emph{Variational Analysis in {S}obolev and {BV} spaces}, vol.~6 of
  MPS/SIAM Series on Optimization,
\newblock Society for Industrial and Applied Mathematics (SIAM), Philadelphia,
  PA; Mathematical Programming Society (MPS), Philadelphia, PA, 2006,
\newblock Applications to PDEs and optimization.

  

\bibitem{berkels2006cartoon}
\newblock B.~Berkels, M.~Burger, M.~Droske, O.~Nemitz and M.~Rumpf,
\newblock \emph{Cartoon extraction based on anisotropic image classification},
\newblock CAM Report 06-42, 2006.

  

\bibitem{MR1040232}
\newblock P.~Colli, M.~Fr\'emond and A.~Visintin,
\newblock Thermo-mechanical evolution of shape memory alloys,
\newblock \emph{Quart. Appl. Math.}, \textbf{48} (1990), 31--47,
\newblock \urlprefix\url{https://doi.org/10.1090/qam/1040232}.

  

\bibitem{MR3661429}
\newblock P.~Colli, G.~Gilardi, R.~Nakayashiki and K.~Shirakawa,
\newblock A class of quasi-linear {A}llen--{C}ahn type equations with dynamic
  boundary conditions,
\newblock \emph{Nonlinear Anal.}, \textbf{158} (2017), 32--59,
\newblock \urlprefix\url{http://dx.doi.org/10.1016/j.na.2017.03.020}.

  

\bibitem{MR1201152}
\newblock G.~Dal~Maso,
\newblock \emph{An Introduction to {$\Gamma$}-convergence}, vol.~8 of Progress
  in Nonlinear Differential Equations and their Applications,
\newblock Birkh\"auser Boston, Inc., Boston, MA, 1993,
\newblock \urlprefix\url{http://dx.doi.org/10.1007/978-1-4612-0327-8}.

  

\bibitem{MR1885252}
\newblock M.~Fr\'emond,
\newblock \emph{Non-smooth thermomechanics},
\newblock Springer-Verlag, Berlin, 2002,
\newblock \urlprefix\url{https://doi.org/10.1007/978-3-662-04800-9}.

  

\bibitem{MR2096945}
\newblock Y.~Giga, Y.~Kashima and N.~Yamazaki,
\newblock Local solvability of a constrained gradient system of total
  variation,
\newblock \emph{Abstr. Appl. Anal.}, 651--682,
\newblock \urlprefix\url{https://doi.org/10.1155/S1085337504311048}.

  

\bibitem{MR2469586}
\newblock A.~Ito, N.~Kenmochi and N.~Yamazaki,
\newblock A phase-field model of grain boundary motion,
\newblock \emph{Appl. Math.}, \textbf{53} (2008), 433--454,
\newblock \urlprefix\url{http://dx.doi.org/10.1007/s10492-008-0035-8}.

  

\bibitem{MR2548486}
\newblock A.~Ito, N.~Kenmochi and N.~Yamazaki,
\newblock Weak solutions of grain boundary motion model with singularity,
\newblock \emph{Rend. Mat. Appl. (7)}, \textbf{29} (2009), 51--63.

  

\bibitem{MR2836555}
\newblock A.~Ito, N.~Kenmochi and N.~Yamazaki,
\newblock Global solvability of a model for grain boundary motion with
  constraint,
\newblock \emph{Discrete Contin. Dyn. Syst. Ser. S}, \textbf{5} (2012),
  127--146.

  

\bibitem{Kenmochi81}
\newblock N.~Kenmochi,
\newblock Solvability of nonlinear evolution equations with time-dependent
  constraints and applications,
\newblock \emph{Bull. Fac. Education, Chiba Univ.
  (\url{http://ci.nii.ac.jp/naid/110004715232})}, \textbf{30} (1981), 1--87.

  

\bibitem{MR2668289}
\newblock N.~Kenmochi and N.~Yamazaki,
\newblock Large-time behavior of solutions to a phase-field model of grain
  boundary motion with constraint,
\newblock in \emph{Current advances in nonlinear analysis and related topics},
  vol.~32 of GAKUTO Internat. Ser. Math. Sci. Appl.,
\newblock Gakk\=otosho, Tokyo, 2010,
\newblock 389--403.

  

\bibitem{MR1752970}
\newblock R.~Kobayashi, J.~A. Warren and W.~C. Carter,
\newblock A continuum model of grain boundaries,
\newblock \emph{Phys. D}, \textbf{140} (2000), 141--150,
\newblock \urlprefix\url{http://dx.doi.org/10.1016/S0167-2789(00)00023-3}.

  

\bibitem{MR1794359}
\newblock R.~Kobayashi, J.~A. Warren and W.~C. Carter,
\newblock Grain boundary model and singular diffusivity,
\newblock in \emph{Free boundary problems: theory and applications, {II}
  ({C}hiba, 1999)}, vol.~14 of GAKUTO Internat. Ser. Math. Sci. Appl.,
\newblock Gakk\=otosho, Tokyo, 2000,
\newblock 283--294.

  

\bibitem{lang1968analysisI}
\newblock S.~Lang,
\newblock \emph{Analysis I},
\newblock Addison-Wesley Publishing Company, 1968.

  

\bibitem{MR0298508}
\newblock U.~Mosco,
\newblock Convergence of convex sets and of solutions of variational
  inequalities,
\newblock \emph{Advances in Math.}, \textbf{3} (1969), 510--585,
\newblock \urlprefix\url{http://dx.doi.org/10.1016/0001-8708(69)90009-7}.

  

\bibitem{rudin1976principles}
\newblock W.~Rudin,
\newblock \emph{Principles of Mathematical Analysis},
\newblock International series in pure and applied mathematics, McGraw-Hill,
  1976,
\newblock \urlprefix\url{https://books.google.co.jp/books?id=kwqzPAAACAAJ}.

  

\bibitem{MR3038131}
\newblock K.~Shirakawa, H.~Watanabe and N.~Yamazaki,
\newblock Solvability of one-dimensional phase field systems associated with
  grain boundary motion,
\newblock \emph{Math. Ann.}, \textbf{356} (2013), 301--330,
\newblock \urlprefix\url{http://dx.doi.org/10.1007/s00208-012-0849-2}.

  

\bibitem{Oberwolfach2018}
\newblock K.~Shirakawa,
\newblock Energy-dissipations for gradient systems associated with anisotropic
  grain boundary motions,
\newblock \emph{Oberwolfach Reports}, \textbf{14} (2018), 324--326,
\newblock \urlprefix\url{https://ems.press/journals/owr/articles/15211}.

  

\bibitem{MR3082861}
\newblock K.~Shirakawa and H.~Watanabe,
\newblock Energy-dissipative solution to a one-dimensional phase field model of
  grain boundary motion,
\newblock \emph{Discrete Contin. Dyn. Syst. Ser. S}, \textbf{7} (2014),
  139--159,
\newblock \urlprefix\url{http://dx.doi.org/10.3934/dcdss.2014.7.139}.

  

\bibitem{MR0916688}
\newblock J.~Simon,
\newblock Compact sets in the space {$L^p(0,T;B)$},
\newblock \emph{Ann. Mat. Pura Appl. (4)}, \textbf{146} (1987), 65--96,
\newblock \urlprefix\url{http://dx.doi.org/10.1007/BF01762360}.

  

\bibitem{MR3203495}
\newblock H.~Watanabe and K.~Shirakawa,
\newblock Qualitative properties of a one-dimensional phase-field system
  associated with grain boundary,
\newblock in \emph{Nonlinear analysis in interdisciplinary
  sciences---modellings, theory and simulations}, vol.~36 of GAKUTO Internat.
  Ser. Math. Sci. Appl.,
\newblock Gakk\=otosho, Tokyo, 2013,
\newblock 301--328.

  

\bibitem{MR3155454}
\newblock N.~Yamazaki,
\newblock Global attractors for non-autonomous phase-field systems of grain
  boundary motion with constraint,
\newblock \emph{Adv. Math. Sci. Appl.}, \textbf{23} (2013), 267--296.

\end{thebibliography}
\end{document}